\documentclass{amsart}
\usepackage{amssymb,latexsym,amsmath,amscd,graphicx,graphics,epic,eepic,bm,array}
\usepackage[usenames,dvipsnames]{color}
\usepackage{enumerate}
\usepackage[all,knot,poly]{xy}

\newcommand{\zed}{\mathbb{Z}}

\newcommand{\C}{\mathbb{C}}
\newcommand{\fil}{\mathcal{F}}
\newcommand{\Hom}{\mathrm{Hom}}

\newcommand{\Sym}{\mathrm{Sym}}
\newcommand{\qb}[2]{\genfrac{[}{]}{0pt}{}{#1}{#2}}
\newcommand{\bn}[2]{\genfrac{(}{)}{0pt}{}{#1}{#2}}
\newcommand{\sfrac}[2]{\left.{#1}\middle/{#2}\right.}
\newcommand{\ve}{\varepsilon}

\newcommand{\id}{\mathrm{id}}

\newcommand{\hmf}{\mathrm{hmf}}
\newcommand{\HMF}{\mathrm{HMF}}

\newcommand{\hch}{\mathsf{hCh}^{\mathsf{b}}}

\newcommand{\slmf}{\mathfrak{sl}}

\theoremstyle{plain}
\newtheorem{theorem}{Theorem}[section]
\newtheorem{lemma}[theorem]{Lemma}
\newtheorem{proposition}[theorem]{Proposition}
\newtheorem{corollary}[theorem]{Corollary}

\newtheorem{question}[theorem]{Question}

\theoremstyle{definition}
\newtheorem{definition}[theorem]{Definition}

\newtheorem{acknowledgments}{Acknowledgments\ignorespaces}

\theoremstyle{remark}
\newtheorem{remark}[theorem]{Remark}

\numberwithin{equation}{section}

\begin{document}

\title{Generic deformations of the colored $\mathfrak{sl}(N)$-homology for links}

\author{Hao Wu}

\address{Department of Mathematics, The George Washington University, Monroe Hall, Room 240, 2115 G Street, NW, Washington DC 20052}

\email{haowu@gwu.edu}



\begin{abstract}
We generalize the works of Lee \cite{Lee2} and Gornik \cite{Gornik} to construct a basis for generic deformations of the colored $\mathfrak{sl}(N)$-homology defined in \cite{Wu-color-equi}. As applications, we construct non-degenerate pairings and co-pairings which lead to dualities of generic deformations of the colored $\mathfrak{sl}(N)$-homology. We also define and study colored $\mathfrak{sl}(N)$-Rasmussen invariants. Among other things, we observe that these invariants vanish on amphicheiral knots and discuss some implications of this observation.
\end{abstract}

\maketitle

\tableofcontents

\section{Introduction}\label{sec-intro}

\subsection{Background} In \cite{MOY}, Murakami, Ohtsuki and Yamada introduced the MOY calculus, which gives a combinatorial description of the quantum $\mathfrak{sl}(N)$-invariant for links in $S^3$ colored by positive integers (or, equivalently, wedge products of the defining representation of $\mathfrak{sl}(N;\C)$.) By categorifying special cases of the MOY calculus, Khovanov and Rozansky \cite{KR1} defined a $\zed^{\oplus2}$-graded $\mathfrak{sl}(N)$-homology for links that categorifies the (uncolored) $\mathfrak{sl}(N)$-HOMFLY-PT polynomial. When $N=2$, their construction recovers the Khovanov homology \cite{K1}. In \cite{Wu-color}, I generalized their construction to give an $\mathfrak{sl}(N)$-homology for links colored by positive integers. This construction is based on matrix factorizations associated to general MOY graphs with potentials induced by $X^{N+1}$. If a link $L$ is colored entirely by $1$, then its colored $\mathfrak{sl}(N)$-homology is its Khovanov-Rozansky $\mathfrak{sl}(N)$-homology. 

For any polynomial $P(X)$ in $\C[X]$ of the form 
\begin{equation}\label{def-P}
P(X)=X^{N+1} + \sum_{k=1}^N (-1)^{k}\frac{N+1}{N+1-k}b_{k} X^{N+1-k},
\end{equation}
where $b_1,\dots, b_N \in \C$, one can perform the above construction using matrix factorizations associated to general MOY graphs with potentials induced by $P(X)$ and give a deformation of the $\mathfrak{sl}(N)$-homology. In the uncolored case, this deformation was introduced by Lee \cite{Lee2} for the Khovanov homology and then by Gornik \cite{Gornik} for general $N\geq 2$. In \cite{Wu7}, I established the invariance of deformations in the uncolored case. Recently, I \cite{Wu-color-equi} defined a deformation $H_P$ of the colored $\mathfrak{sl}(N)$-homology and proved its invariance. If a link $L$ is colored entirely by $1$, then $H_P(L)$ is the uncolored deformation studied in \cite{Gornik,Lee2,Wu7}.

We say that a polynomial in $\C[X]$ is generic if its derivative has only simple roots. In particular, $P(X)$ in \eqref{def-P} is generic if and only if  
\[
P'(X)= (N+1)(X^N+\sum_{k=1}^N (-1)^{k}b_{k} X^{N-k})
\] 
has $N$ distinct roots. For a generic polynomial $P(X)$ of the form \eqref{def-P}, we call $H_P$ a generic deformation of the colored $\mathfrak{sl}(N)$-homology. Lee \cite{Lee2} and Gornik \cite{Gornik} constructed a basis for generic deformations of the Khovanov-Rozansky $\mathfrak{sl}(N)$-homology. This basis has been used to define the Rasmussen invariant \cite{Ras1} and the $\slmf(N)$-Rasmussen invariant \cite{Lobb,Wu7}, which, among other things, have led to new bounds of the slice genus \cite{Lobb,Ras1,Wu7} and a new proof of Milnor's conjecture \cite{Ras1}. 

The primary goal of the present paper is to generalize the Lee-Gornik construction to give a basis for generic deformations of the colored $\mathfrak{sl}(N)$-homology. We also give several applications of this basis.

\subsection{A basis for generic deformations of colored $\mathfrak{sl}(N)$-homology}
Throughout this paper, $N$ is a fixed integer greater than or equal to $2$. All links in this paper are oriented and colored. That is, every component of the link is assigned an orientation and an element of $\{0,1,\dots,N\}$, which we call the color of this component. A link that is completely color by $1$ is called uncolored.

\begin{definition}\label{def-states-homological-grading}
Let $L$ be a colored link diagram with components $K_1,\dots,K_l$. Assume the color of $K_i$ is $\mathfrak{c}(K_i) \in \{0,1,\dots,N\}$. 

Assume $P(X)$ is a generic polynomial of the form \eqref{def-P}. Denote by $\Sigma(P)$ the set of roots of $P'$, which is a set of $N$ distinct complex numbers, and by $\mathcal{P}(\Sigma(P))$ the set of subsets of $\Sigma(P)$.

A state $\psi$ of $L$ is a function $\psi:\{K_1,\dots,K_l\} \rightarrow \mathcal{P}(\Sigma(P))$ satisfying $|\psi(K_i)|=\mathfrak{c}(K_i)$. Denote by $\mathcal{S}_P(L)$ the set of all states of $L$.

Given a state $\psi\in \mathcal{S}_P(L)$, for each crossing $c$ in $L$, define a number $\mathsf{h}_\psi(c)$ by
\[
\mathsf{h}_\psi\left(\setlength{\unitlength}{1pt}
\begin{picture}(40,40)(-20,0)

\put(-20,-20){\vector(1,1){40}}

\put(20,-20){\line(-1,1){15}}

\put(-5,5){\vector(-1,1){15}}

\put(-11,15){\tiny{$_{K_i}$}}

\put(3,15){\tiny{$_{K_j}$}}

\end{picture}\right) = \begin{cases}
|\psi(K_i) \cap \psi(K_j)| & \text{if }  \mathfrak{c}(K_i) \neq \mathfrak{c}(K_j), \\
|\psi(K_i) \cap \psi(K_j)| - \mathfrak{c}(K_i) & \text{if }  \mathfrak{c}(K_i) = \mathfrak{c}(K_j),
\end{cases}
\]
\[
\mathsf{h}_\psi\left(\setlength{\unitlength}{1pt}
\begin{picture}(40,40)(-20,0)

\put(20,-20){\vector(-1,1){40}}

\put(-20,-20){\line(1,1){15}}

\put(5,5){\vector(1,1){15}}

\put(-11,15){\tiny{$_{K_i}$}}

\put(3,15){\tiny{$_{K_j}$}}

\end{picture}\right) = \begin{cases}
-|\psi(K_i) \cap \psi(K_j)| & \text{if }  \mathfrak{c}(K_i) \neq \mathfrak{c}(K_j), \\
\mathfrak{c}(K_i) - |\psi(K_i) \cap \psi(K_j)| & \text{if }  \mathfrak{c}(K_i) = \mathfrak{c}(K_j).
\end{cases}
\]
Define 
\[
\mathsf{h}(\psi) = \sum_{c} \mathsf{h}_\psi(c),
\]
where $c$ runs through all crossings of $L$. It is easy to check that $\mathsf{h}(\psi)$ is invariant under Reidemeister moves of $L$.
\end{definition}

The following theorem generalizes the Lee-Gornik construction \cite{Gornik,Lee2}.

\begin{theorem}\label{thm-basis}
Let $P$ be a generic polynomial of the form \eqref{def-P} and $L$ a colored link diagram. For each state $\psi$ of $L$, there exists, up to scaling, a non-zero element $v_\psi$ of $H_P(L)$ of homological grading $\mathsf{h}(\psi)$ such that 
\[
H_P(L) = \bigoplus_{\psi \in \mathcal{S}_P(L)} \C \cdot v_\psi.
\]
The above decomposition is invariant under the Reidemeister moves.
\end{theorem}

\begin{remark}
From \cite{Wu-color-equi}, we know that $H_P(L)$ has a $\zed_2$-grading, a quantum filtration and a homological grading. As in the uncolored case, the construction of the basis $\{v_\psi\}$ in Sections \ref{sec-basis-MOY-construction}, \ref{sec-morphisms} and \ref{sec-basis-link-construction} implies that the decomposition in Theorem \ref{thm-basis} does not preserve the quantum filtration. But, as we will see, it can be used to study the quantum filtration in some indirect ways. From \cite[Theorem 9.9]{Wu-color-equi}, we know that the $\zed_2$-grading of $H_P(L)$ is trivial. So we will not keep track of this $\zed_2$-grading.
\end{remark}

\subsection{Linear change of $\Sigma(P)$} It is not clear whether the generically deformed colored $\mathfrak{sl}(N)$-homology $H_P$ depends on the choice of the generic polynomial $P(X)$. But it is not hard to show that $H_P$ is invariant under linear changes of $\Sigma(P)$. The following generalizes \cite[Propositions 2.2 and 3.1]{Mackaay-Turner-Vaz}.

\begin{proposition}\label{thm-linear-change-roots}
Let $P(X)$ and $\hat{P}(X)$ be generic polynomials of the form \eqref{def-P}. If there are $a,b\in \C$ with $a\neq0$ such that $\Sigma(\hat{P}) = \{ar+b|r\in\Sigma(P)\}$, then $H_P(L) \cong H_{\hat{P}}(L)$ for any colored link $L$, where the isomorphism preserves both the quantum filtration and the homological grading.
\end{proposition}

As a byproduct of the proof of Proposition \ref{thm-linear-change-roots}, we have that $H_P(L)$ is invariant under reversal of orientation of all the components at the same time.

\begin{corollary}\label{coro-reverse-orientation}
Let $L$ be a colored link. Denote by $-L$ the colored link obtained from $L$ by reversing the orientation of every component of $L$. Suppose that $P(X)$ is a polynomial of the form \eqref{def-P}, not necessarily generic. Then $H_P(L) \cong H_P(-L)$, where the isomorphism preserves both the quantum filtration and the homological grading.
\end{corollary}

\subsection{Non-degenerate pairings and co-pairings} One application of Theorem \ref{thm-basis} is to establish dualities between generic deformations of the $\mathfrak{sl}(N)$-homology of certain colored links. We do this by constructing non-degenerate pairings and co-pairings. The motivation for these constructions is to understand how the colored $\mathfrak{sl}(N)$-Rasmussen invariants defined in Definition \ref{def-ras} below behave under certain changes of the colored knot. (See Corollary \ref{coro-Ras-equations} below.)

\begin{definition}\label{def-pairings-copairings}
Let $V$ and $W$ be finite dimensional linear spaces of the same dimension over $\C$. A pairing of $V$ and $W$ is a $\C$-linear map $\nabla:V \otimes_\C W \rightarrow \C$. We say that $\nabla$ is non-degenerate if it is non-degenerate as a bilinear form from $V\times W$ to $\C$. A co-pairing of $V$ and $W$ is a $\C$-linear map $\Delta: \C \rightarrow V \otimes_\C W$. We say that $\Delta$ is non-degenerate if its dual, $\Delta^\ast: V^\ast \otimes_\C W^\ast \rightarrow \C$, is non-degenerate as a bilinear form from $V^\ast \times W^\ast$ to $\C$.
\end{definition}

\begin{proposition}\label{thm-pairing-inverse}
Let $L$ be a colored link. Denote by $\overline{L}$ the colored link obtained from $L$ by reversing the orientation of every component of $L$ and switching the upper- and lower-branches at each crossing. Suppose that $P(X)$ is a generic polynomial of the form \eqref{def-P}. Then there are a non-degenerate pairing $\overline{\nabla}:H_P(L)\otimes_\C H_P(\overline{L}) \rightarrow \C$ and a non-degenerate co-pairing $\overline{\Delta}:\C \rightarrow H_P(L)\otimes_\C H_P(\overline{L})$, both of which preserve the quantum filtration and the homological grading.

In particular, $H_P(\overline{L}) \cong \Hom_\C(H_P(L),\C)$, where the isomorphism preserves both the quantum filtration and the homological grading.
\end{proposition}

To state the next duality theorem, we need to introduce a grading shift.

\begin{lemma}\label{lemma-grading-shift-pairing-op}
For each colored crossing, define
\[
\mathsf{s}\left(\setlength{\unitlength}{1pt}
\begin{picture}(40,40)(-20,0)

\put(-20,-20){\vector(1,1){40}}

\put(20,-20){\line(-1,1){15}}

\put(-5,5){\vector(-1,1){15}}

\put(-11,15){\tiny{$_m$}}

\put(9,15){\tiny{$_n$}}

\end{picture}\right) = \begin{cases}
0 & \text{if } m=n \text{ or } m=N-n, \\
N-2n & \text{otherwise},
\end{cases}
\]
\[
\mathsf{s}\left(\setlength{\unitlength}{1pt}
\begin{picture}(40,40)(-20,0)

\put(20,-20){\vector(-1,1){40}}

\put(-20,-20){\line(1,1){15}}

\put(5,5){\vector(1,1){15}}

\put(-11,15){\tiny{$_m$}}

\put(9,15){\tiny{$_n$}}

\end{picture}\right) = \begin{cases}
0 & \text{if } m=n \text{ or } m=N-n, \\
2m-N & \text{otherwise},
\end{cases}
\]
\[
\mathsf{s}'\left(\setlength{\unitlength}{1pt}
\begin{picture}(40,40)(-20,0)

\put(-20,-20){\vector(1,1){40}}

\put(20,-20){\line(-1,1){15}}

\put(-5,5){\vector(-1,1){15}}

\put(-11,15){\tiny{$_m$}}

\put(9,15){\tiny{$_n$}}

\end{picture}\right) = \begin{cases}
0 & \text{if } m=n \text{ or } m=N-n, \\
N-2m & \text{otherwise},
\end{cases}
\]
\[
\mathsf{s}'\left(\setlength{\unitlength}{1pt}
\begin{picture}(40,40)(-20,0)

\put(20,-20){\vector(-1,1){40}}

\put(-20,-20){\line(1,1){15}}

\put(5,5){\vector(1,1){15}}

\put(-11,15){\tiny{$_m$}}

\put(9,15){\tiny{$_n$}}

\end{picture}\right) = \begin{cases}
0 & \text{if } m=n \text{ or } m=N-n, \\
2n-N & \text{otherwise}.
\end{cases}
\]
For a colored link diagram $D$, define 
\begin{eqnarray*}
\mathsf{s}(D) = \sum_c \mathsf{s}(c), \\
\mathsf{s}'(D) = \sum_c \mathsf{s}'(c),
\end{eqnarray*}
where $c$ runs through all crossings of $D$. Then $\mathsf{s}(D)$ and $\mathsf{s}'(D)$ are both invariant under the Reidemeister moves, and $\mathsf{s}(D)=\mathsf{s}'(D)$ for any colored link diagram $D$.
\end{lemma}

\begin{proof}
It is easy to check that $\mathsf{s}(D)$ and $\mathsf{s}'(D)$ are invariant under the Reidemeister moves. To prove $\mathsf{s}(D)=\mathsf{s}'(D)$, consider the colored link diagram $D'$ obtained by ``looking at $D$ from the backside". $D$ and $D'$ represent the same colored link. So they are connected by a finite sequence of Reidemeister moves. Therefore, $\mathsf{s}(D)=\mathsf{s}(D')$. But, by their definitions, it clear that $\mathsf{s}(D')=\mathsf{s}'(D)$. Thus, $\mathsf{s}(D)=\mathsf{s}'(D)$.
\end{proof}

\begin{proposition}\label{thm-pairing-op}
Let $L$ be a colored link. Denote by $L^{op}$ the colored link obtained from $L$ by switching the upper- and lower-branches at each crossing and changing each color $k$ to $N-k$. Suppose that $P(X)$ is a generic polynomial of the form \eqref{def-P}. Then there is a non-degenerate pairing $\nabla^{op}:H_P(L)\otimes_\C H_P(L^{op}) \rightarrow \C$ of homological grading $\mathsf{s}(L)$ and quantum degree not greater than $-\mathsf{s}(L)$. There is also a non-degenerate co-pairing $\Delta^{op}:\C \rightarrow H_P(L)\otimes_\C H_P(L^{op})$ of homological grading $-\mathsf{s}(L)$ and quantum degree not greater than $\mathsf{s}(L)$.

In particular, $H_P(L^{op}) \cong \Hom_\C(H_P(L),\C) \|-\mathsf{s}(L)\| \{q^{\mathsf{s}(L)}\}$, where, as in \cite{Wu-color,Wu-color-equi}, $\|\star\|$ means shifting the homological grading by $\star$ and $\{q^{\bullet}\}$ means shifting the quantum filtration by $\bullet$. The isomorphism here preserves both the quantum filtration and the homological grading.
\end{proposition}

Note that $(\overline{L})^{op} = \overline{L^{op}}$ is the colored link obtained from $L$ by reversing the orientation of every component and changing every color $k$ to $N-k$. Comparing Propositions \ref{thm-pairing-inverse} and \ref{thm-pairing-op}, we have the following corollary.

\begin{corollary}\label{coro-color-orientation}
Let $L$ be a colored link. Suppose that $P(X)$ is a generic polynomial of the form \eqref{def-P}. Then $H_P(L) \cong H_P(\overline{L^{op}})\|-\mathsf{s}(L)\| \{q^{\mathsf{s}(L)}\}$.
\end{corollary}

\begin{proof}
By Propositions \ref{thm-pairing-inverse} and \ref{thm-pairing-op}, we have  
\begin{eqnarray*}
H_P(L) & \cong & \Hom_\C(H_P(L^{op}),\C) \|-\mathsf{s}(L^{op})\| \{q^{\mathsf{s}(L^{op})}\} \\ 
& \cong & H_P(\overline{L^{op}})\|-\mathsf{s}(L^{op})\| \{q^{\mathsf{s}(L^{op})}\}.
\end{eqnarray*}
It is easy to check that $\mathsf{s}(L^{op}) = \mathsf{s}'(L)$. But, by Lemma \ref{lemma-grading-shift-pairing-op}, $\mathsf{s}'(L)=\mathsf{s}(L)$. So $H_P(L) \cong H_P(\overline{L^{op}})\|-\mathsf{s}(L)\| \{q^{\mathsf{s}(L)}\}$.
\end{proof}

\begin{remark}
The constructions of the pairings and co-pairings in Propositions \ref{thm-pairing-inverse} and \ref{thm-pairing-op} are done component by component. So one can have ``mixed" non-degenerate pairings and co-pairings by applying Proposition \ref{thm-pairing-inverse} to some components and Proposition \ref{thm-pairing-op} to others. In fact, Corollary \ref{coro-color-orientation} can be made more general. Roughly speaking, we have:

Let $P$ be a polynomial of the form \eqref{def-P}, not necessarily generic. If $L$ is a colored link and $K$ is a component of $L$ with color $k$. Denote by $L'$ the link obtained from $L$ by reversing the orientation of $K$ and changing its color to $N-k$. Then $H_P(L')$ and $H_P(L)$ differ only by overall shiftings of the quantum filtration and the homological grading. (I will prove this in a forthcoming note \cite{Wu-color-orientation}.)
\end{remark}

\begin{question}
Are Propositions \ref{thm-pairing-inverse} and \ref{thm-pairing-op} true for non-generic $P(X)$?
\end{question}

The pairings and co-pairings in Propositions \ref{thm-pairing-inverse} and \ref{thm-pairing-op} are defined for any $P(X)$ of the form \eqref{def-P} (not necessarily generic.) But it is not clear how to establish the non-degeneracy without using the basis in Theorem \ref{thm-basis}.

\subsection{Colored $\mathfrak{sl}(N)$-Rasmussen invariants} Denote by $\fil$ the quantum filtration of $H_P$. For a colored link $L$, define 
\begin{eqnarray*}
\max\deg_q H_P(L) & = & \min\{j|\fil^j H_P(L) = H_P(L)\}, \\
\min\deg_q H_P(L) & = & \max\{j|\fil^{j-1} H_P(L) = 0\}.
\end{eqnarray*}

\begin{definition}\label{def-ras}
Let $P$ be a generic polynomial of the form \eqref{def-P}, and $m \in \{1.\dots,N-1\}$. For a knot $K$, coloring it by $m$ gives a colored knot $K^{(m)}$. Define 
\[
s_P^{(m)}(K) = \frac{1}{2} (\max\deg_q H_P(K^{(m)}) + \min\deg_q H_P(K^{(m)})). 
\]
We call $s_P^{(m)}(K)$ the $m$-colored $\mathfrak{sl}(N)$-Rasmussen invariant of $K$.
\end{definition}

It is clear that, when $m=1$ and $P(X)=X^{N+1}-(N+1)X$, $s_P^{(1)}(K)$ is the $\mathfrak{sl}(N)$-Rasmussen invariant defined in \cite[Definition 1.5]{Wu7}. In particular, when $N=2$, $s_P^{(1)}(K)$ is the original Rasmussen invariant \cite{Ras1}.

Recently, Lobb \cite{Lobb2} proved that, for $P(X)=X^{N+1} - (N+1)X$, $s_P^{(1)}(K)$ is a concordance invariant. It seems that his proof can be generalized to the colored $\mathfrak{sl}(N)$-Rasmussen invariants $s_P^{(m)}(K)$ for this special $P(X)$. But it is still not clear whether $s_P^{(m)}(K)$ is a concordance invariant for a general generic $P(X)$.

From the results in the previous subsections, we have the following corollaries.

\begin{corollary}\label{coro-Ras-linear-root-change}
Let $P(X)$ and $\hat{P}(X)$ be generic polynomials of the form \eqref{def-P}. If there are $a,b\in \C$ with $a\neq0$ such that $\Sigma(\hat{P}) = \{ar+b|r\in\Sigma(P)\}$, then $s_P^{(m)}(K) = s_{\hat{P}}^{(m)}(K)$ for any knot $K$.
\end{corollary}

\begin{proof}
This follows from Proposition \ref{thm-linear-change-roots}.
\end{proof}

For a knot $K$, $-K$ is $K$ with reversed orientation, $K_{mir}$ is the mirror image of $K$ and $\overline{K}=-K_{mir}$. 

\begin{corollary}\label{coro-Ras-equations}
Let $P$ be a generic polynomial of the form \eqref{def-P}, and $m \in \{1.\dots,N-1\}$. Then, for any knot $K$,
\[
s_P^{(m)}(K) = s_P^{(N-m)}(K)= s_P^{(m)}(-K) = -s_P^{(m)}(\overline{K}) = -s_P^{(m)}(K_{mir}).
\]
\end{corollary}

\begin{proof}
By Propositions \ref{thm-pairing-inverse} and \ref{thm-pairing-op}, we have $s_P^{(m)}(K) = -s_P^{(m)}(\overline{K}) = -s_P^{(N-m)}(K_{mir})$. Plugging $\overline{K}$ into this equation, we get $s_P^{(m)}(K) = s_P^{(N-m)}(-K)$. By Corollary \ref{coro-reverse-orientation}, we have $s_P^{(N-m)}(K) = s_P^{(N-m)}(-K)$. So $s_P^{(m)}(K) = s_P^{(N-m)}(K)$. Finally, plugging $K_{mir}$ into this equation, we get $s_P^{(m)}(K_{mir}) = s_P^{(N-m)}(K_{mir})$.
\end{proof}

Recall that $K$ is positively amphicheiral if $K=K_{mir}$ and $K$ is negatively amphicheiral if $K=\overline{K}$. $K$ is called chiral if it is neither positively amphicheiral nor negatively amphicheiral.

\begin{corollary}\label{coro-amphicheiral}
If $K$ is (positively or negatively) amphicheiral, then $s_P^{(m)}(K)=0$ for any generic polynomial $P$ of the form \eqref{def-P} and any $m \in \{1,\dots,N-1\}$. 
\end{corollary}

\begin{proof}
This follows from Corollary \ref{coro-Ras-equations}. 
\end{proof}

The colored $\mathfrak{sl}(N)$-Rasmussen invariants also provide new bounds for the slice genus and self linking number. The following is yet another generalization of Rasmussen's genus bound \cite[Theorem 1]{Ras1}.

\begin{theorem}\label{thm-ras-genus}
Let $P$ be a generic polynomial of the form \eqref{def-P}, and $m \in \{1,\dots,N-1\}$. Then, for any knot $K$, $|s_P^{(m)}(K)| \leq 2 m(N-m) g_\ast (K)$, where $g_\ast (K)$ is the smooth slice genus of $K$.
\end{theorem}

Recall that, for a closed braid $B$ of $n$ strands with writhe $w$, the self linking number of $B$ is defined to be $SL(B)=w-n$. This is a classical invariant for transversal links in the standard contact $S^3$. (See for example \cite{Ben}.) Given a link $L$, the maximal self linking number of $L$ is defined to be 
\[
\overline{SL}(L)=\max\{SL(B)~|~ B \text{ is a closed braid representing } L\}. 
\]
The colored $\mathfrak{sl}(N)$-Rasmussen invariants provide new bounds for the maximal self linking number, which generalize \cite[Proposition 4]{Pl4} and \cite[Lemma 1.C]{Shumakovitch}.

\begin{theorem}\label{thm-ras-bennequin}
Let $P$ be a generic polynomial of the form \eqref{def-P}, and $m \in \{1,\dots,N-1\}$. Then, for any knot $K$, $s_P^{(m)}(K) \geq  m(N-m) (\overline{SL}(K)+1)$.
\end{theorem}

\begin{corollary}\label{coro-amphicheiral-topo}
\begin{enumerate}
	\item A knot $K$ is chiral if $\overline{SL}(K) \geq 0$.
	\item Quasipositive amphicheiral knots are smoothly slice.
\end{enumerate}
\end{corollary}

\begin{proof}
Part (1) follows from Corollary \ref{coro-amphicheiral} and Theorem \ref{thm-ras-bennequin}. To prove Part (2), assume $K$ is a quasipositive amphicheiral knot. By Theorems \ref{thm-ras-genus} and \ref{thm-ras-bennequin}, we have $\overline{SL}(K)+1 \leq \frac{s_P^{(m)}(K)}{m(N-m)} \leq 2 g_\ast (K)$, where $P$ is any generic polynomial of the form \eqref{def-P}. Since $K$ is quasipositive, by \cite[Corollary in Section 3]{Ru}, we know that $\overline{SL}(K)+1 = 2 g_\ast (K)$. Since $K$ is amphicheiral, by Corollary \ref{coro-amphicheiral}, we have $s_P^{(m)}(K)=0$. Putting everything together, we get that $g_\ast (K)=0$.
\end{proof}

\begin{remark}
The Seifert genus of a strongly quasipositive knot is equal to its smooth slice genus. So Part (2) of Corollary \ref{coro-amphicheiral-topo} implies that the unknot is the only strongly quasipositive  amphicheiral knot. This and similar results have been proved by a variety of methods. See for example \cite{Rud} for a brief historical review of these results. 
\end{remark}

\subsection{Organization of the paper} We assume the reader is somewhat familiar with the techniques used in \cite{Wu-color, Wu-color-equi}. Most basic notations and definitions in this paper are taken from \cite{Wu-color, Wu-color-equi}. For the convenience of the reader, we summarize in Section \ref{sec-notations} the notations and conventions used in the present paper. 

In Section \ref{sec-changing-P}, we study linear changes of $P(X)$ and prove Proposition \ref{thm-linear-change-roots} and Corollary \ref{coro-reverse-orientation}. Then, in Sections \ref{sec-basis-MOY-construction}-\ref{sec-basis-link-construction}, we generalize Gornik's method to construct a basis for each generic deformation of the colored $\mathfrak{sl}(N)$-homology and prove Theorem \ref{thm-basis}. After that, we establish non-degenerate pairings and co-pairings and prove Propositions \ref{thm-pairing-inverse} and \ref{thm-pairing-op} in Section \ref{sec-pairings}. Finally, we establish the bounds for the slice genus and the self linking number in Section \ref{sec-bounds}.

\begin{acknowledgments}
I would like to thank the referee for many helpful suggestions.
\end{acknowledgments}

\section{Notations and Conventions}\label{sec-notations} 

\subsection{Basics} Throughout this paper, $N$ is a fixed integer greater than or equal to $2$.

All links in this paper are oriented and colored. That is, every component of the link is assigned an orientation and an element of $\{0,1,\dots,N\}$, which we call the color of this component. A link that is completely colored by $1$ is called uncolored.

Unless otherwise specified, $P$ is a generic polynomial of the form \eqref{def-P}. Denote by $\Sigma(P)$ the set of roots of $P'$, which is a set of $N$ distinct complex numbers, and by $\mathcal{P}(\Sigma(P))$ the set of subsets of $\Sigma(P)$. If $P$ is clear from the context, we drop it from the notations and just write $\Sigma$ and $\mathcal{P}(\Sigma)$.

\subsection{Matrix factorizations}\label{subsec-sum-mf}
Filtered matrix factorizations are defined in \cite[Definition 9.1]{Wu-color-equi}. A morphism of filtered matrix factorizations over a graded ring $R$ is a homomorphism of the underlying $R$-module that preserves the filtration and commutes with the differentials of the matrix factorizations. 

For a matrix factorization $M$ over $R$ and $r\in R$, $\mathfrak{m}(r):M\rightarrow M$ is the endomorphism of $M$ given by the multiplication of $r$.

We use two-column matrices of the form
\[
\left(%
\begin{array}{cc}
  a_1 & b_1 \\
  a_2 & b_2 \\
  \vdots & \vdots \\
  a_k & b_k \\
\end{array}%
\right)_R
\]
to represent Koszul matrix factorizations. (See \cite[Definition 2.3]{Wu-color}.) 

\subsection{The chain complex associated to a knotted MOY graph}
Knotted MOY graphs are defined in \cite[Definitions 5.1 and 11.1]{Wu-color}. Note that colored link diagrams and embedded MOY graphs \cite[Definition 5.1]{Wu-color} are knotted MOY graphs. In the present paper, edges of knotted MOY graphs are colored by elements of $\{0,1,\dots,N\}$. 

To define and compute the homology of a knotted MOY graph, one needs to put a marking on the knotted MOY graph. See \cite[Definitions 5.3 and 11.2]{Wu-color} for the definition of markings.

Let $D$ be a knotted MOY graph. We denote by $C_P(D)$ the chain complex given in \cite[Theorem 1.2]{Wu-color-equi}. That is, in the notations of \cite{Wu-color-equi}, $P=\pi(f)$ and $C_P(D)=C_{f,\pi}(D)=\varpi(C_f(D))$. $H_P(D)$ is the homology of $C_P(D)$ as defined in \cite[Subsection 1.2]{Wu-color-equi}. $C_P(D)$ has a homological grading and a quantum filtration, both of which are inherited by $H_P(D)$. We denote by $\|k\|$ the shifting of the homological grading up by $k$ (as given in \cite[Definition 2.33]{Wu-color}) and by $\{F(q)\}$ the shifting of quantum filtration up by $F(q)$ (as given in \cite[Subsection 2.1]{Wu-color}.)

Note that, in \cite{Wu-color-equi}, I mostly dealt with the equivariant link homology $H_f$. In the present paper, whenever I cite a result from \cite{Wu-color-equi} about $H_f$ and apply it to the deformed link homology $H_P$, it should be understood that the functor $\varpi$ is applied to that result to change it into a result about $H_P$.

\subsection{Symmetric polynomials}\label{subsec-sum-sym-poly}
An alphabet is a set of finitely many homogeneous indeterminates of degree $2$. (Note that, the degree of a polynomial in this paper is twice its usual degree.) We denote alphabets by $\mathbb{A},~\mathbb{B}, \dots,~\mathbb{X},~\mathbb{Y}$. Of course, we avoid using letters $\mathbb{C},~\mathbb{N},~\mathbb{Q},~\mathbb{R},~\mathbb{Z}$ to represent alphabets. For an alphabet $\mathbb{X}=\{x_1,\dots,x_m\}$, denote by $\Sym(\mathbb{X})$ the graded ring of symmetric polynomials over $\C$ in $\mathbb{X}=\{x_1,\dots,x_m\}$. (Again, note that our grading of $\Sym(\mathbb{X})$ is twice the usual grading.) We use following notations for the elementary symmetric polynomials in $\mathbb{X}$:
\begin{eqnarray*}
X_k & = & \begin{cases}
\sum_{1\leq i_1<i_2<\cdots<i_k\leq m} x_{i_1}x_{i_1}\cdots x_{i_k} & \text{if } 1\leq k\leq m, \\
1 & \text{if } k=0, \\
0 & \text{if } k<0 \text{ or } k>m.
\end{cases} \\
p_k(\mathbb{X}) & = & 
\left\{%
\begin{array}{ll}
    \sum_{i=1}^{m} x_i^k & \text{if } k\geq0, \\
    0 & \text{if } k<0, 
\end{array}%
\right. \\
h_k(\mathbb{X}) & = & 
\left\{%
\begin{array}{ll}
    \sum_{1\leq i_1\leq i_2 \leq \cdots \leq i_k\leq m} x_{i_1}x_{i_1}\cdots x_{i_k} & \text{if } k>0, \\
    1 & \text{if } k=0, \\
    0 & \text{if } k<0. 
\end{array}%
\right.
\end{eqnarray*}
$X_k$, $p_k(\mathbb{X})$ and $h_k(\mathbb{X})$ are homogeneous elements of $\Sym(\mathbb{X})$ is degree $2k$.

For a polynomial $P(X)$ of the form \eqref{def-P}, we define $$P(\mathbb{X}):=\sum_{i=1}^m P(x_i)= p_{N+1}(\mathbb{X})+ \sum_{k=1}^N (-1)^{k}\frac{N+1}{N+1-k}b_{k} p_{N+1-k}(\mathbb{X}).$$

A partition $\lambda=(\lambda_1\geq\dots\geq\lambda_m)$ is a finite non-increasing sequence of non-negative integers. We denote by $\Lambda_{m,n}$ the set of partitions $$\Lambda_{m,n}=\{(\lambda_1\geq\dots\geq\lambda_m)~|~\lambda_1\leq n\}.$$

For an alphabet of $m$ indeterminates and a partition $\lambda=(\lambda_1\geq\dots\geq\lambda_m)$, we denote by $S_\lambda(\pm\mathbb{X})$ the Schur polynomial in $\pm\mathbb{X}$ associated to the partition $\lambda$, which is a generalization of $X_k$ and $h_k(\mathbb{X})$. (See \cite[Subsection 4.2]{Wu-color}.)

Given a collection $\{\mathbb{X}_1,\dots,\mathbb{X}_l\}$ of pairwise disjoint alphabets, we denote by $\Sym(\mathbb{X}_1|\cdots|\mathbb{X}_l)$ the ring of polynomials in $\mathbb{X}_1\cup\cdots\cup\mathbb{X}_l$ over $\C$ that are symmetric in each $\mathbb{X}_i$, which is a graded-free $\Sym(\mathbb{X}_1\cup\cdots\cup\mathbb{X}_l)$-module. Clearly,
\[
\Sym(\mathbb{X}_1|\cdots|\mathbb{X}_l) \cong \Sym(\mathbb{X}_1) \otimes_{\C} \cdots \otimes_{\C} \Sym(\mathbb{X}_l).
\]
$S_\lambda(\pm\mathbb{X}_1\pm\cdots\pm\mathbb{X}_l)$ is the Schur polynomial in $\pm\mathbb{X}_1\pm\cdots\pm\mathbb{X}_l$ associated to the partition $\lambda$ as defined in \cite{Lascoux-notes}. Let $\Omega_1,\dots, \Omega_l$ be sets of complex numbers such that $|\Omega_i| = |\mathbb{X}_i|$, $i=1,\dots,l$. Then $S_\lambda(\pm\Omega_1\pm\cdots\pm\Omega_l)$ is defined to be the valuation $S_\lambda(\pm\mathbb{X}_1\pm\cdots\pm\mathbb{X}_l)|_{\mathbb{X}_i=\Omega_i, ~i=1,\dots,l}$. Sometimes, we only valuate some of the alphabets, which gives a polynomial that is symmetric in each of the remaining alphabets.

\section{Changing $P(X)$}\label{sec-changing-P}

In this section, we study the effects of certain simple changes of the potential on matrix factorizations and link homology. The goal is to prove Proposition \ref{thm-linear-change-roots} and Corollary \ref{coro-reverse-orientation}.

\subsection{Scaling the potential of a matrix factorization} Let $R=\C[X_1,\dots,X_m]$, where $X_1,\dots,X_m$ are homogeneous indeterminates of positive degrees, and $\mathfrak{I}=(X_1,\dots,X_m)$ the maximal homogeneous ideal of $R$. Assume that $w$ is an element of $\mathfrak{I}$ with $\deg{w}\leq2N+2$. (See our degree convention in Subsection \ref{subsec-sum-sym-poly}.)

\begin{lemma}\label{lemma-scaling-potential}
For any $a\in \C\setminus\{0\}$, there is a fully faithful functor $F_a:\HMF_{R,w}\rightarrow \HMF_{R,a^2w}$ that preserves the $\zed_2$-grading and the quantum filtration on $\Hom_\HMF$. In particular, $F_a$ restricts to a fully faithful functor $F_a:\hmf_{R,w}\rightarrow \hmf_{R,a^2w}$.

For every object $M$ of $\HMF_{R,w}$, there is an isomorphism 
\[
H(M;R)\xrightarrow{g_{M,a}}H(F_a(M);R)
\] 
preserving both the $\zed_2$-grading and the quantum filtration. $g_{M,a}$ is natural in the sense that, for every morphisms $f:M\rightarrow M'$ of $\HMF_{R,w}$ preserving the $\zed_2$-grading, the square
\[
\xymatrix{
H(M;R) \ar[rr]^{g_{M,a}} \ar[d]^{f} && H(F_a(M);R) \ar[d]^{F_a(f)} \\
H(M';R) \ar[rr]^{g_{M',a}} && H(F_a(M');R) \\
}
\]
commutes.

Moreover, $F_a$ induces a functor $F_a:\hch(\hmf_{R,w})\rightarrow \hch(\hmf_{R,a^2w})$. We have
\begin{equation}\label{equ-iso-homology-chain-scaling}
H(O;R) \cong H(F_a(O);R)
\end{equation} 
for every object $O$ of $\hch(\hmf_{R,w})$.
\end{lemma}

\begin{proof}
Let $M$ be an object of $\HMF_{R,w}$ given by 
\[
M_0 \xrightarrow{d_0} M_1 \xrightarrow{d_1} M_0.
\]
Define $F_a(M)$ to be the object of $\HMF_{R,a^2w}$ given by
\[
M_0 \xrightarrow{a\cdot d_0} M_1 \xrightarrow{a\cdot d_1} M_0.
\]
Let $M,M'$ be objects of $\HMF_{R,w}$ and $f:M\rightarrow M'$ an $R$-module homomorphism. It is clear that $f$ is a morphism from $M$ to $M'$ if and only if it is a morphism from $F_a(M)$ to $F_a(M')$. Moreover, if $f$ is a morphism from $M$ to $M'$, then it is homotopic to $0$ as a morphism from $M$ to $M'$ if and only if it is homotopic to $0$ as a morphism from $F_a(M)$ to $F_a(M')$. So the morphisms of $\HMF_{R,w}$ and $\HMF_{R,a^2w}$ are identical. Now for every morphism $f$ in $\HMF_{R,w}$, define $F_a(f)=f$, where the right hand side is considered a morphism in $\HMF_{R,a^2w}$. This makes $F_a$ a fully faithful functor. It is easy to see that $F_a$ preserves the the $\zed_2$-grading and the quantum filtration on $\Hom_\HMF$. 

Let $M$ be as above. Recall that $H(M;R)$ is the homology of 
\[
M/\mathfrak{I}\cdot M = `` M_0/\mathfrak{I}\cdot M_0 \xrightarrow{d_0} M_1/\mathfrak{I}\cdot M_1 \xrightarrow{d_0} M_0/\mathfrak{I}\cdot M_0".
\]
Now, instead of considering $M/\mathfrak{I}\cdot M$ a $\zed_2$-graded chain complex, we consider it a $\zed$-graded chain complex of period $2$. That is, we consider the chain complex $(C_\ast,d_\ast)$ given by $C_{2n}= M_0/\mathfrak{I}\cdot M_0$, $C_{2n+1}= M_1/\mathfrak{I}\cdot M_1$, $d_{2n}=d_0$ and $d_{2n+1}=d_1$. Then $H(M;R) = H_0(C_\ast) \oplus H_1(C_\ast)$. 

If we do that same to $F_a(M)/\mathfrak{I}\cdot F_a(M)$, then we get the chain complex $(C_\ast,a\cdot d_\ast)$. Define a chain map $g_{M,a}:(C_\ast,d_\ast)\rightarrow (C_\ast,a\cdot d_\ast)$ by $g(x)=a^n x ~\forall ~x\in C_n$. This is a chain isomorphism preserving both the $\zed$-grading and the quantum filtration and, therefore, induces an isomorphism $g_{M,a}:H(M;R)\rightarrow H(F_a(M);R)$ that preserves the $\zed_2$-grading and the quantum filtration. 

Any morphism $f:M\rightarrow M'$ preserving the $\zed_2$-grading induces a chain map $f:(C_\ast,d_\ast) \rightarrow (C'_\ast,d'_\ast)$ preserving the $\zed$-grading, where $(C'_\ast,d'_\ast)$ is the chain complex obtained from $M'/\mathfrak{I}\cdot M'$. It is easy to see that the square of chain maps
\[
\xymatrix{
(C_\ast,d_\ast) \ar[rr]^{g_{M,a}} \ar[d]^{f} && (C_\ast,a\cdot d_\ast) \ar[d]^{f=F_a(f)} \\
(C'_\ast,d'_\ast) \ar[rr]^{g_{M',a}} && (C'_\ast,a\cdot d'_\ast) \\
}
\]
commutes. This induces the commutative square
\[
\xymatrix{
H(M;R) \ar[rr]^{g_{M,a}} \ar[d]^{f} && H(F_a(M);R) \ar[d]^{F_a(f)} \\
H(M';R) \ar[rr]^{g_{M',a}} && H(F_a(M');R). \\
}
\]
And the naturality of $g_{M,a}$ follows.

Finally, the isomorphism \eqref{equ-iso-homology-chain-scaling} follows easily from the naturality of $g_{M,a}$.
\end{proof}

\begin{corollary}\label{coro-scaling-P}
Suppose that $a\in\C\setminus\{0\}$ and $P(X)$ is a polynomial of the form \eqref{def-P}, not necessarily generic. Then $H_P(D) \cong H_{a^2 \cdot P}(D)$ for any knotted MOY graph $D$, where the isomorphism preserves the quantum filtration and the homological grading.
\end{corollary}

\begin{proof}

Recall that, for each MOY graph $\Gamma$, $C_P(\Gamma)$ is a Koszul matrix factorization of the form
\[
\left(%
\begin{array}{cc}
  f_{1,0}, & f_{1,1} \\
  f_{2,0}, & f_{2,1} \\
  \dots & \dots \\
  f_{k,0}, & f_{k,1}
\end{array}%
\right)_R,
\]
where the base ring $R$ is the tensor product over $\C$ of the symmetric polynomial rings of all the alphabets marking $\Gamma$, and $f_{j,\ve} \in R$. By definition, it is easy to see that $C_{a^2\cdot P}(\Gamma)$ is the Koszul matrix factorization
\[
\left(%
\begin{array}{cc}
  a^2\cdot f_{1,0}, & f_{1,1} \\
  a^2\cdot f_{2,0}, & f_{2,1} \\
  \dots & \dots \\
  a^2\cdot f_{k,0}, & f_{k,1}
\end{array}%
\right)_R.
\]
Using the morphisms $\psi$ and $\psi'$ defined in \cite[Subsection 2.2]{KR2}, or more generally, in \cite[Proposition 7.12]{Wu-color}, we have that
\[
C_{a^2\cdot P}(\Gamma)=
\left(%
\begin{array}{cc}
  a^2\cdot f_{1,0}, & f_{1,1} \\
  a^2\cdot f_{2,0}, & f_{2,1} \\
  \dots & \dots \\
  a^2\cdot f_{k,0}, & f_{k,1}
\end{array}%
\right)_R \cong
\left(%
\begin{array}{cc}
  a\cdot f_{1,0}, & a\cdot f_{1,1} \\
  a\cdot f_{2,0}, & a\cdot f_{2,1} \\
  \dots & \dots \\
  a\cdot f_{k,0}, & a\cdot f_{k,1}
\end{array}%
\right)_R = F_a(C_{P}(\Gamma)).
\]
Therefore, for any knotted MOY graph $D$, the chain complexes $C_{a^2\cdot P}(D)$ and $F_a(C_{P}(D))$ have isomorphic terms. By the definition of the differential map of $C_{P}(D)$, especially \cite[Lemma 11.11]{Wu-color}, it follows that $C_{a^2\cdot P}(D) \cong F_a(C_{P}(D))$ as objects of $\hch(\hmf)$. Then, by isomorphism \eqref{equ-iso-homology-chain-scaling}, we have 
\[
H_{a^2\cdot P}(D) \cong H(F_a(C_{P}(D));R) \cong H_{P}(D).
\]
\end{proof}

\subsection{Orientation reversal} Now we are ready to prove Corollary \ref{coro-reverse-orientation}. In fact, we prove the following slightly more general corollary.

\begin{corollary}\label{coro-MOY-reverse-orientation}
Suppose that $P(X)$ is a polynomial of the form \eqref{def-P}, not necessarily generic. For a knotted MOY graph $D$, denote by $-D$ the knotted MOY graph obtained by reversing the orientation of every edge of $D$. Then $H_P(D) \cong H_{P}(-D)$, where the isomorphism preserves the quantum filtration and the homological grading.
\end{corollary}

\begin{proof}
Let $\Gamma$ be a complete resolution of $D$. Then $-\Gamma$ is a complete resolution of $-D$. Recall that $C_P(\Gamma)$ is a Koszul matrix factorization of the form
\[
\left(%
\begin{array}{cc}
  f_{1,0}, & f_{1,1} \\
  f_{2,0}, & f_{2,1} \\
  \dots & \dots \\
  f_{k,0}, & f_{k,1}
\end{array}%
\right)_R.
\]
By definition, it is easy to see that $C_P(-\Gamma)$ is the Koszul matrix factorization
\[
\left(%
\begin{array}{cc}
  f_{1,0}, & -f_{1,1} \\
  f_{2,0}, & -f_{2,1} \\
  \dots & \dots \\
  f_{k,0}, & -f_{k,1}
\end{array}%
\right)_R.
\]
So, by \cite[Proposition 7.12]{Wu-color},
\[
C_P(-\Gamma) \cong \left(%
\begin{array}{cc}
  f_{1,0}, & -f_{1,1} \\
  f_{2,0}, & -f_{2,1} \\
  \dots & \dots \\
  f_{k,0}, & -f_{k,1}
\end{array}%
\right)_R \cong \left(%
\begin{array}{cc}
  \sqrt{-1}\cdot f_{1,0}, & \sqrt{-1}\cdot f_{1,1} \\
  \sqrt{-1}\cdot f_{2,0}, & \sqrt{-1}\cdot f_{2,1} \\
  \dots & \dots \\
  \sqrt{-1}\cdot f_{k,0}, & \sqrt{-1}\cdot f_{k,1}
\end{array}%
\right)_R =
F_{\sqrt{-1}}(C_P(\Gamma)).
\]
This shows that, for any $D$, the chain complexes $C_P(-D)$ and $F_{\sqrt{-1}} (C_P(D))$ have isomorphic terms. By the definition of the differential map of $C_{P}(D)$, especially \cite[Lemma 11.11]{Wu-color}, it follows that $C_{P}(-D) \cong F_{\sqrt{-1}} (C_P(D))$ as objects of $\hch(\hmf)$. Then, by isomorphism \eqref{equ-iso-homology-chain-scaling}, we have 
\[
H_{P}(-D) \cong H(F_{\sqrt{-1}} (C_P(D));R) \cong H_{P}(D).
\]
\end{proof}

\subsection{Linear change of roots} In this subsection, we study the effects of linear substitutions of variables and prove Proposition \ref{thm-linear-change-roots}.

\begin{lemma}\label{lemma-twist-elementary-sym-poly}
Let $\mathbb{X}=\{x_1,\dots,x_m\}$, $\mathbb{Y}=\{y_1,\dots,y_m\}$, $\hat{\mathbb{X}}=\{\hat{x}_1,\dots,\hat{x}_m\}$ and $\hat{\mathbb{Y}}=\{\hat{y}_1,\dots,\hat{y}_m\}$ be four disjoint alphabets of $m$ indeterminates. Define $\psi$ to be the ring isomorphism $\psi:\C[\hat{\mathbb{X}},\hat{\mathbb{Y}}] \rightarrow \C[\mathbb{X},\mathbb{Y}]$ given by $\psi(\hat{x}_i)=ax_i+b$, $\psi(\hat{y}_i)=ay_i+b$, where $a, b \in \C$ and $a \neq 0$. Denote by $X_i$, $Y_i$, $\hat{X}_i$ and $\hat{Y}_i$ the $i$th elementary symmetric polynomials in $\mathbb{X}$, $\mathbb{Y}$, $\hat{\mathbb{X}}$ and $\hat{\mathbb{Y}}$. Then $\psi(\hat{X}_i-\hat{Y}_i) = a^i (X_i-Y_i) + \sum_{j=1}^{i-1} (X_j-Y_j) \cdot g_j$, where each $g_j$ is an element of $\C[X_1,\dots,X_{i-1},Y_1,\dots,Y_{i-1}]$ with total degree at most $2i-2j-2$. (Recall that each of $x_j$ and $y_j$ has degree $2$.)
\end{lemma}

\begin{proof}
Note $\psi(\hat{X}_i)=a^i X_i + f_{m,a,b}(X_1,\dots,X_{i-1})$ and $\psi(\hat{Y}_i)=a^i Y_i + f_{m,a,b}(Y_1,\dots,Y_{i-1})$, where $f_{m,a,b}(X_1,\dots,X_{i-1}) \in \C[X_{i-1},\dots,X_1]$ has total degree at most $2i-2$. For $1 \leq j \leq i-1$ define
\[
g_j = \frac{f_{m,a,b}(Y_1,\dots,Y_{j-1},X_j,X_{j+1},\dots,X_{i-1}) - f_{m,a,b}(Y_1,\dots,Y_{j-1},Y_j,X_{j+1},\dots,X_{i-1})}{X_j-Y_j}.
\]
Then $g_1,\dots,g_{i-1}$ satisfy all the requirements in the lemma.
\end{proof}

\begin{proposition}\label{prop-linear-twist-MOY}
Let $\hat{P}(X)$ be a polynomial of the form \eqref{def-P}, not necessarily generic. Define $P(X)=\hat{P}(aX+b)$, where $a,b\in \C$ and $a\neq0$. Then $H_P(D) \cong H_{\hat{P}}(D)$ for any knotted MOY graph $D$, where the isomorphism preserves both the quantum filtration and the homological grading.
\end{proposition}

\begin{figure}[ht]
\[
\xymatrix{
\input{general-MOY-vertex} && \input{general-MOY-vertex-hat}
}
\]
\caption{}\label{prop-linear-twist-MOY-fig-1}

\end{figure}

\begin{proof}
We first consider embedded MOY graphs. Let $v$ be a vertex in a MOY graph. We give two markings of a neighborhood $\Gamma$ of $v$ as in Figure \ref{prop-linear-twist-MOY-fig-1}. Recall that we have $i_1+i_2+\cdots +i_k = j_1+j_2+\cdots +j_l\triangleq m$. Define 
\begin{eqnarray*}
R & = & \Sym(\mathbb{X}_1|\dots|\mathbb{X}_k|\mathbb{Y}_1|\dots|\mathbb{Y}_l),\\
\hat{R} & = & \Sym(\hat{\mathbb{X}}_1|\dots|\hat{\mathbb{X}}_k|\hat{\mathbb{Y}}_1|\dots|\hat{\mathbb{Y}}_l).
\end{eqnarray*} 
Write 
\begin{eqnarray*}
\mathbb{X} & = & \mathbb{X}_1\cup\cdots\cup \mathbb{X}_k, \\
\mathbb{Y} & = & \mathbb{Y}_1\cup\cdots\cup \mathbb{Y}_l, \\
\hat{\mathbb{X}} & = & \hat{\mathbb{X}}_1\cup\cdots\cup \hat{\mathbb{X}}_k, \\
\hat{\mathbb{Y}} & = & \hat{\mathbb{Y}}_1\cup\cdots\cup \hat{\mathbb{Y}}_l.
\end{eqnarray*}
Denote by $X_i$, $Y_i$, $\hat{X}_i$ and $\hat{Y}_i$ the $i$th elementary symmetric polynomials in $\mathbb{X}$, $\mathbb{Y}$, $\hat{\mathbb{X}}$ and $\hat{\mathbb{Y}}$. Then
\begin{eqnarray*}
C_P(\Gamma) & = & \left(%
\begin{array}{cc}
  \ast & X_1-Y_1 \\
  \ast & X_2-Y_2 \\
  \dots & \dots \\
  \ast & X_m-Y_m
\end{array}%
\right)_R
\{q^{-\sum_{1\leq s<t \leq k} i_si_t}\}, \\
C_{\hat{P}}(\Gamma) & = & \left(%
\begin{array}{cc}
  \ast & \hat{X_1}-\hat{Y_1} \\
  \ast & \hat{X_2}-\hat{Y_2} \\
  \dots & \dots \\
  \ast & \hat{X_m}-\hat{Y_m}
\end{array}%
\right)_{\hat{R}}
\{q^{-\sum_{1\leq s<t \leq k} i_si_t}\}.
\end{eqnarray*}

Using the isomorphism $\psi$ in Lemma \ref{lemma-twist-elementary-sym-poly}, $C_{\hat{P}}(\Gamma)$ can be viewed as a $R$-module and, therefore, a matrix factorization over $R$ with potential 
\[
\psi(\hat{P}(\hat{\mathbb{X}}) - \hat{P}(\hat{\mathbb{Y}}))= P(\mathbb{X})-P(\mathbb{Y}).
\]
In fact, as a matrix factorization over $R$, $C_{\hat{P}}(\Gamma)$ is the filtered Koszul matrix factorization
\[
C_{\hat{P}}(\Gamma) = \left(%
\begin{array}{cc}
  \ast & \psi(\hat{X_1}-\hat{Y_1}) \\
  \ast & \psi(\hat{X_2}-\hat{Y_2}) \\
  \dots & \dots \\
  \ast & \psi(\hat{X_m}-\hat{Y_m})
\end{array}%
\right)_R
\{q^{-\sum_{1\leq s<t \leq k} i_si_t}\}.
\]
By Lemma \ref{lemma-twist-elementary-sym-poly} and \cite[Corollary 2.15]{Wu7}, we have
\[
C_{\hat{P}}(\Gamma) \cong \left(%
\begin{array}{cc}
  \ast & a(X_1-Y_1) \\
  \ast & a^2(X_2-Y_2) \\
  \dots & \dots \\
  \ast & a^m(X_m-Y_m)
\end{array}%
\right)_R
\{q^{-\sum_{1\leq s<t \leq k} i_si_t}\}.
\]
Then, by \cite[Proposition 7.12]{Wu-color}, 
\[
C_{\hat{P}}(\Gamma) \cong \left(%
\begin{array}{cc}
  \ast & X_1-Y_1 \\
  \ast & X_2-Y_2 \\
  \dots & \dots \\
  \ast & X_m-Y_m
\end{array}%
\right)_R
\{q^{-\sum_{1\leq s<t \leq k} i_si_t}\}.
\]
Finally, by (the filtered version of) \cite[Theorem 2.1]{KR3}, we have $C_{\hat{P}}(\Gamma) \cong C_P(\Gamma)$ as filtered matrix factorizations over $R$. By \cite[Definitions 3.4 and 9.2]{Wu-color-equi}, this implies that, for any embedded MOY graph $\Gamma$,
\begin{equation}\label{prop-linear-twist-MOY-eq-1}
C_{\hat{P}}(\Gamma) \cong C_P(\Gamma).
\end{equation} 

The above shows that, for any knotted MOY graph $D$, the chain complexes $C_{\hat{P}}(D)$ and $C_P(D)$ have isomorphic terms (as filtered matrix factorizations.) Note that any morphism on $C_{\hat{P}}(\Gamma)$ can be pulled back to $C_P(\Gamma)$ via the isomorphism in \eqref{prop-linear-twist-MOY-eq-1}. By the definition of the differential map of $C_{P}(D)$, especially \cite[Lemma 11.11]{Wu-color}, it follows that $C_{\hat{P}}(D) \cong C_P(D)$ as objects of $\hch(\hmf_{R,w})$. Therefore, $H_{\hat{P}}(D) \cong H_P(D)$.
\end{proof}

Next we prove the following slight generalization of Proposition \ref{thm-linear-change-roots}.

\begin{proposition}\label{thm-linear-change-roots-MOY}
Let $P(X)$ and $\hat{P}(X)$ be generic polynomials of the form \eqref{def-P}. If there are $a,b\in \C$ with $a\neq0$ such that $\Sigma(\hat{P}) = \{ar+b|r\in\Sigma(P)\}$, then $H_P(D) \cong H_{\hat{P}}(D)$ for any knotted MOY graph $D$, where the isomorphism preserves both the quantum filtration and the homological grading.
\end{proposition}

\begin{proof}
Suppose $\Sigma(P)=\{r_1,\dots,r_N\}$. Then 
\[
P(X) = (N+1) \int_0^X (t-r_1)\cdots(t-r_N) ~dt
\]
and
\[
\hat{P}(X) = (N+1) \int_0^X (t-ar_1-b)\cdots(t-ar_N-b) ~dt.
\]
So $\hat{P}(aX+b) = a^{N+1} P(X) +c$, where $c$ is a constant.

By Proposition \ref{prop-linear-twist-MOY}, we have $H_{\hat{P}}(D) \cong H_{a^{N+1} P(X) +c}(D)$. It is clear that changing $c$ does not change the homology. So $H_{\hat{P}}(D) \cong H_{a^{N+1} P(X)}(D)$. It then follows from Corollary \ref{coro-scaling-P} that $H_{\hat{P}}(D) \cong H_P(D)$.
\end{proof}

\section{A Basis for Generic Deformations of the Graph Homology}\label{sec-basis-MOY-construction}

Gornik's construction of the basis for generic deformations of the uncolored $\slmf(N)$-homology in \cite{Gornik} is essentially based on the classical Lagrange interpolation for single variable polynomials. We generalize Gornik's construction to the colored situation using the interpolation formula for symmetric polynomials given by Chen and Louck \cite{Chen-Louck}, which generalize the classical Lagrange interpolation. In this section, we construct a basis for generic deformations of the graph homology. A basis for generically deformed link homology will be constructed in Section \ref{sec-basis-link-construction} using the results from this section and next section.

\subsection{The Chen-Louck interpolation formula} In this subsection, we review the interpolation formula for symmetric polynomials given by Chen and Louck \cite{Chen-Louck}, which plays a central role in our construction. 

Recall that an alphabet $\mathbb{X}=\{x_1,\dots,x_m\}$ is a finite collection of homogeneous indeterminates of degree $2$. (That is, our degree convention is twice the usual degree.) For a monomial $f=x_1^{k_1}\cdots x_m^{k_m}$, the partial degree of $f$ in $x_i$ is $2k_i$. For a polynomial $f \in \C[x_1,\dots,x_m]$, its partial degree in $x_i$ is the maximal partial degree in $x_i$ of its monomials. If $f(\mathbb{X})$ is a symmetric polynomial in $\mathbb{X}$ over $\C$, and $\Omega$ a set of $m$ complex numbers, then $f(\Omega)=f(\mathbb{X})|_{\mathbb{X}=\Omega}$ is well defined. The following is the main theorem of \cite{Chen-Louck}.

\begin{theorem}\cite[Theorem 2.1]{Chen-Louck}\label{Chen-Louck-interpolation}
Suppose that $1\leq m \leq N$. Let $\mathbb{X}$ be an alphabet of $m$ indeterminates and $\Sigma$ a set of $N$ distinct complex numbers. Assume that $f(\mathbb{X})$ is a symmetric polynomial in $\mathbb{X}$ over $\C$ and that its partial degrees are all less than or equal to $2(N-m)$. Then
\[
f(\mathbb{X}) = \sum_{\Omega \subset \Sigma,~|\Omega|=m} f(\Omega) \sfrac{\prod_{x \in \mathbb{X}, ~r \in \Sigma \setminus \Omega} (x-r)}{\prod_{s \in \Omega, ~r \in \Sigma \setminus \Omega} (s-r)}.
\]
\end{theorem}

It is clear that the above theorem becomes the classical Lagrange interpolation if $m=1$.

\subsection{Homology of a colored circle} Recall that $S_{\lambda}$ is the Schur polynomial associated to the partition $\lambda$. In particular, $h_k=S_k$ is the $k$-th complete symmetric polynomial. As in \cite{Wu-color}, define
\begin{eqnarray*}
\lambda_{m,n} & = & (\underbrace{n\geq\cdots\geq n}_{m \text{ parts}}), \\
\Lambda_{m,n} & = & \{\lambda=(\lambda_1\geq\cdots\geq\lambda_m)~|~\lambda_m \geq 0,~\lambda_1\leq n\}.
\end{eqnarray*}
Recall that, for an alphabet $\mathbb{X}$, we denote by $\Sym(\mathbb{X})$ the ring of symmetric polynomials over $\C$ in $\mathbb{X}$. More generally, given a collection $\{\mathbb{X}_1,\dots,\mathbb{X}_l\}$ of pairwise disjoint alphabets, we denote by $\Sym(\mathbb{X}_1|\cdots|\mathbb{X}_l)$ the ring of polynomials in $\mathbb{X}_1\cup\cdots\cup\mathbb{X}_l$ over $\C$ that are symmetric in each $\mathbb{X}_i$.

\begin{proposition}\label{prop-circle-ring-dim-base}
Suppose that $1\leq m \leq N$. Let $\mathbb{X}$ be an alphabet of $m$ indeterminates and $\Sigma$ an unordered tuple of $N$ complex numbers. Then 
\[
R=\sfrac{\Sym(\mathbb{X})}{(h_{N+1-m}(\mathbb{X}-\Sigma),h_{N+2-m}(\mathbb{X}-\Sigma),\dots,h_{N}(\mathbb{X}-\Sigma))}
\] 
is $\bn{N}{m}$-dimensional over $\C$.

Moreover, $\{S_{\lambda}(\mathbb{X})~|~\lambda \in \Lambda_{m,N-m}\}$ and $\{S_{\lambda}(\mathbb{X}-\Sigma)~|~\lambda \in \Lambda_{m,N-m}\}$ are two $\C$-linear bases for $R$. And there is a $\C$-linear map $\zeta_{\Sigma}:R\rightarrow \C$ satisfying 
\begin{equation}\label{eq-zeta-Sigma-bases}
\zeta_{\Sigma} (S_{\lambda}(\mathbb{X}) \cdot S_{\mu}(\mathbb{X}-\Sigma)) = \begin{cases}
1 & \text{if } \mu=\lambda^c , \\
0 & \text{if } \mu\neq\lambda^c,
\end{cases}
\end{equation}
where, for $\lambda=(\lambda_1\geq\cdots\geq\lambda_m) \in \Lambda_{m,N-m}$, $\lambda^c = (N-m-\lambda_m,\dots,N-m-\lambda_1)$.
\end{proposition}

\begin{proof}
Let $\mathbb{B}=\{\beta_1,\dots,\beta_N\}$ be an alphabet with $N$ indeterminates disjoint from $\mathbb{X}$. Then, by \cite[Theorem 2.10]{Wu-color-equi}, the quotient ring
\[
\hat{R}:=\Sym(\mathbb{X}|\mathbb{B})/(h_{N}(\mathbb{X}-\mathbb{B}), h_{N-1}(\mathbb{X}-\mathbb{B}),\dots, h_{N+1-m}(\mathbb{X}-\mathbb{B}))
\] 
is a graded-free $\Sym(\mathbb{B})$-module. $\{S_\lambda(\mathbb{X})~|~ \lambda \in \Lambda_{m,N-m}\}$ and $\{S_\lambda(\mathbb{X}-\mathbb{B})~|~ \lambda \in \Lambda_{m,N-m}\}$ are two homogeneous bases for the $\Sym(\mathbb{B})$-module $\hat{R}$. In particular, $\hat{R} \cong \Sym(\mathbb{B}) \{\qb{N}{m}\}$
as $\Sym(\mathbb{B})$-modules. Moreover, there is a unique $\Sym(\mathbb{B})$-module homomorphism $\zeta:\hat{R} \rightarrow \Sym(\mathbb{B})$ such that, for $\lambda,\mu \in \Lambda_{m,N-m}$,
\begin{equation}\label{eq-zeta-bases}
\zeta(S_\lambda(\mathbb{X}) \cdot S_\mu(\mathbb{X}-\mathbb{B})) = \begin{cases}
1 & \text{if } \mu=\lambda^c , \\
0 & \text{if } \mu\neq\lambda^c.
\end{cases}
\end{equation}

Write $\Sigma =\{r_1,\dots,r_N\}$. Consider the diagram
\[
\xymatrix{
\hat{R} \ar[d]^{\zeta} \ar[rr]^<<<<<<<<<<<<<<{\hat{\pi}} && R \cong \sfrac{\hat{R}}{(\beta_1-r_1,\dots,\beta_N-r_N)} \ar@{-->}[d]^{\zeta_\Sigma} \\
\Sym(\mathbb{B}) \ar[rr]^<<<<<<<<<<{\pi} && \C \cong \sfrac{\Sym(\mathbb{B})}{(\beta_1-r_1,\dots,\beta_N-r_N)} 
}
\]
where $\hat{\pi}$ and $\pi$ are the standard quotient maps. Since $\zeta$ is $\Sym(\mathbb{B})$-linear, it induces a $\C$-linear map $\zeta_{\Sigma}:R\rightarrow \C$, which make the above diagram commutative. Since $\{S_\lambda(\mathbb{X})~|~ \lambda \in \Lambda_{m,N-m}\}$ and $\{S_\lambda(\mathbb{X}-\mathbb{B})~|~ \lambda \in \Lambda_{m,N-m}\}$ are two bases for the $\Sym(\mathbb{B})$-module $\hat{R}$, we know that each of $\{S_{\lambda}(\mathbb{X})=\hat{\pi}(S_\lambda(\mathbb{X}))~|~\lambda \in \Lambda_{m,N-m}\}$ and $\{S_{\lambda}(\mathbb{X}-\Sigma)=\hat{\pi}(S_\lambda(\mathbb{X}-\mathbb{B}))~|~\lambda \in \Lambda_{m,N-m}\}$ spans $R$ over $\C$. Using the above commutative diagram and Equation \eqref{eq-zeta-bases}, it easy to check that Equation \eqref{eq-zeta-Sigma-bases} is true. This implies that each of $\{S_{\lambda}(\mathbb{X})~|~\lambda \in \Lambda_{m,N-m}\}$ and $\{S_{\lambda}(\mathbb{X}-\Sigma)~|~\lambda \in \Lambda_{m,N-m}\}$ is linearly independent and, therefore, a basis for $R$.
\end{proof}

Recall that, for a commutative ring $R$ and $a_1,\dots,a_k \in R$, the sequence $\{a_1,\dots,a_k\}$ is called $R$-regular if $a_1$ is not a zero divisor in $R$ and $a_j$ is not a zero divisor in $R/(a_1,\dots,a_{j-1})$ for $j=2,\dots,k$.

\begin{lemma}\label{lemma-regular-complete-Sigma}
Suppose that $1\leq m \leq N$. Let $\mathbb{X}$ be an alphabet of $m$ indeterminates and $\Sigma$ an unordered tuple of $N$ complex numbers. Then $\{h_N(\mathbb{X}-\Sigma), h_{N-1}(\mathbb{X}-\Sigma), \dots, h_{N+1-m}(\mathbb{X}-\Sigma)\}$ is $\Sym(\mathbb{X})$-regular.
\end{lemma}

\begin{proof}
By \cite[Proposition 6.2]{Wu-color}, we know that $\{h_N(\mathbb{X}), h_{N-1}(\mathbb{X}), \dots, h_{N+1-m}(\mathbb{X})\}$ is $\Sym(\mathbb{X})$-regular. Note that 
\[
h_m(\mathbb{X}-\Sigma) = h_m(\mathbb{X}) +\sum_{k=0}^{m-1} h_k(\mathbb{X}) h_{m-k}(-\Sigma). 
\]

We prove by induction that $\{h_N(\mathbb{X}-\Sigma), h_{N-1}(\mathbb{X}-\Sigma), \dots, h_{N+1-k}(\mathbb{X}-\Sigma)\}$ is $\Sym(\mathbb{X})$-regular for $1\leq k \leq m$. When $k=1$, this statement is trivially true. Assume $\{h_N(\mathbb{X}-\Sigma), h_{N-1}(\mathbb{X}-\Sigma), \dots, h_{N+1-k}(\mathbb{X}-\Sigma)\}$ is $\Sym(\mathbb{X})$-regular for a given $1\leq k \leq m-1$. 

Consider the sequence 
\[
\{h_N(\mathbb{X}-\Sigma), h_{N-1}(\mathbb{X}-\Sigma), \dots, h_{N+1-k}(\mathbb{X}-\Sigma), h_{N-k}(\mathbb{X}-\Sigma)\}.
\] 
To prove this sequence is $\Sym(\mathbb{X})$-regular, we only need to show that $h_{N-k}(\mathbb{X}-\Sigma)$ is not a zero divisor in $\sfrac{\Sym(\mathbb{X})}{(h_N(\mathbb{X}-\Sigma), h_{N-1}(\mathbb{X}-\Sigma), \dots, h_{N+1-k}(\mathbb{X}-\Sigma))}$. Suppose that $h_{N-k}(\mathbb{X}-\Sigma)$ is a zero divisor in this quotient ring. Then there is a $g \in \Sym(\mathbb{X})$, such that
\begin{enumerate}[(i)]
	\item $g \notin (h_N(\mathbb{X}-\Sigma), h_{N-1}(\mathbb{X}-\Sigma), \dots, h_{N+1-k}(\mathbb{X}-\Sigma))$,
	\item $g \cdot h_{N-k}(\mathbb{X}-\Sigma) \in (h_N(\mathbb{X}-\Sigma), h_{N-1}(\mathbb{X}-\Sigma), \dots, h_{N+1-k}(\mathbb{X}-\Sigma))$,
	\item $\deg g$ is minimal among elements of $\Sym(\mathbb{X})$ satisfying (i) and (ii).
\end{enumerate}

Consider the set 
\[
G=\{(g_1,\dots,g_k)~|~ g_j \in \Sym(\mathbb{X})~\forall j, ~\sum_{j=1}^k g_j h_{N+1-j}(\mathbb{X}-\Sigma)= g h_{N-k}(\mathbb{X}-\Sigma)\}.
\]
Define a positive integer-valued function $\delta:G\rightarrow \zed_+$ by 
\[
\delta(g_1,\dots,g_k) = \max \{\deg g_j + 2(N+1-j)~|~ j=1,\dots, k\}.
\]
(Keep in mind that our degree is twice the usual degree.) Let $\delta_0$ be the minimal value of $\delta$. It is clear that $\delta_0 \geq \deg g +2(N-k)$. We claim that 
\begin{equation}\label{eq-1-lemma-regular-complete-Sigma}
\delta_0 = \deg g +2(N-k).
\end{equation} 

We prove Equation \eqref{eq-1-lemma-regular-complete-Sigma} by contradiction. Assume that $\delta_0 > \deg g +2(N-k)$. Let $G_{\delta_0} = \{(g_1,\dots,g_k) \in G ~|~ \delta(g_1,\dots,g_k)=\delta_0\}$. Define a function $\jmath:G_{\delta_0} \rightarrow \{1,\dots,k\}$ by 
\[
\jmath(g_1,\dots,g_k) = \max \{j~|~ \deg g_j + 2(N+1-j) = \delta_0\}.
\]
Denote by $\jmath_0$ the minimal value of $\jmath$. Choose a $(g_1,\dots,g_k) \in G_{\delta_0}$ so that $\jmath(g_1,\dots,g_k) =\jmath_0$. 

If $\jmath_0=1$, let $g_1^{top}$ be the homogeneous part of $g_1$ of degree $\delta_0 - 2N$. Then $g_1^{top} \neq 0$, $\deg g_j h_{N+1-j}(\mathbb{X}-\Sigma) <\delta_0$ and $\deg g h_{N-k}(\mathbb{X}-\Sigma)<\delta_0$. Comparing the homogeneous parts of degree $\delta_0$ of both sides of
\begin{equation}\label{eq-2-lemma-regular-complete-Sigma}
\sum_{j=1}^k g_j h_{N+1-j}(\mathbb{X}-\Sigma)= g h_{N-k}(\mathbb{X}-\Sigma),
\end{equation} 
we get that $g_1^{top} h_N(\mathbb{X}) = 0$. This is a contradiction.

If $\jmath_0 >1$, let $g_j^{top}$ be the homogeneous part of $g_j$ of degree $\delta_0 - 2(N+1-j)$ for $j=1,\dots,\jmath_0$. Then $g_{\jmath_0}^{top} \neq 0$. Comparing the homogeneous parts of degree $\delta_0$ of both sides of Equation \eqref{eq-2-lemma-regular-complete-Sigma}, we get
\[
\sum_{j=1}^{\jmath_0} g_j^{top} h_{N+1-j}(\mathbb{X})=0.
\]
But $\{h_N(\mathbb{X}), h_{N-1}(\mathbb{X}), \dots, h_{N+1-k}(\mathbb{X})\}$ is $\Sym(\mathbb{X})$-regular, this implies that $g_{\jmath_0}^{top} \in (h_{N}(\mathbb{X}),\dots,h_{N+2-\jmath_0}(\mathbb{X}))$. So 
\[
g_{\jmath_0}^{top} = \sum_{j=1}^{\jmath_0-1} f_j h_{N+1-j}(\mathbb{X}),
\]
where $f_j\in \Sym(\mathbb{X})$ is homogeneous of degree $\delta_0 - 2(N+1-\jmath_0) - 2(N+1-j)$. So 
\begin{eqnarray*}
g_{\jmath_0} & = &  (g_{\jmath_0}-g_{\jmath_0}^{top}) + g_{\jmath_0}^{top} \\
& = &  (g_{\jmath_0}-g_{\jmath_0}^{top}) + \sum_{j=1}^{\jmath_0-1} f_j h_{N+1-j}(\mathbb{X}) \\
& = & (g_{\jmath_0}-g_{\jmath_0}^{top}) + \sum_{j=1}^{\jmath_0-1} f_j (h_{N+1-j}(\mathbb{X}) - h_{N+1-j}(\mathbb{X}-\Sigma)) + \sum_{j=1}^{\jmath_0-1} f_j h_{N+1-j}(\mathbb{X}-\Sigma).
\end{eqnarray*}
Define $\hat{g}_j \in \Sym(\mathbb{X})$ by
\[
\hat{g}_j = \begin{cases}
g_j + f_j h_{N+1-\jmath_0}(\mathbb{X}-\Sigma) & \text{if } j < \jmath_0, \\
(g_{\jmath_0}-g_{\jmath_0}^{top}) + \sum_{j=1}^{\jmath_0-1} f_j (h_{N+1-j}(\mathbb{X}) - h_{N+1-j}(\mathbb{X}-\Sigma)) & \text{if } j = \jmath_0, \\
g_j & \text{if } j > \jmath_0.
\end{cases}
\]
Then it is straightforward to check that $(\hat{g}_1,\dots,\hat{g}_k) \in G$, $\delta(\hat{g}_1,\dots,\hat{g}_k) \leq \delta_0$ and $\jmath(\hat{g}_1,\dots,\hat{g}_k) < \jmath_0$. This is a contradiction. Therefore Equation \eqref{eq-1-lemma-regular-complete-Sigma} is true.

Now pick a $(g_1,\dots,g_k) \in G_{\delta_0}$. Let $g^{top}$ be the homogeneous part of $g$ of degree $\deg g=\delta_0-2(N-k)$, and $g_j^{top}$ the homogeneous part of $g_j$ of degree $\delta_0 - 2(N+1-j)$ for $j=1,\dots,k$. Comparing the homogeneous parts of degree $\delta_0$ on both sides of Equation \eqref{eq-2-lemma-regular-complete-Sigma}, one can see that
\[
\sum_{j=1}^k g_j^{top} h_{N+1-j}(\mathbb{X})= g^{top} h_{N-k}(\mathbb{X}).
\]
But $\{h_N(\mathbb{X}), h_{N-1}(\mathbb{X}), \dots, h_{N+1-k}(\mathbb{X})\}$ is $\Sym(\mathbb{X})$-regular, this implies that $g^{top} \in (h_{N}(\mathbb{X}),\dots,h_{N+1-k}(\mathbb{X}))$. So 
\[
g^{top} = \sum_{j=1}^{k} \tilde{f}_j h_{N+1-j}(\mathbb{X}),
\]
where $\tilde{f}_j\in \Sym(\mathbb{X})$ is homogeneous of degree $\deg g - 2(N+1-j)$. Thus,
\begin{eqnarray*}
g & = &  (g-g^{top}) + g^{top} \\
& = &  (g-g^{top}) + \sum_{j=1}^{k} \tilde{f}_j h_{N+1-j}(\mathbb{X}) \\
& = & (g-g^{top}) + \sum_{j=1}^{k} \tilde{f}_j (h_{N+1-j}(\mathbb{X}) - h_{N+1-j}(\mathbb{X}-\Sigma)) + \sum_{j=1}^{k} \tilde{f}_j h_{N+1-j}(\mathbb{X}-\Sigma).
\end{eqnarray*}
So 
\[
\tilde{g}:= g - \sum_{j=1}^{k} \tilde{f}_j h_{N+1-j}(\mathbb{X}-\Sigma) =(g-g^{top}) + \sum_{j=1}^{k} \tilde{f}_j (h_{N+1-j}(\mathbb{X}) - h_{N+1-j}(\mathbb{X}-\Sigma))
\]
also satisfies conditions (i) and (ii) above and $\deg \tilde{g} < \deg g$. This is a contradiction. So $h_{N-k}(\mathbb{X}-\Sigma)$ is not a zero divisor in 
\[
\sfrac{\Sym(\mathbb{X})}{(h_N(\mathbb{X}-\Sigma), h_{N-1}(\mathbb{X}-\Sigma), \dots, h_{N+1-k}(\mathbb{X}-\Sigma))}.
\] 
Therefore, the sequence
\[
\{h_N(\mathbb{X}-\Sigma), h_{N-1}(\mathbb{X}-\Sigma), \dots, h_{N+1-k}(\mathbb{X}-\Sigma), h_{N-k}(\mathbb{X}-\Sigma)\}.
\] 
is $\Sym(\mathbb{X})$-regular. This completes the induction.
\end{proof}

\begin{figure}[ht]

\setlength{\unitlength}{1pt}

\begin{picture}(360,60)(-180,0)


\qbezier(0,60)(-20,60)(-20,45)

\qbezier(0,60)(20,60)(20,45)

\put(-20,15){\vector(0,1){30}}

\put(20,15){\line(0,1){30}}

\qbezier(0,0)(-20,0)(-20,15)

\qbezier(0,0)(20,0)(20,15)

\put(19,30){\line(1,0){2}}

\put(-15,30){\tiny{$m$}}

\put(25,25){\small{$\mathbb{X}$}}
 
\end{picture}

\caption{}\label{circle-module-figure}

\end{figure}

\begin{proposition}\label{prop-circle-module}
Suppose that $P(X)$ is a polynomial of the form \eqref{def-P}, not necessarily generic. Denote by $\Sigma(P)$ the unordered tuple of the $N$ roots of $P'(X)$, counting multiplicity. For $1 \leq m \leq N$, let $\bigcirc_m$ be the circle colored by $m$ with a single marked point in Figure \ref{circle-module-figure}. Then, as a filtered $\Sym(\mathbb{X})$-module, 
\[
H_P(\bigcirc_m) \cong \sfrac{\Sym(\mathbb{X})}{(h_N(\mathbb{X}-\Sigma(P)),h_{N-1}(\mathbb{X}-\Sigma(P)),\dots,h_{N+1-m}(\mathbb{X}-\Sigma(P)))} \{q^{-m(N-m)}\} \left\langle m \right\rangle,
\] 
where $\mathbb{X}$ is an alphabet of $m$ indeterminates.
\end{proposition}

\begin{proof}
By definition,
\[
C_P(\bigcirc_m) =
\left(%
\begin{array}{cc}
  U_1 & 0 \\
  \dots & \dots \\
  U_m & 0 
\end{array}%
\right)_{\Sym(\mathbb{X})},
\]
where $U_j=\frac{\partial}{\partial X_j}P(\mathbb{X})$ and $X_j$ is the $j$-th elementary symmetric polynomial in $\mathbb{X}$. By \cite[Lemma 2.9]{Wu-color-equi}, we know $U_j=(-1)^{j+1} (N+1) h_{N+1-j}(\mathbb{X}-\Sigma(P))$. The proposition then follows from Lemma \ref{lemma-regular-complete-Sigma} and \cite[Corollary 2]{KR1}. (See \cite[Proposition 2.16]{Wu7} for the filtered version of \cite[Corollary 2]{KR1} and \cite[Corollary 2.25]{Wu-color} for the grading/filtration shifting rules.)
\end{proof}

\subsection{The ring $\overline{R}_{\Gamma}$ associated to a closed abstract MOY graph $\Gamma$} In the remainder of this section, we fix a generic polynomial $P(X)$ of the form \eqref{def-P}. Denote by $\Sigma=\Sigma(P)=\{r_1,\dots,r_N\}$ the set of roots of $P'(X)$.

Let $\Gamma$ be a closed abstract MOY graph with a marking, and $\mathbb{X},\mathbb{Y},\dots,\mathbb{A}$ the alphabets assigned to all the marked points on $\Gamma$. Set $R_\Gamma=\Sym(\mathbb{X}|\mathbb{Y}|\dots|\mathbb{A})$.

If a component $\bigcirc$ of $\Gamma$ is a circle of color $m$ with a single marked point assigned the alphabet $\mathbb{X}_k$, then define $\mathcal{I}_{\bigcirc}$ to be the ideal 
\[
\mathcal{I}_{\bigcirc} = (h_N(\mathbb{X}_k-\Sigma),h_{N-1}(\mathbb{X}_k-\Sigma),\dots,h_{N+1-m}(\mathbb{X}_k-\Sigma))
\] 
of $R_\Gamma$.

Let $\hat{\Gamma}$ be $\Gamma$ with all circles with a single marked point removed. Cut $\hat{\Gamma}$ at all the marked points. This gives us a set of simple abstract MOY graphs $\Gamma_1,\dots,\Gamma_n$, each of which is a neighborhood of a vertex of $\Gamma$ of valence at least $2$. (If there are more than one marked points on an edge, we consider the arc between two adjacent marked points a neighborhood of an additional vertex of valence $2$.)

\begin{figure}[ht]

\setlength{\unitlength}{1pt}

\begin{picture}(360,80)(-180,-40)


\put(0,0){\vector(-1,1){15}}

\put(-15,15){\line(-1,1){15}}

\put(-23,25){\tiny{$i_1$}}

\put(-33,32){\small{$\mathbb{X}_1$}}

\put(0,0){\vector(-1,2){7.5}}

\put(-7.5,15){\line(-1,2){7.5}}

\put(-11,25){\tiny{$i_2$}}

\put(-18,32){\small{$\mathbb{X}_2$}}

\put(3,25){$\cdots$}

\put(0,0){\vector(1,1){15}}

\put(15,15){\line(1,1){15}}

\put(31,25){\tiny{$i_k$}}

\put(27,32){\small{$\mathbb{X}_k$}}


\put(-30,-30){\vector(1,1){15}}

\put(-15,-15){\line(1,1){15}}

\put(-26,-30){\tiny{$j_1$}}

\put(-33,-40){\small{$\mathbb{Y}_1$}}

\put(-15,-30){\vector(1,2){7.5}}

\put(-7.5,-15){\line(1,2){7.5}}

\put(-13,-30){\tiny{$j_2$}}

\put(-18,-40){\small{$\mathbb{Y}_2$}}

\put(3,-30){$\cdots$}

\put(30,-30){\vector(-1,1){15}}

\put(15,-15){\line(-1,1){15}}

\put(31,-30){\tiny{$j_l$}}

\put(27,-40){\small{$\mathbb{Y}_l$}}

\end{picture}

\caption{$\Gamma_p$}\label{piece-Gamma}

\end{figure}

Let $\Gamma_p$ in Figure \ref{piece-Gamma} be one of these neighborhoods. Note that $i_1+i_2+\cdots +i_k = j_1+j_2+\cdots +j_l \triangleq m$. Define 
\[
R_{\Gamma_p}=\Sym(\mathbb{X}_1|\dots|\mathbb{X}_k|\mathbb{Y}_1|\dots|\mathbb{Y}_l).
\]  
Denote by $\widetilde{X}_j$ the $j$-th elementary symmetric polynomial in $\mathbb{X}_1\cup\cdots\cup \mathbb{X}_k$ and by $\widetilde{Y}_j$ the $j$-th elementary symmetric polynomial in $\mathbb{Y}_1\cup\cdots\cup \mathbb{Y}_l$. Then 
\begin{eqnarray*}
P(\mathbb{X}_1\cup\cdots\cup \mathbb{X}_k) & = & P(\widetilde{X}_1\dots,\widetilde{X}_m), \\
P(\mathbb{Y}_1\cup\cdots\cup \mathbb{Y}_l) & = & P(\widetilde{Y}_1\dots,\widetilde{Y}_m).
\end{eqnarray*}
The matrix factorization associated to $\Gamma_p$ is
\[
C_P(\Gamma_p)=\left(%
\begin{array}{cc}
  U_1 & \widetilde{X}_1-\widetilde{Y}_1 \\
  U_2 & \widetilde{X}_2-\widetilde{Y}_2 \\
  \dots & \dots \\
  U_m & \widetilde{X}_m-\widetilde{Y}_m
\end{array}%
\right)_{R_{\Gamma_p}}
\{q^{-\sum_{1\leq s<t \leq k} i_si_t}\},
\]
where
\[
U_j = \frac{P(\widetilde{Y}_1,\dots,\widetilde{Y}_{j-1},\widetilde{X}_j,\dots,\widetilde{X}_m) - P(\widetilde{Y}_1,\dots,\widetilde{Y}_j,\widetilde{X}_{j+1},\dots,\widetilde{X}_m)}{\widetilde{X}_j-\widetilde{Y}_j},
\]
and the potential of $C_P(\Gamma_p)$ is $\sum_{j=1}^m (\widetilde{X}_j-\widetilde{Y}_j)U_j = P(\mathbb{X}_1\cup\cdots\cup \mathbb{X}_k)-P(\mathbb{Y}_1\cup\cdots\cup \mathbb{Y}_l)$. Let $\mathcal{I}_{\Gamma_p}$ be the ideal 
\[
\mathcal{I}_{\Gamma_p} = (\widetilde{X}_1-\widetilde{Y}_1,\dots,\widetilde{X}_m-\widetilde{Y}_m,U_1,\dots,U_m)
\]
of $R_\Gamma$.

\begin{definition}\label{def-overline-R-Gamma} 
Define 
\[
\mathcal{I}_{\Gamma} = \sum_{p=1}^n \mathcal{I}_{\Gamma_p} + \sum_{\bigcirc} \mathcal{I}_{\bigcirc},
\]
where $\bigcirc$ runs through all circles in $\Gamma$ with a single marked point.

The ring $\overline{R}_\Gamma$ is defined to be $\overline{R}_\Gamma = \sfrac{R_\Gamma}{\mathcal{I}_{\Gamma}}$. (Compare to \cite[Definition 2.2]{Gornik}.)
\end{definition}

\begin{lemma}\label{lemma-overline-R-Gamma}
Let $\Gamma$, $R_\Gamma$ and $\overline{R}_\Gamma$ be as above. Then
\begin{enumerate}[(i)]
	\item The action of $R_\Gamma$ on $H_P(\Gamma)$ factors through $\overline{R}_\Gamma$.
	\item Let $\mathbb{X}$ be an alphabet of $m$ indeterminates assigned to a marked point on $\Gamma$. Then, as elements of $\overline{R}_\Gamma$, 
	\[
	h_N(\mathbb{X}-\Sigma)=h_{N-1}(\mathbb{X}-\Sigma)=\cdots=h_{N+1-m}(\mathbb{X}-\Sigma)=0.
	\]
	\item Up to a natural isomorphism, $\overline{R}_\Gamma$ does not depend on the choice of marking.
\end{enumerate}
\end{lemma}

\begin{proof}
Part (i) follows from \cite[Proposition 2]{KR1}. 

It is easy to see that Part (ii) is true if $\mathbb{X}$ is the single alphabet on a circle colored by $m$. Otherwise, assume that $\mathbb{X}$ is at one of the end points of $\Gamma_p$ above. Note that
\[
\sum_{j=1}^m (\widetilde{X}_j-\widetilde{Y}_j)U_j = P(\mathbb{X}_1\cup\cdots\cup \mathbb{X}_k)-P(\mathbb{Y}_1\cup\cdots\cup \mathbb{Y}_l) = \pm P(\mathbb{X}) +\ast,
\]
where $\ast$ is independent of $\mathbb{X}$. Denote by $X_k$ the $k$-th elementary symmetric polynomial in $\mathbb{X}$. By \cite[Lemma 2.9]{Wu-color-equi},
\begin{eqnarray*}
 h_{N+1-k}(\mathbb{X}-\Sigma) & = & \frac{(-1)^{k+1}}{(N+1)} \frac{\partial}{\partial X_k}P(\mathbb{X}) \\
& = & \pm \frac{1}{(N+1)} \sum_{j=1}^m (U_j\frac{\partial}{\partial X_k}(\widetilde{X}_j-\widetilde{Y}_j) + (\widetilde{X}_j-\widetilde{Y}_j)\frac{\partial}{\partial X_k}U_j) \in ~\mathcal{I}_{\Gamma_p}.
\end{eqnarray*}
The proves Part (ii).

It is now easy to check that adding/removing a marked point does not change the ring $\overline{R}_{\Gamma}$, which implies Part (iii).
\end{proof}

\subsection{States of a closed abstract MOY graph} In this subsection, $\Gamma$ is again a closed abstract MOY graph. Denote by $E(\Gamma)$ the set of edges of $\Gamma$. We put exactly one marked point on each edge $e$ of $\Gamma$ and denote by $\mathbb{X}_e$ the alphabet assigned to this marked point. Recall that $\mathcal{P}(\Sigma)$ is the power set of $\Sigma$.

\begin{figure}[ht]

\setlength{\unitlength}{1pt}

\begin{picture}(360,80)(-180,-40)


\put(0,0){\vector(-1,1){15}}

\put(-15,15){\line(-1,1){15}}

\put(-33,32){\small{$e_1$}}

\put(0,0){\vector(-1,2){7.5}}

\put(-7.5,15){\line(-1,2){7.5}}

\put(-18,32){\small{$e_2$}}

\put(3,25){$\cdots$}

\put(0,0){\vector(1,1){15}}

\put(15,15){\line(1,1){15}}

\put(27,32){\small{$e_k$}}


\put(4,-2){$v$}


\put(-30,-30){\vector(1,1){15}}

\put(-15,-15){\line(1,1){15}}

\put(-33,-40){\small{$e'_1$}}

\put(-15,-30){\vector(1,2){7.5}}

\put(-7.5,-15){\line(1,2){7.5}}

\put(-18,-40){\small{$e'_2$}}

\put(3,-30){$\cdots$}

\put(30,-30){\vector(-1,1){15}}

\put(15,-15){\line(-1,1){15}}

\put(27,-40){\small{$e'_l$}}

\end{picture}

\caption{}\label{def-state-MOY-vertex-fig}

\end{figure}

\begin{definition}\label{def-state-MOY}
A pre-state $\varphi$ of $\Gamma$ is a function $\varphi:E(\Gamma)\rightarrow \mathcal{P}(\Sigma)$ such that $|\varphi(e)|$ is equal to the color of $e$. We denote by $\mathcal{S}'(\Gamma)$ the set of all pre-states of $\Gamma$.

Let $v$ be a vertex of $\Gamma$ of the form in Figure \ref{def-state-MOY-vertex-fig}. We say that a pre-state $\varphi$ is admissible at $v$ if 
\begin{enumerate}
	\item $\varphi(e_1),\dots,\varphi(e_k)$ are pairwise disjoint,
	\item $\varphi(e'_1),\dots,\varphi(e'_l)$ are pairwise disjoint,
	\item $\varphi(e_1)\cup\cdots\cup\varphi(e_k)=\varphi(e'_1)\cup\dots\cup\varphi(e'_l)$.
\end{enumerate}

A pre-state of $\Gamma$ is called a state of $\Gamma$ if it is admissible at every vertex of $\Gamma$. We denote by $\mathcal{S}(\Gamma)$ the set of all states of $\Gamma$.
\end{definition}

The goal of this subsection is to establish the decomposition of the ring $\overline{R}_{\Gamma}$ in Theorem \ref{thm-bar-R-Gamma-decompose}. We start with the following technical lemma.

\begin{lemma}\label{lemma-complete-symmetric-vanish}
Suppose that $1\leq m \leq N$ and $\Omega$ is a subset of $\Sigma$ with $|\Omega|=m$. Then $h_k(\Omega - \Sigma)=0$ for all $k>N-m$.
\end{lemma}

\begin{proof}
Recall that 
\[
\sum_{k=0}^\infty h_k(\Omega - \Sigma) t^k = \sfrac{\prod_{r \in \Sigma} (1-rt)}{\prod_{s \in \Omega} (1-st)} = \prod_{r \in \Sigma \setminus \Omega} (1-rt).
\]
The right hand side is clearly a polynomial in $t$ with (the usual) degree at most $N-m$. And the lemma follows. 
\end{proof}

For an edge $e$ in $\Gamma$ of color $m_e$, denote by $\mathcal{I}_e$ the ideal of $R_\Gamma$ generated by $h_N(\mathbb{X}_e-\Sigma),~h_{N-1}(\mathbb{X}_e-\Sigma),\dots,~h_{N+1-m_e}(\mathbb{X}_e-\Sigma)$. Define 
\begin{eqnarray*}
\widetilde{R}_{\Gamma} & = & \sfrac{R_\Gamma}{\sum_{e \in E(\Gamma)} \mathcal{I}_e} \\
& = & \bigotimes_{e \in E(\Gamma)}  \sfrac{\Sym (\mathbb{X}_e)}{(h_N(\mathbb{X}_e-\Sigma),h_{N-1}(\mathbb{X}_e-\Sigma),\dots,h_{N+1-m_e}(\mathbb{X}_e-\Sigma))},
\end{eqnarray*}
where $m_e$ is the color of $e$ and the tensor is taken over $\C$. By Part (ii) of Lemma \ref{lemma-overline-R-Gamma}, we know that $\overline{R}_{\Gamma}$ is a quotient ring of $\widetilde{R}_{\Gamma}$.

\begin{proposition}\label{prop-def-Q-varphi}
Let $\varphi \in \mathcal{S}'(\Gamma)$. Define 
\[
Q_{\varphi} = \prod_{e \in E(\Gamma)} \left( \sfrac{\prod_{x \in \mathbb{X}_e,~r\in \Sigma\setminus \varphi(e)} (x-r)}{\prod_{s \in \varphi(e),~r\in \Sigma\setminus \varphi(e)} (s-r)}\right).
\]
Then, as elements of $\widetilde{R}_{\Gamma}$,
\begin{enumerate}[(i)]
	\item $\{Q_\varphi~|~\varphi \in \mathcal{S}'(\Gamma)\}$ spans $\widetilde{R}_{\Gamma}$ over $\C$,
	\item $Q_\varphi^2 = Q_\varphi$ for all $\varphi \in \mathcal{S}'(\Gamma)$,
	\item $Q_{\varphi_1} Q_{\varphi_2}=0$ for all $\varphi_1,~\varphi_2 \in \mathcal{S}'(\Gamma)$ with $\varphi_1 \neq\varphi_2$,
	\item $\sum_{\varphi \in \mathcal{S}'(\Gamma)} Q_\varphi = 1$.
\end{enumerate}

In particular, since $\overline{R}_{\Gamma}$ is a quotient ring of $\widetilde{R}_{\Gamma}$, the above conclusions remain true in $\overline{R}_{\Gamma}$.
\end{proposition}

\begin{proof}
Recall that $R_\Gamma = \bigotimes_{e \in E(\Gamma)} \Sym (\mathbb{X}_e)$, where the tensor product is taken over $\C$. By Theorem \ref{Chen-Louck-interpolation}, $\sum_{\varphi \in \mathcal{S}'(\Gamma)} Q_\varphi = 1$ in $R_\Gamma$. So it is true in $\widetilde{R}_{\Gamma}$. This proves Part (iv).

By Proposition \ref{prop-circle-ring-dim-base}, every element of 
\[
\sfrac{\Sym (\mathbb{X}_e)}{(h_N(\mathbb{X}_e-\Sigma),h_{N-1}(\mathbb{X}_e-\Sigma),\dots,h_{N+1-m_e}(\mathbb{X}_e-\Sigma))}
\] 
can be represented by an element of $\Sym (\mathbb{X}_e)$ with all partial degrees less than or equal to $2(N-m_e)$.  By Theorem \ref{Chen-Louck-interpolation}, this implies that $\{Q_\varphi~|~\varphi \in \mathcal{S}'(\Gamma)\}$ spans $\widetilde{R}_{\Gamma}$ over $\C$. That is, we have proved Part (i).

Next, consider $Q_{\varphi_1} Q_{\varphi_2}$ for $\varphi_1,~\varphi_2 \in \mathcal{S}'(\Gamma)$. Note that, for any $\varphi,~\varphi' \in \mathcal{S}'(\Gamma)$,
\begin{equation}\label{eq-phi-phiprime-value}
Q_{\varphi'}|_{\mathbb{X}_e=\varphi(e),~\forall e \in E(\Gamma)} = \begin{cases}
1 & \text{if } \varphi=\varphi', \\
0 & \text{if } \varphi\neq\varphi'.
\end{cases}
\end{equation}
From Part (i), we know that, as an element of $R_{\Gamma}$,
\begin{equation}\label{prop-def-Q-varphi-eq-2}
Q_{\varphi_1} Q_{\varphi_2} = g + \sum_{\varphi \in \mathcal{S}'(\Gamma)} \alpha_{\varphi} Q_{\varphi},
\end{equation}
where $g \in \sum_{e \in E(\Gamma)} \mathcal{I}_e$ and $\alpha_{\varphi} \in \C$. Given any $\varphi \in \mathcal{S}'(\Gamma)$, by Lemma \ref{lemma-complete-symmetric-vanish}, we know that 
\begin{equation}\label{evaluation-well-defined-bar-R-Gamma}
g|_{\mathbb{X}_e=\varphi(e),~\forall e \in E(\Gamma)} =0.
\end{equation}
Plugging $\mathbb{X}_e=\varphi(e),~\forall e \in E(\Gamma)$ into Equation \eqref{prop-def-Q-varphi-eq-2}, we get, by Equation \eqref{eq-phi-phiprime-value},
\[
\alpha_{\varphi} = \begin{cases}
1 & \text{if } \varphi=\varphi_1=\varphi_2, \\
0 & \text{otherwise}.
\end{cases}
\]
This implies that Parts (ii) and (iii) are true in $\widetilde{R}_{\Gamma}$.
\end{proof}

\begin{corollary}\label{coro-multiply-by-Q-varphi-1}
Let $\widetilde{R}_{\Gamma}$ be defined as above. For any $\varphi \in \mathcal{S}'(\Gamma)$, the evaluation homomorphism 
\[
|_{\mathbb{X}_e=\varphi(e),~\forall e \in E(\Gamma)}:\widetilde{R}_{\Gamma} \rightarrow \C
\] 
given by $f \mapsto f|_{\mathbb{X}_e=\varphi(e),~\forall e \in E(\Gamma)}$ is well defined. Moreover, as elements of $\widetilde{R}_{\Gamma}$, 
$f \cdot Q_{\varphi} = (f|_{\mathbb{X}_e=\varphi(e),~\forall e \in E(\Gamma)}) \cdot Q_{\varphi}$ for any $f \in \widetilde{R}_{\Gamma}$
\end{corollary}

\begin{proof}
By Equation \eqref{evaluation-well-defined-bar-R-Gamma}, we know that $f|_{\mathbb{X}_e=\varphi(e),~\forall e \in E(\Gamma)}$ is well defined. By Part (i) of Proposition \ref{prop-def-Q-varphi}, we know that $f = \sum_{\varphi' \in \mathcal{S}'(\Gamma)} \alpha_{\varphi'} Q_{\varphi'}$, where $\alpha_{\varphi'} \in \C$. Plug $\mathbb{X}_e=\varphi(e),~\forall e \in E(\Gamma)$ into this equation, we get by Equation \eqref{eq-phi-phiprime-value} that $\alpha_\varphi = f|_{\mathbb{X}_e=\varphi(e),~\forall e \in E(\Gamma)}$. Then, by Parts (ii) and (iii) of Proposition \ref{prop-def-Q-varphi}, we know that $f \cdot Q_{\varphi} = \sum_{\varphi' \in \mathcal{S}'(\Gamma)} \alpha_{\varphi'} Q_{\varphi'}Q_{\varphi} = \alpha_\varphi Q_{\varphi}= (f|_{\mathbb{X}_e=\varphi(e),~\forall e \in E(\Gamma)})\cdot Q_{\varphi}$.
\end{proof}

Let $v$ be a vertex of $\Gamma$ as depicted in Figure \ref{def-state-MOY-vertex-fig}. Let $\mathbb{X} = \bigcup_{i=1}^k\mathbb{X}_{e_i}$ and $\mathbb{Y}= \bigcup_{j=1}^l\mathbb{X}_{e_j'}$. For an edge $e$ of $\Gamma$, denote by $m_e$ the color of $e$. Then $\sum_{i=1}^k m_{e_i} = \sum_{j=1}^l m_{e_j'} \triangleq m$. For $1\leq j \leq m$, denote by $\widetilde{X}_j$ and $\widetilde{Y}_j$ the $j$-th elementary symmetric polynomials in $\mathbb{X}$ and $\mathbb{Y}$, and define
\[
U_j = \frac{P(\widetilde{Y}_1,\dots,\widetilde{Y}_{j-1},\widetilde{X}_j,\dots,\widetilde{X}_m) - P(\widetilde{Y}_1,\dots,\widetilde{Y}_j,\widetilde{X}_{j+1},\dots,\widetilde{X}_m)}{\widetilde{X}_j-\widetilde{Y}_j},
\]
where $P$ is the polynomial in $\widetilde{X}_1,\dots,\widetilde{X}_m$ satisfying $P(\widetilde{X}_1,\dots,\widetilde{X}_m) = P(\mathbb{X})$.

\begin{lemma}\label{lemma-admissibility-evaluation}
For a pre-state $\varphi$ of $\Gamma$, the following two statements are equivalent.
\begin{enumerate}[(i)]
	\item $\varphi$ is admissible at $v$.
	\item For all $p=1,\dots,m$,
	\[
	\begin{cases}
	\widetilde{X}_p|_{\mathbb{X}_{e_i}=\varphi(e_i)~\forall i=1,\dots,k} = \widetilde{Y}_p|_{\mathbb{X}_{e_j'}=\varphi(e_j')~\forall j=1,\dots,l} ,\\
	U_p|_{\mathbb{X}_{e_i}=\varphi(e_i)~\forall i=1,\dots,k,~\mathbb{X}_{e_j'}=\varphi(e_j')~\forall j=1,\dots,l} = 0.
	\end{cases}
	\]
\end{enumerate}
\end{lemma}

\begin{proof}
Assume $\varphi$ is admissible at $v$. Then
\begin{enumerate}
	\item $\varphi(e_1),\dots,\varphi(e_k)$ are pairwise disjoint,
	\item $\varphi(e'_1),\dots,\varphi(e'_l)$ are pairwise disjoint,
	\item $\varphi(e_1)\cup\cdots\cup\varphi(e_k)=\varphi(e'_1)\cup\dots\cup\varphi(e'_l)$.
\end{enumerate}
So 
\begin{eqnarray*}
\widetilde{X}_p|_{\mathbb{X}_{e_i}=\varphi(e_i)~\forall i=1,\dots,k} 
& = & \widetilde{X}_p|_{\mathbb{X}=\varphi(e_1)\cup\cdots\cup\varphi(e_k)} \\
=~ \widetilde{Y}_p|_{\mathbb{Y}=\varphi(e'_1)\cup\dots\cup\varphi(e'_l)} 
& = & \widetilde{Y}_p|_{\mathbb{X}_{e_j'}=\varphi(e_j')~\forall j=1,\dots,l},
\end{eqnarray*}
and, by Lemma \ref{lemma-complete-symmetric-vanish} and \cite[Lemma 2.9]{Wu-color-equi},
\begin{eqnarray*}
&& U_p|_{\mathbb{X}_{e_i}=\varphi(e_i)~\forall i=1,\dots,k,~\mathbb{X}_{e_j'}=\varphi(e_j')~\forall j=1,\dots,l} \\
& = & \left(\frac{\partial}{\partial X_p} P(\widetilde{X}_1,\dots,\widetilde{X}_m)\right)|_{\mathbb{X}=\varphi(e_1)\cup\cdots\cup\varphi(e_k)} \\
& = & (-1)^{p+1} (N+1) h_{N+1-p}(\mathbb{X}-\Sigma)|_{\mathbb{X}=\varphi(e_1)\cup\cdots\cup\varphi(e_k)} \\
& = & 0.
\end{eqnarray*}
This show that statement (ii) in the lemma is true.

Now assume that statement (ii) in the lemma is true. We need to show that $\varphi$ is admissible at $v$, that is, conditions (1)-(3) above are satisfied. 

First, set 
\[
\beta_p = \widetilde{X}_p|_{\mathbb{X}_{e_i}=\varphi(e_i)~\forall i=1,\dots,k} = \widetilde{Y}_p|_{\mathbb{X}_{e_j'}=\varphi(e_j')~\forall j=1,\dots,l}.
\]
Then, counting multiplicity, $\varphi(e_1)\cup\cdots\cup\varphi(e_k)$ and $\varphi(e'_1)\cup\dots\cup\varphi(e'_l)$ are both the set of roots of the polynomial $t^m +\sum_{j=1}^m (-1)^j \beta_j t^{m-j}$. So condition (3) is true. 

It remains to check (1) and (2), which are equivalent to each other since (3) is true. Since $\widetilde{X}_p|_{\mathbb{X}_{e_i}=\varphi(e_i)~\forall i=1,\dots,k} = \widetilde{Y}_p|_{\mathbb{X}_{e_j'}=\varphi(e_j')~\forall j=1,\dots,l}$, we have that, by \cite[Lemma 2.9]{Wu-color-equi},
\begin{eqnarray*}
&& U_p|_{\mathbb{X}_{e_i}=\varphi(e_i)~\forall i=1,\dots,k,~\mathbb{X}_{e_j'}=\varphi(e_j')~\forall j=1,\dots,l} \\
& = & (-1)^{p+1} (N+1) h_{N+1-p}(\mathbb{X}-\Sigma)|_{\mathbb{X}_{e_i}=\varphi(e_i)~\forall i=1,\dots,k}.
\end{eqnarray*}
Recall that
\begin{equation}\label{lemma-admissibility-evaluation-eq-1}
\sum_{p=0}^\infty t^p h_{p}(\mathbb{X}-\Sigma)|_{\mathbb{X}_{e_i}=\varphi(e_i)~\forall i=1,\dots,k} 
= \sfrac{\prod_{r\in\Sigma}(1-rt)}{\prod_{i=1}^k \prod_{r\in\varphi(e_i)}(1-rt)}.
\end{equation}
If (1) is not true, then there is a repeated root $r \in \varphi(e_i) \cap \varphi(e_j)$ for some $i,j=1,\dots,k$ and $i \neq j$. If $r\neq 0$, then the right hand side of Equation \eqref{lemma-admissibility-evaluation-eq-1} is a power series with infinitely many non-vanishing terms. If the only repeated root is $r=0$, then the right hand side of Equation \eqref{lemma-admissibility-evaluation-eq-1} is a polynomial of (the usual) degree in $t$ at least $N+1-m$. In either case, there is a $p\geq N+1-m$ such that $h_{p}(\mathbb{X}-\Sigma)|_{\mathbb{X}_{e_i}=\varphi(e_i)~\forall i=1,\dots,k} \neq 0$. Note that, for any $p\geq N+1$, we have 
\begin{eqnarray*}
h_{p}(\mathbb{X}-\Sigma) & = & \sum_{i=0}^N h_{p-i}(\mathbb{X})h_{i}(-\Sigma) \\
& = & \sum_{i=0}^N \sum_{j=1}^m (-1)^{j-1} \widetilde{X}_j h_{p-i-j}(\mathbb{X}) h_{i}(-\Sigma) \\
& = & \sum_{j=1}^m (-1)^{j-1} \widetilde{X}_j \left(\sum_{i=0}^N  h_{p-i-j}(\mathbb{X}) h_{i}(-\Sigma) \right) \\
& = & \sum_{j=1}^m (-1)^{j-1} \widetilde{X}_j h_{p-j}(\mathbb{X}-\Sigma).
\end{eqnarray*}
This implies that the sequence $\{h_{p}(\mathbb{X}-\Sigma)\}_{p=N+1-m}^\infty$ is contained in the ideal generated by $h_{N+1-m}(\mathbb{X}-\Sigma), \dots, h_{N}(\mathbb{X}-\Sigma)$. So there must be some $p=1, \dots,m$, such that $h_{N+1-p}(\mathbb{X}-\Sigma)|_{\mathbb{X}_{e_i}=\varphi(e_i)~\forall i=1,\dots,k} \neq 0$. Thus, for this $p$, 
\[
U_p|_{\mathbb{X}_{e_i}=\varphi(e_i)~\forall i=1,\dots,k,~\mathbb{X}_{e_j'}=\varphi(e_j')~\forall j=1,\dots,l} \neq 0,
\] 
which contradicts statement (ii). This shows that condition (1) and, therefore, condition (2) are true. So $\varphi$ is admissible at $v$.
\end{proof}

\begin{corollary}\label{coro-multiply-by-Q-varphi-2}
For any $\varphi \in \mathcal{S}(\Gamma)$, the evaluation homomorphism 
\[
|_{\mathbb{X}_e=\varphi(e),~\forall e \in E(\Gamma)}:\overline{R}_{\Gamma} \rightarrow \C
\] 
given by $f \mapsto f|_{\mathbb{X}_e=\varphi(e),~\forall e \in E(\Gamma)}$ is well defined. Moreover, as elements of $\overline{R}_{\Gamma}$, $f \cdot Q_{\varphi} = (f|_{\mathbb{X}_e=\varphi(e),~\forall e \in E(\Gamma)}) \cdot Q_{\varphi}$. for any $f \in \overline{R}_{\Gamma}$.
\end{corollary}

\begin{proof}
By Lemma \ref{lemma-admissibility-evaluation}, $f|_{\mathbb{X}_e=\varphi(e),~\forall e \in E(\Gamma)}$ is well defined for $f \in \overline{R}_{\Gamma}$. Choose an $\tilde{f} \in \widetilde{R}_{\Gamma}$ that projects to $f$ under the standard quotient map $\widetilde{R}_{\Gamma} \rightarrow \overline{R}_{\Gamma}$. By Corollary \ref{coro-multiply-by-Q-varphi-1}, $\tilde{f} \cdot Q_{\varphi} = (\tilde{f}|_{\mathbb{X}_e=\varphi(e),~\forall e \in E(\Gamma)}) \cdot Q_{\varphi}$. Clearly, $\tilde{f}|_{\mathbb{X}_e=\varphi(e),~\forall e \in E(\Gamma)} = f|_{\mathbb{X}_e=\varphi(e),~\forall e \in E(\Gamma)}$. Then, projecting the previous equation onto $\overline{R}_{\Gamma}$, we get $f \cdot Q_{\varphi} = (f|_{\mathbb{X}_e=\varphi(e),~\forall e \in E(\Gamma)}) \cdot Q_{\varphi}$.
\end{proof}

\begin{theorem}\label{thm-bar-R-Gamma-decompose}
\begin{enumerate}[(a)]
	\item If $\varphi \in \mathcal{S}'(\Gamma) \setminus \mathcal{S}(\Gamma)$, then $Q_\varphi =0$ in $\overline{R}_{\Gamma}$.
	\item If $\varphi \in \mathcal{S}(\Gamma)$, then $Q_\varphi \neq 0$ in $\overline{R}_{\Gamma}$ and $\overline{R}_{\Gamma} \cdot Q_\varphi = \C \cdot Q_\varphi$.
	\item $\overline{R}_{\Gamma} = \bigoplus_{\varphi \in \mathcal{S}(\Gamma)} \C \cdot Q_\varphi$.
\end{enumerate}
\end{theorem}

\begin{proof}
Suppose $\varphi \in \mathcal{S}'(\Gamma) \setminus \mathcal{S}(\Gamma)$. Then there is a vertex $v$ of $\Gamma$ at which $\varphi$ is not admissible. Assume $v$ is given by Figure \ref{def-state-MOY-vertex-fig} and let $m$, $\widetilde{X}_j$, $\widetilde{Y}_j$ and $U_j$ be as in Lemma \ref{lemma-admissibility-evaluation}. Then, by Corollary \ref{coro-multiply-by-Q-varphi-1}, we know that, as elements of $\widetilde{R}_\Gamma$, 
\begin{eqnarray}
\label{thm-bar-R-Gamma-decompose-eq-1} (\widetilde{X}_p-\widetilde{Y}_p) \cdot Q_\varphi & = & (\widetilde{X}_p|_{\mathbb{X}_{e_i}=\varphi(e_i)~\forall i=1,\dots,k} - \widetilde{Y}_p|_{\mathbb{X}_{e_j'}=\varphi(e_j')~\forall j=1,\dots,l}) \cdot Q_\varphi, \\
\label{thm-bar-R-Gamma-decompose-eq-2} U_p \cdot Q_\varphi & = & U_p|_{\mathbb{X}_{e_i}=\varphi(e_i)~\forall i=1,\dots,k,~\mathbb{X}_{e_j'}=\varphi(e_j')~\forall j=1,\dots,l} \cdot Q_\varphi
\end{eqnarray}
for $p=1,\dots,m$. But, by Lemma \ref{lemma-admissibility-evaluation}, there is a $p=1,\dots,m$ such that either 
\[
\widetilde{X}_p|_{\mathbb{X}_{e_i}=\varphi(e_i)~\forall i=1,\dots,k} - \widetilde{Y}_p|_{\mathbb{X}_{e_j'}=\varphi(e_j')~\forall j=1,\dots,l} \neq 0
\] 
or 
\[U_p|_{\mathbb{X}_{e_i}=\varphi(e_i)~\forall i=1,\dots,k,~\mathbb{X}_{e_j'}=\varphi(e_j')~\forall j=1,\dots,l} \neq 0.
\]
Project Equations \eqref{thm-bar-R-Gamma-decompose-eq-1} and \eqref{thm-bar-R-Gamma-decompose-eq-2} onto $\overline{R}_{\Gamma}$. Note that $(\widetilde{X}_p-\widetilde{Y}_p) \cdot Q_\varphi = U_p \cdot Q_\varphi =0$ in $\overline{R}_{\Gamma}$. We know that there is a non-zero scalar $c$ such that $cQ_\varphi = 0$ in $\overline{R}_{\Gamma}$, which implies that $Q_\varphi =0$ in $\overline{R}_{\Gamma}$ and proves Part (a).

Now suppose that $\varphi \in \mathcal{S}(\Gamma)$. Note that $Q_{\varphi}|_{\mathbb{X}_e=\varphi(e),~\forall e \in E(\Gamma)}=1\neq0$. By Corollary \ref{coro-multiply-by-Q-varphi-2}, this implies that $Q_\varphi \neq 0$ in $\overline{R}_{\Gamma}$. The fact that $\overline{R}_{\Gamma} \cdot Q_\varphi = \C \cdot Q_\varphi$ also follows easily from Corollary \ref{coro-multiply-by-Q-varphi-2}. This proves Part (b).

By Part (a) and Part (i) of Proposition \ref{prop-def-Q-varphi}, we know that $\overline{R}_{\Gamma} = \sum_{\varphi \in \mathcal{S}(\Gamma)} \C \cdot Q_\varphi$. Then, using Parts (ii) and (iii) of Proposition \ref{prop-def-Q-varphi}, one can easily check that $\sum_{\varphi \in \mathcal{S}(\Gamma)} \C \cdot Q_\varphi$ is actually a direct sum. This shows $\overline{R}_{\Gamma} = \bigoplus_{\varphi \in \mathcal{S}(\Gamma)} \C \cdot Q_\varphi$ and proves Part (c).
\end{proof}

\subsection{Decomposition of the homology of embedded MOY graphs} We are now ready to generalize Gornik's decomposition of the graph homology \cite[Theorem 4]{Gornik}.

\begin{lemma}\label{lemma-MOY-homology-multiplication}
Let $\varphi \in \mathcal{S}(\Gamma)$ and $w \in H_P(\Gamma)$. Then the following statements are equivalent:
\begin{enumerate}[(i)]
	\item $w \in Q_\varphi H_P(\Gamma)$.
	\item $g \cdot w = (g|_{\mathbb{X}_e=\varphi(e),~\forall e \in E(\Gamma)}) \cdot w$ for all $g \in R_\Gamma$.
	\item $X_{e,i} \cdot w = (X_{e,i}|_{\mathbb{X}_e=\varphi(e)}) \cdot w$ for all $e \in E(\Gamma)$, where $X_{e,i}$ is the $i$-th elementary symmetric polynomial in $\mathbb{X}_e$.
\end{enumerate}
\end{lemma}

\begin{proof}
Since $R_\Gamma$ is the polynomial ring generated by $\{X_{e,i}\}$, one can see that statements (ii) and (iii) are equivalent. So we only need to check that they are equivalent to statement (i). 

Assume $w \in Q_\varphi H_P(\Gamma)$, then there exists $u \in H_P(\Gamma)$ such that $w = Q_\varphi u$. Recall that the action of $R_\Gamma$ on $H_P(\Gamma)$ factors through $\overline{R}_\Gamma$. So, by Proposition \ref{prop-def-Q-varphi}, $Q_\varphi w = Q_\varphi^2 u = Q_\varphi u =w$. Then, by Corollary \ref{coro-multiply-by-Q-varphi-2}, for any $g \in R_\Gamma$, 
\[
g \cdot w = g \cdot (Q_\varphi w) =(g Q_\varphi) \cdot w = (g|_{\mathbb{X}_e=\varphi(e),~\forall e \in E(\Gamma)}) \cdot Q_\varphi w = (g|_{\mathbb{X}_e=\varphi(e),~\forall e \in E(\Gamma)}) \cdot w.
\]
So (i) implies (ii).

Now assume (ii) is true. Then $Q_\varphi w = (Q_\varphi|_{\mathbb{X}_e=\varphi(e),~\forall e \in E(\Gamma)}) \cdot w =w$. So $w \in Q_\varphi H_P(\Gamma)$. It then follows that (i) and (ii) are equivalent.
\end{proof}

\begin{lemma}\label{lemma-decomp-circle}
Denote by $\bigcirc_m$ a circle colored by $m$. Then $Q_\varphi H_P(\bigcirc_m)$ is $1$-dimensional for each $\varphi \in \mathcal{S} (\bigcirc_m)$ and $H_P(\bigcirc_m) = \bigoplus_{\varphi \in \mathcal{S} (\bigcirc_m)} Q_\varphi H_P(\bigcirc_m)$.
\end{lemma}

\begin{proof}
Put a single marked point on $\bigcirc_m$ and associate to it the alphabet $\mathbb{X}$. By Proposition \ref{prop-circle-module}, as $\Sym(\mathbb{X})$-modules, $H_P(\bigcirc_m) \cong \overline{R}_{\bigcirc_m}\{q^{-m(N-m)}\} \left\langle m \right\rangle$. Then the lemma follows from Theorem \ref{thm-bar-R-Gamma-decompose}.
\end{proof}

Note that the decomposition in Lemma \ref{lemma-decomp-circle} does not preserve the quantum filtration.

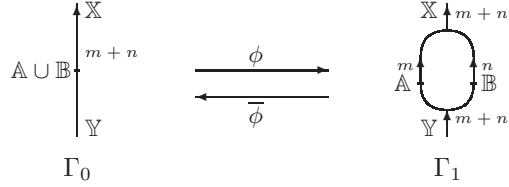
\begin{figure}[ht]

\setlength{\unitlength}{1pt}

\begin{picture}(360,75)(-180,-90)


\put(-67,-45){\tiny{$m+n$}}

\put(-70,-75){\vector(0,1){50}}

\put(-71,-50){\line(1,0){2}}

\put(-67,-30){\small{$\mathbb{X}$}}

\put(-95,-53){\small{$\mathbb{A}\cup\mathbb{B}$}}

\put(-67,-75){\small{$\mathbb{Y}$}}

\put(-75,-90){$\Gamma_0$}


\put(-25,-50){\vector(1,0){50}}

\put(25,-60){\vector(-1,0){50}}

\put(-5,-47){\small{$\phi$}}

\put(-5,-70){\small{$\overline{\phi}$}}


\put(70,-75){\vector(0,1){10}}

\put(70,-35){\vector(0,1){10}}

\qbezier(70,-65)(60,-65)(60,-55)

\qbezier(70,-35)(60,-35)(60,-45)

\put(60,-55){\vector(0,1){10}}

\put(59,-55){\line(1,0){2}}

\qbezier(70,-65)(80,-65)(80,-55)

\qbezier(70,-35)(80,-35)(80,-45)

\put(80,-55){\vector(0,1){10}}

\put(79,-55){\line(1,0){2}}

\put(73,-30){\tiny{$m+n$}}

\put(73,-70){\tiny{$m+n$}}

\put(83,-50){\tiny{$n$}}

\put(51,-50){\tiny{$m$}}

\put(60,-30){\small{$\mathbb{X}$}}

\put(60,-75){\small{$\mathbb{Y}$}}

\put(50,-58){\small{$\mathbb{A}$}}

\put(83,-58){\small{$\mathbb{B}$}}

\put(65,-90){$\Gamma_1$}

\end{picture}

\caption{Edge splitting and merging}\label{edge-splitting}

\end{figure}

Let $\Gamma_0$ and $\Gamma_1$ be closed embedded MOY graphs that are identical outside the part shown in Figure \ref{edge-splitting}. We call the local change $\Gamma_0 \leadsto \Gamma_1$ an edge splitting. It induces a homomorphism $\phi:H_P(\Gamma_0) \rightarrow H_P(\Gamma_1)$. We call the local change $\Gamma_1 \leadsto \Gamma_0$ an edge merging. It induces a homomorphism $\overline{\phi}:H_P(\Gamma_1) \rightarrow H_P(\Gamma_0)$. See \cite[Section 4]{Wu-color-equi} for more details.

\begin{lemma}\cite[Lemma 7.11]{Wu-color}\label{phibar-compose-phi}
Let $\Gamma_0$ and $\Gamma_1$ be the MOY graphs in Figure \ref{edge-splitting}. For $\lambda=(\lambda_1\geq\cdots\geq\lambda_m) \in \Lambda_{m,n}$, define $\lambda^c=(n-\lambda_m\geq\cdots\geq n-\lambda_1) \in \Lambda_{m,n}$. Then, for $\lambda,\mu\in \Lambda_{m,n}$,
\begin{equation}\label{phibar-compose-phi-eq}
\overline{\phi} \circ \mathfrak{m}(S_{\lambda}(\mathbb{A})\cdot S_{\mu}(-\mathbb{B})) \circ \phi \approx 
\begin{cases}
    \id_{C_P(\Gamma_0)} & \text{if } \mu=\lambda^c, \\ 
    0 & \text{otherwise.}
\end{cases}
\end{equation}
Moreover,
\begin{equation}\label{phi-compose-phibar}
\id_{C_P(\Gamma_1)} \approx \sum_{\mu\in \Lambda_{m,n}} \mathfrak{m}(S_{\mu^c}(\mathbb{A})) \circ \phi \circ \overline{\phi} \circ \mathfrak{m}(S_\mu (-\mathbb{B})).
\end{equation}
Here, ``$\approx$" means that the two morphisms are equal up to homotopy and scaling by a non-zero complex number.
\end{lemma}

\begin{proof}
This follows from the explicit descriptions of $\phi$ and $\overline{\phi}$ in \cite[Subsection 7.4]{Wu-color} and \cite[Proposition \textit{GR}5]{Lascoux-notes}.
\end{proof}

\begin{theorem}\label{thm-decomp-MOY}
Let $\Gamma$ be a closed abstract MOY graph. Then, for any state $\varphi$ of $\Gamma$, $Q_\varphi H_P(\Gamma) \neq 0$. In particular, if $\Gamma$ is a closed embedded MOY graph, then $Q_\varphi H_P(\Gamma)$ is $1$-dimensional for each $\varphi \in \mathcal{S} (\Gamma)$ and $H_P(\Gamma) = \bigoplus_{\varphi \in \mathcal{S} (\Gamma)} Q_\varphi H_P(\Gamma)$, where the decomposition does not preserve the quantum filtration.
\end{theorem}

\begin{proof}
Let $\Gamma$ be a closed abstract MOY graph, and $\varphi$ a state of $\Gamma$. By \cite[Lemma 3.10]{Wu-color-equi}, we assume that $\Gamma$ is trivalent. We prove that $Q_\varphi H_P(\Gamma) \neq 0$ by an induction on the highest color appearing in $\Gamma$. 

If the highest color appearing in $\Gamma$ is $1$, then $\Gamma$ is a collection of circles colored by $1$. So $Q_\varphi H_P(\Gamma)\neq 0$ by Lemma \ref{lemma-decomp-circle}. Now assume that $Q_\varphi H_P(\Gamma)\neq 0$ for all $\varphi \in \mathcal{S} (\Gamma)$ if the colors of edges of $\Gamma$ are all less than $m+1$, where $1\leq m\leq N-1$. We claim that $Q_\varphi H_P(\Gamma)\neq 0$ for all $\varphi \in \mathcal{S} (\Gamma)$ if the colors of edges of $\Gamma$ are not greater than $m+1$. We do this by inducting on the number of edges in $\Gamma$ of color $m+1$.

If there are $0$ edges in $\Gamma$ of color $m+1$, then $Q_\varphi H_P(\Gamma)\neq 0$ for all $\varphi \in \mathcal{S} (\Gamma)$ by the induction assumption. Now assume that $Q_\varphi H_P(\Gamma)\neq 0$ for all $\varphi \in \mathcal{S} (\Gamma)$ if the colors appearing in $\Gamma$ are not greater than $m+1$ and there are no more than $k~(\geq 0)$ edges of color $m+1$. Assume that $\Gamma$ is a closed abstract MOY graph colored by $1,\dots m+1$ and have exactly $k+1$ edges of color $m+1$.

If $\Gamma$ contains a circle $\bigcirc_{m+1}$ colored by $m+1$, then 
\[
H_P(\Gamma) = H_P(\Gamma \setminus \bigcirc_{m+1}) \otimes_\C H_P(\bigcirc_{m+1}).
\]
Note that $\Gamma \setminus \bigcirc_{m+1}$ has at most $k$ edges of color $m+1$. Therefore, by induction assumption and Lemma \ref{lemma-decomp-circle}, we know that $Q_\varphi H_P(\Gamma)\neq 0$ for all $\varphi \in \mathcal{S} (\Gamma)$.

\begin{figure}[ht]
\[
\xymatrix{
\input{edge-m+1-1}
}
\]
\caption{}\label{thm-decomp-MOY-proof-fig-1} 

\end{figure}

If $\Gamma$ contains no circle colored by $m+1$, then it contains an edge $e$ of color $m+1$ with a neighborhood of the form given in Figure \ref{thm-decomp-MOY-proof-fig-1}. By admissibility, we get that $\varphi(e)=\varphi(e_1)\cup \varphi(e_2) = \varphi(e_3)\cup \varphi(e_4)$. So either $\varphi(e_1)\cap \varphi(e_3) \neq \emptyset$ or $\varphi(e_1)\cap \varphi(e_4) \neq \emptyset$.

\begin{figure}[ht]
\[
\xymatrix{
\input{edge-m+1-1} \ar@<1ex>[r]^{\phi}& \ar@<1ex>[l]^{\overline{\phi}} \input{edge-m+1-2} \ar@<1ex>[r]^{\chi^1 \otimes \chi^1} & \ar@<1ex>[l]^{\chi^0 \otimes \chi^0} \input{edge-m+1-3}
}
\]
\caption{}\label{thm-decomp-MOY-proof-fig-2} 

\end{figure}

If $\varphi(e_1)\cap \varphi(e_3) \neq \emptyset$, choose an $r \in \varphi(e_1)\cap \varphi(e_3)$. Then, by admissibility, $r \notin \varphi(e_2)\cup \varphi(e_4)$. Consider the morphisms in Figure \ref{thm-decomp-MOY-proof-fig-2}, where $\hat{\Gamma}$ and $\tilde{\Gamma}$ are closed abstract MOY graphs identical to $\Gamma$ outside the part shown in Figure \ref{thm-decomp-MOY-proof-fig-2}, $\phi$ and $\overline{\phi}$ are the morphisms induced by the apparent edge splitting and merging, and $\chi^1 \otimes \chi^1$, $\chi^0 \otimes \chi^0$ are the apparent $\chi$-morphisms. (See \cite[Section 4]{Wu-color-equi}.) Denote by $\hat{\varphi}$ and $\tilde{\varphi}$ the unique states of $\hat{\Gamma}$ and $\tilde{\Gamma}$ that agree with $\varphi$ on all the unchanged edges of $\Gamma$, including $e_1,e_2,e_3,e_4$, and take value $\{r\}$ on $e_5$. Note that $\mathbb{X}_{e_5}=\{x\}$ is an alphabet of a single indeterminate.

Note that $\tilde{\Gamma}$ has at most $k$ edges of color $m+1$. So $Q_{\tilde{\varphi}} H_P(\tilde{\Gamma})\neq 0$ by induction hypothesis. Pick a $\tilde{v} \in Q_{\tilde{\varphi}} H_P(\tilde{\Gamma})$ with $\tilde{v} \neq 0$. Define $\hat{v} \in H_P(\hat{\Gamma})$ by $\hat{v} = (\chi^0 \otimes \chi^0) (\tilde{v})$. Note that $\chi^0 \otimes \chi^0$ is linear with respect to $\Sym(\mathbb{X}_{e_1}|\cdots |\mathbb{X}_{e_6})$ and with respect to symmetric polynomials in alphabets assigned to edges outside the part shown in Figure \ref{thm-decomp-MOY-proof-fig-2}. By Lemma \ref{lemma-MOY-homology-multiplication}, one can see that $\hat{v} \in Q_{\hat{\varphi}} H_P(\hat{\Gamma})$. Moreover, by Lemma \ref{lemma-MOY-homology-multiplication} and \cite[Proposition 4.12 and Lemma 4.13]{Wu-color-equi}, 
\begin{eqnarray*}
(\chi^1 \otimes \chi^1) (\hat{v}) & = & (\chi^1 \otimes \chi^1)(\chi^0 \otimes \chi^0) (\tilde{v}) \\
& = & h_{m+1-j}(\{x\}-\mathbb{X}_{e_4}) \cdot h_{m+1-i}(\{x\}-\mathbb{X}_{e_2})\cdot \tilde{v} \\
& = & h_{m+1-j}(\{r\}-\varphi(e_4)) \cdot h_{m+1-i}(\{r\}-\varphi(e_2))\cdot \tilde{v} \\
& = & \left(\prod_{s \in \varphi(e_2)\cup \varphi(e_4)} (r-s) \right) \cdot \tilde{v} \neq 0.
\end{eqnarray*}
So $\hat{v} \neq 0$. 

Now let $v = \overline{\phi}(\hat{v})$. Note that $\overline{\phi}$ is linear with respect to $\Sym(\mathbb{X}_{e_1}|\cdots |\mathbb{X}_{e_4})$ and with respect to symmetric polynomials in alphabets assigned to edges outside the part shown in Figure \ref{thm-decomp-MOY-proof-fig-2}. By Lemma \ref{lemma-MOY-homology-multiplication}, one can see that $v \in Q_{\varphi} H_P(\Gamma)$. Moreover, by Lemmas \ref{lemma-MOY-homology-multiplication} and \ref{phibar-compose-phi}, we know that
\begin{eqnarray*}
\hat{v} & = & c \cdot \sum_{l=0}^m x^{m-l}\cdot \phi \circ \overline{\phi}((-1)^l X_{e_6,l}\cdot \hat{v}) \\
& = & c \cdot \sum_{l=0}^m x^{m-l}\cdot \phi \circ \overline{\phi}((-1)^l (X_{e_6,l}|_{\mathbb{X}_{e_6}=(\varphi(e_1)\cup\varphi(e_2))\setminus \{r\}})  \cdot \hat{v}) \\
& = & c \cdot \sum_{l=0}^m (-1)^l (X_{e_6,l}|_{\mathbb{X}_{e_6}=(\varphi(e_1)\cup\varphi(e_2))\setminus \{r\}})  \cdot x^{m-l}\cdot \phi \circ \overline{\phi}(\hat{v}) \\
& = & c \cdot (\sum_{l=0}^m (-1)^l (X_{e_6,l}|_{\mathbb{X}_{e_6}=(\varphi(e_1)\cup\varphi(e_2))\setminus \{r\}})  \cdot x^{m-l})\cdot \phi (v),
\end{eqnarray*}
where $c\in \C\setminus \{0\}$. This implies that $v\neq 0$ and, therefore, $Q_\varphi H_P(\Gamma) \neq 0$.

\begin{figure}[ht]
\[
\xymatrix{
\input{edge-m+1-1} \ar@{=}[r] & \input{edge-m+1-4}
}
\]
\caption{}\label{thm-decomp-MOY-proof-fig-3} 

\end{figure}

If $\varphi(e_1)\cap \varphi(e_4) \neq \emptyset$, we change the immersion of $\Gamma$ by the move in Figure \ref{thm-decomp-MOY-proof-fig-3}. Then the proof reduces to the case we have just dealt with.

Thus, we have shown that $Q_\varphi H_P(\Gamma)\neq 0$ for all $\varphi \in \mathcal{S} (\Gamma)$ if the colors appearing in $\Gamma$ are not greater than $m+1$. This completes the induction. Therefore, $Q_\varphi H_P(\Gamma)\neq 0$ for any state $\varphi$ of any closed abstract MOY graph $\Gamma$.

Next, let $\Gamma$ be a closed embedded MOY graph. Then by \cite[Proposition 9.8]{Wu-color-equi}, \cite[Theorem 14.7]{Wu-color} and the definition of the MOY graph polynomial $\left\langle \Gamma \right\rangle_N$ in \cite{MOY}, we know that 
\begin{equation}\label{thm-decomp-MOY-proof-eq-1} 
\dim_\C H_P(\Gamma) = \left\langle \Gamma \right\rangle_N|_{q=1} = |\mathcal{S}(\Gamma)|.
\end{equation} 
On the other hand, by Theorem \ref{thm-bar-R-Gamma-decompose}, we know $\overline{R}_{\Gamma} = \bigoplus_{\varphi \in \mathcal{S}(\Gamma)} \C \cdot Q_\varphi$, which implies that $H_P(\Gamma) = \sum_{\varphi \in \mathcal{S} (\Gamma)} Q_\varphi H_P(\Gamma)$. Using Lemma \ref{lemma-MOY-homology-multiplication} and Parts (ii) and (iii) of Proposition \ref{prop-def-Q-varphi}, it is easy to check that the sum on the right hand side is a direct sum. So 
\[
H_P(\Gamma) = \bigoplus_{\varphi \in \mathcal{S} (\Gamma)} Q_\varphi H_P(\Gamma). 
\]
But $Q_\varphi H_P(\Gamma)\neq 0$ for all $\varphi \in \mathcal{S} (\Gamma)$. So 
\begin{equation}\label{thm-decomp-MOY-proof-eq-2} 
\dim_\C H_P(\Gamma) = \sum_{\varphi \in \mathcal{S} (\Gamma)} \dim_\C Q_\varphi H_P(\Gamma) \geq |\mathcal{S}(\Gamma)|.
\end{equation}
In order for both of Equations \eqref{thm-decomp-MOY-proof-eq-1} and \eqref{thm-decomp-MOY-proof-eq-2} to be true, we must have $\dim_\C Q_\varphi H_P(\Gamma) = 1$ for all $\varphi \in \mathcal{S} (\Gamma)$. This completes the proof.
\end{proof}

\section{Morphisms Induced by Local Changes of MOY Graphs}\label{sec-morphisms}

\begin{definition}\label{def-basis-graph-homology}
Let $\Gamma$ be a closed embedded MOY graph, and $P(X)$ a generic polynomial of the form \eqref{def-P}. By Theorem \ref{thm-decomp-MOY}, for each $\varphi \in \mathcal{S}(\Gamma)$, $Q_\varphi H_P(\Gamma)$ is spanned by a single non-vanishing element $v_{\Gamma,\varphi} \in Q_\varphi H_P(\Gamma)$, where $v_{\Gamma,\varphi}$ is unique up to scaling. Moreover, $\{v_{\Gamma,\varphi}~|~\varphi \in \mathcal{S}(\Gamma)\}$ is a $\C$-linear basis for $H_P(\Gamma)$. This basis is well defined up to scaling each $v_{\Gamma,\varphi}$.

In particular, the decomposition in Theorem \ref{thm-decomp-MOY} can be written as 
\[
H_P(\Gamma) = \bigoplus_{\varphi\in \mathcal{S}(\Gamma)} \C \cdot v_{\Gamma,\varphi}.
\]
\end{definition}

In this section, we study effects of morphisms induced by local changes of MOY graphs on the above basis. Again, we fix a generic polynomial $P(X)$ of the form \eqref{def-P} and denote by $\Sigma=\Sigma(P)=\{r_1,\dots,r_N\}$ the set of roots of $P'(X)$.

\subsection{Compatible states} Before going into specific local changes, we first introduce the simple concept of compatible states, which will be useful later on.

\begin{definition}\label{def-sub-MOY-graph}
Let $\Gamma$ be an embedded MOY graph with a marking. Denote by $\{p_1,\dots,p_k\}$ the set of internal marked points on $\Gamma$ and by $\tilde{\Gamma}$ the embedded MOY graph resulted from cutting $\Gamma$ at $p_1,\dots,p_k$. Note that, up to homotopy, there is a natural continuous map $\xi:\tilde{\Gamma} \rightarrow \Gamma$ that is injective except at $p_1,\dots,p_k$.

A MOY subgraph of $\Gamma$ is the image under $\xi$ of the union of several connected components of $\tilde{\Gamma}$. Note that all MOY subgraphs of $\Gamma$ are embedded MOY graphs.

An edge $e$ of $\Gamma$ is said to be outside a MOY subgraph of $\Gamma$ if $e$ is not a subset of that subgraph (even if part of $e$ is contained in it.)
\end{definition}

\begin{definition}\label{def-compatible-states}
Assume that, for $i=1,2$, $\Gamma_i$ is a closed embedded MOY graph containing a MOY subgraph $\hat{\Gamma}_i$ such that $\Gamma_1 \setminus \hat{\Gamma}_1$ and $\Gamma_2 \setminus \hat{\Gamma}_2$ are identical. Let $\varphi_1$ and $\varphi_2$ be states of $\Gamma_1$ and $\Gamma_2$. We say that $\varphi_1$ and $\varphi_2$ are compatible outside $\hat{\Gamma}_1$ and $\hat{\Gamma}_2$ if their values coincide on every pair of corresponding edges in $\Gamma_1$ and $\Gamma_2$ outside $\hat{\Gamma}_1$ and $\hat{\Gamma}_2$.

If $\hat{\Gamma}_1$ and $\hat{\Gamma}_2$ are clear from the context, then we simply say that $\varphi_1$ and $\varphi_2$ are compatible.
\end{definition}

\begin{lemma}\label{lemma-morphism-compatible-states}
Assume $\Gamma_1$, $\Gamma_2$, $\hat{\Gamma}_1$ and $\hat{\Gamma}_2$ are as in Definition \ref{def-compatible-states}, and $\varphi \in \mathcal{S}(\Gamma_1)$. Any morphism $f \in \Hom_{\HMF}(C_P(\hat{\Gamma}_1), C_P(\hat{\Gamma}_2))$ induces a homomorphism $f: H_P(\Gamma_1) \rightarrow H_P(\Gamma_2)$. Then $f(v_{\Gamma_1,\varphi})$ is contained in the subspace of $H_P(\Gamma_2)$ spanned by $\{v_{\Gamma_2,\varphi'}~|~ \varphi' \in \mathcal{S}(\Gamma_2)$ is compatible with $\varphi$ outside $\hat{\Gamma}_2\}$. In particular, if none of the states of $\Gamma_2$ are compatible with $\varphi$ outside $\hat{\Gamma}_2$, then $f(v_{\Gamma_1,\varphi})=0$.
\end{lemma}

\begin{proof}
Note that $f(v_{\Gamma_1,\varphi}) = \sum_{\varphi' \in \mathcal{S}(\Gamma_2)} \alpha_{\varphi'}v_{\Gamma_2,\varphi'}$, where $\alpha_{\varphi'} \in \C$. Let $e_1,\dots,e_l$ be the edges of $\Gamma_1$ outside $\hat{\Gamma}_1$. Mark each $e_j$ and its corresponding edge $e_j'$ in $\Gamma_2$ by $\mathbb{X}_{e_j}$. Note that $f: H_P(\Gamma_1) \rightarrow H_P(\Gamma_2)$ is $\Sym(\mathbb{X}_{e_1}|\cdots|\mathbb{X}_{e_l})$-linear.  Assume a particular $\varphi' \in \mathcal{S}(\Gamma_2)$ is not compatible with $\varphi$ outside $\hat{\Gamma}_2$. Consider
\[
Q = \prod_{j=1}^l \left( \sfrac{\prod_{x \in \mathbb{X}_{e_j},~r\in \Sigma\setminus \varphi'(e_j')} (x-r)}{\prod_{s \in \varphi'(e_j'),~r\in \Sigma\setminus \varphi'(e_j')} (s-r)}\right) \in \Sym(\mathbb{X}_{e_1}|\cdots|\mathbb{X}_{e_l}).
\]
Clearly, 
\begin{eqnarray*}
Q|_{\mathbb{X}_{e_j}=\varphi'(e_j'),~\forall j =1,\dots,l}& = & 1, \\
Q|_{\mathbb{X}_{e_j}=\varphi(e_j),~\forall j =1,\dots,l}& = & 0.
\end{eqnarray*}
By Lemma \ref{lemma-MOY-homology-multiplication}, we have that
\[
Q \cdot f(v_{\Gamma_1,\varphi}) = f(Q \cdot v_{\Gamma_1,\varphi}) = f((Q|_{\mathbb{X}_{e_j}=\varphi(e_j),~\forall j =1,\dots,l}) \cdot v_{\Gamma_1,\varphi}) = 0
\]
and 
\begin{eqnarray*}
Q \cdot f(v_{\Gamma_1,\varphi}) & = & Q \cdot \sum_{\varphi'' \in \mathcal{S}(\Gamma_2)} \alpha_{\varphi''}v_{\Gamma_2,\varphi''} \\
& = & \alpha_{\varphi'}v_{\Gamma_2,\varphi'} + \sum_{\varphi'' \in \mathcal{S}(\Gamma_2), \varphi'' \neq \varphi'}  (Q|_{\mathbb{X}_{e_j}=\varphi''(e_j'),~\forall j =1,\dots,l}) \cdot \alpha_{\varphi''}v_{\Gamma_2,\varphi''}.
\end{eqnarray*}
But $\{v_{\Gamma_2,\varphi'}~|~ \varphi' \in \mathcal{S}(\Gamma_2)\}$ is a basis for $H_P(\Gamma_2)$. So the above two equations imply that $\alpha_{\varphi'}=0$. And the lemma follows.
\end{proof}

\begin{figure}[ht]

\setlength{\unitlength}{1pt}

\begin{picture}(360,100)(-180,-50)


\put(-100,25){$\Gamma_1$:}

\put(-60,10){\vector(0,1){10}}

\put(-60,20){\vector(-1,1){20}}

\put(-60,20){\vector(1,1){10}}

\put(-50,30){\vector(-1,1){10}}

\put(-50,30){\vector(1,1){10}}

\put(-75,3){\tiny{$i+j+k$}}

\put(-55,21){\tiny{$j+k$}}

\put(-80,42){\tiny{$i$}}

\put(-60,42){\tiny{$j$}}

\put(-40,42){\tiny{$k$}}


\put(-15,25){$\longleftrightarrow$}


\put(20,25){$\Gamma'_1$:}

\put(60,10){\vector(0,1){10}}

\put(60,20){\vector(1,1){20}}

\put(60,20){\vector(-1,1){10}}

\put(50,30){\vector(1,1){10}}

\put(50,30){\vector(-1,1){10}}

\put(45,3){\tiny{$i+j+k$}}

\put(38,21){\tiny{$i+j$}}

\put(80,42){\tiny{$k$}}

\put(60,42){\tiny{$j$}}

\put(40,42){\tiny{$i$}}


\put(-100,-25){$\Gamma_2$:}

\put(-60,-30){\vector(0,-1){10}}

\put(-80,-10){\vector(1,-1){20}}

\put(-50,-20){\vector(-1,-1){10}}

\put(-60,-10){\vector(1,-1){10}}

\put(-40,-10){\vector(-1,-1){10}}

\put(-75,-47){\tiny{$i+j+k$}}

\put(-55,-29){\tiny{$j+k$}}

\put(-80,-8){\tiny{$i$}}

\put(-60,-8){\tiny{$j$}}

\put(-40,-8){\tiny{$k$}}


\put(-15,-25){$\longleftrightarrow$}


\put(20,-25){$\Gamma'_2$:}

\put(60,-30){\vector(0,-1){10}}

\put(80,-10){\vector(-1,-1){20}}

\put(50,-20){\vector(1,-1){10}}

\put(60,-10){\vector(-1,-1){10}}

\put(40,-10){\vector(1,-1){10}}

\put(45,-47){\tiny{$i+j+k$}}

\put(38,-29){\tiny{$i+j$}}

\put(80,-8){\tiny{$k$}}

\put(60,-8){\tiny{$j$}}

\put(40,-8){\tiny{$i$}}

\end{picture}

\caption{Bouquet moves}\label{bouquet-move-figure-states}

\end{figure}
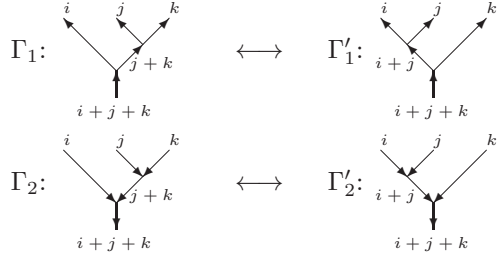

\subsection{Bouquet moves} For $i=1,2$, assume $\Gamma_i$ and $\Gamma_i'$ in Figure \ref{bouquet-move-figure-states} are closed embedded MOY graphs that are identical outside the part changed by the bouquet move. By \cite[Corollary 3.11 and Lemma 4.2]{Wu-color-equi}, the bouquet move induces, up to scaling, an isomorphism $h_i:H_P(\Gamma_i)\xrightarrow{\cong} H_P(\Gamma_i')$ and its inverse $h_i^{-1}:H_P(\Gamma_i')\xrightarrow{\cong} H_P(\Gamma_i)$.

\begin{lemma}\label{lemma-bouquet-states}
If $\varphi$ is a state of $\Gamma_i$, then there is a unique state $\varphi'$ of $\Gamma_i'$ that is compatible with $\varphi$ outside the part changed by the bouquet move. In particular, $h_i(v_{\Gamma_i,\varphi}) = c\cdot v_{\Gamma_i',\varphi'}$, where $c\in \C\setminus \{0\}$.

Similarly, if $\varphi'$ is a state of $\Gamma_i'$, then there is a unique state $\varphi$ of $\Gamma_i$ that is compatible with $\varphi'$ outside the part changed by the bouquet move. In particular, $h_i^{-1}(v_{\Gamma_i',\varphi'}) = c'\cdot v_{\Gamma_i,\varphi}$, where $c'\in \C\setminus \{0\}$.
\end{lemma} 

\begin{proof}
Note that there is a natural one-to-one correspondence between states of $\Gamma_i$ and $\Gamma_i'$ that associate each state to the unique state compatible with it. Then the lemma follows easily from Lemma \ref{lemma-morphism-compatible-states}.
\end{proof}

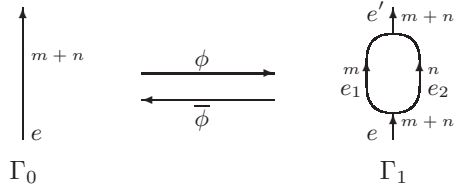
\begin{figure}[ht]

\setlength{\unitlength}{1pt}

\begin{picture}(360,75)(-180,-90)


\put(-67,-45){\tiny{$m+n$}}

\put(-70,-75){\vector(0,1){50}}

\put(-67,-75){\small{$e$}}

\put(-75,-90){$\Gamma_0$}


\put(-25,-50){\vector(1,0){50}}

\put(25,-60){\vector(-1,0){50}}

\put(-5,-47){\small{$\phi$}}

\put(-5,-70){\small{$\overline{\phi}$}}


\put(70,-75){\vector(0,1){10}}

\put(70,-35){\vector(0,1){10}}

\qbezier(70,-65)(60,-65)(60,-55)

\qbezier(70,-35)(60,-35)(60,-45)

\put(60,-55){\vector(0,1){10}}

\qbezier(70,-65)(80,-65)(80,-55)

\qbezier(70,-35)(80,-35)(80,-45)

\put(80,-55){\vector(0,1){10}}

\put(73,-30){\tiny{$m+n$}}

\put(73,-70){\tiny{$m+n$}}

\put(83,-50){\tiny{$n$}}

\put(51,-50){\tiny{$m$}}

\put(60,-30){\small{$e'$}}

\put(60,-75){\small{$e$}}

\put(50,-58){\small{$e_1$}}

\put(83,-58){\small{$e_2$}}

\put(65,-90){$\Gamma_1$}

\end{picture}

\caption{Edge splitting and merging}\label{edge-splitting-state}

\end{figure}

\subsection{Edge splitting and merging} Let $\Gamma_0$ and $\Gamma_1$ be closed embedded MOY graphs that are identical outside the part shown in Figure \ref{edge-splitting-state}. We call the local change $\Gamma_0 \leadsto \Gamma_1$ an edge splitting. It induces a homomorphism $\phi:H_P(\Gamma_0) \rightarrow H_P(\Gamma_1)$. We call the local change $\Gamma_1 \leadsto \Gamma_0$ an edge merging. It induces a homomorphism $\overline{\phi}:H_P(\Gamma_1) \rightarrow H_P(\Gamma_0)$. See \cite[Section 4]{Wu-color-equi} for more details.

\begin{lemma}\label{lemma-edge-splitting-states}
For each state $\varphi'$ of $\Gamma_1$ there is a unique state $\varphi$ of $\Gamma_0$ that is compatible with $\varphi$ outside the changed part. And 
\begin{equation}\label{lemma-edge-splitting-states-eq-1}
\overline{\phi}(v_{\Gamma_1,\varphi'}) = c \cdot v_{\Gamma_0,\varphi},
\end{equation}
where $c \in \C\setminus\{0\}$.

For every state $\varphi$ of $\Gamma_0$, there are exactly $\bn{m+n}{m}$ states of $\Gamma_1$ that are compatible with $\varphi$ outside the changed part. And 
\begin{equation}\label{lemma-edge-splitting-states-eq-2}
\phi(v_{\Gamma_0,\varphi})=\sum_{\varphi' \in \mathcal{S}_{\varphi}(\Gamma_1)} c_{\varphi'}\cdot v_{\Gamma_1,\varphi'},
\end{equation}
where $\mathcal{S}_{\varphi}(\Gamma_1)$ is the set of all states of $\Gamma_1$ compatible with $\varphi$, and $c_{\varphi'}\in \C\setminus\{0\}$ for every $\varphi'\in \mathcal{S}_{\varphi}(\Gamma_1)$.
\end{lemma}

\begin{proof}
The compatibility states $\varphi$ of $\Gamma_0$ and $\varphi'$ of $\Gamma_1$ means that $\varphi$ and $\varphi'$ are identical outside the part shown in Figure \ref{edge-splitting-state} and $\varphi(e)=\varphi'(e)=\varphi'(e_1) \cup \varphi'(e_2) = \varphi'(e')$, which implies the conclusions about compatible states in the lemma. Next we prove the two equations in the lemma. 

We prove Equation \eqref{lemma-edge-splitting-states-eq-1} first. By Lemma \ref{lemma-morphism-compatible-states} and the fact $\varphi$ is the only state of $\Gamma_0$ compatible with $\varphi'$, we only need to show that $\overline{\phi}(v_{\Gamma_1,\varphi'}) \neq 0$. But, by Lemmas \ref{lemma-MOY-homology-multiplication} and \ref{phibar-compose-phi}, we have 
\begin{eqnarray*}
c'\cdot v_{\Gamma_1,\varphi'} & = & \sum_{\mu\in \Lambda_{m,n}} \mathfrak{m}(S_{\mu^c}(\mathbb{X}_{e_1})) \circ \phi \circ \overline{\phi}(S_\mu (-\mathbb{X}_{e_2})\cdot v_{\Gamma_1,\varphi'}) \\
& = & \sum_{\mu\in \Lambda_{m,n}} (S_\mu (-\mathbb{X}_{e_2})|_{\mathbb{X}_{e_2}=\varphi'(e_2)})\cdot  \mathfrak{m}(S_{\mu^c}(\mathbb{X}_{e_1})) \circ \phi \circ \overline{\phi}(v_{\Gamma_1,\varphi'}),
\end{eqnarray*}
where $c'$ is a non-zero scalar and $\mathfrak{m}(\ast)$ is the homomorphism given by multiplying $\ast$. (See Subsection \ref{subsec-sum-mf}.) This shows that $\overline{\phi}(v_{\Gamma_1,\varphi'}) \neq 0$ and proves Equation \eqref{lemma-edge-splitting-states-eq-1}.

Now we prove Equation \eqref{lemma-edge-splitting-states-eq-2}. By Lemma \ref{lemma-morphism-compatible-states}, we know that
\[
\phi(v_{\Gamma_0,\varphi})=\sum_{\varphi' \in \mathcal{S}_{\varphi}(\Gamma_1)} c_{\varphi'}\cdot v_{\Gamma_1,\varphi'}.
\]
It remains to show that $c_{\varphi'}\neq 0$ for every $\varphi'\in \mathcal{S}_{\varphi}(\Gamma_1)$. Assume $c_{\varphi'}= 0$ for a particular $\varphi'\in \mathcal{S}_{\varphi}(\Gamma_1)$. By Equation \eqref{lemma-edge-splitting-states-eq-1} and the computation in the previous paragraph, we know that
\begin{eqnarray*}
&& c'\cdot v_{\Gamma_1,\varphi'} \\
& = & \sum_{\mu\in \Lambda_{m,n}} (S_\mu (-\mathbb{X}_{e_2})|_{\mathbb{X}_{e_2}=\varphi'(e_2)})\cdot  \mathfrak{m}(S_{\mu^c}(\mathbb{X}_{e_1})) \circ \phi \circ \overline{\phi}(v_{\Gamma_1,\varphi'}) \\
& = & c \cdot \sum_{\mu\in \Lambda_{m,n}} (S_\mu (-\mathbb{X}_{e_2})|_{\mathbb{X}_{e_2}=\varphi'(e_2)})\cdot  \mathfrak{m}(S_{\mu^c}(\mathbb{X}_{e_1})) \circ \phi (v_{\Gamma_0,\varphi}) \\
& = & c \cdot \sum_{\mu\in \Lambda_{m,n}} (S_\mu (-\mathbb{X}_{e_2})|_{\mathbb{X}_{e_2}=\varphi'(e_2)})\cdot  \mathfrak{m}(S_{\mu^c}(\mathbb{X}_{e_1}))\cdot(\sum_{\varphi'' \in \mathcal{S}_{\varphi}(\Gamma_1),~ \varphi'' \neq \varphi'} c_{\varphi''}\cdot v_{\Gamma_1,\varphi''}) \\
& = & c \cdot \sum_{\mu\in \Lambda_{m,n}} (S_\mu (-\mathbb{X}_{e_2})|_{\mathbb{X}_{e_2}=\varphi'(e_2)})\cdot\sum_{\varphi'' \in \mathcal{S}_{\varphi}(\Gamma_1),~ \varphi'' \neq \varphi'} (S_{\mu^c}(\mathbb{X}_{e_1})|_{\mathbb{X}_{e_1}=\varphi''(e_1)})\cdot c_{\varphi''}\cdot v_{\Gamma_1,\varphi''} \\
& \in & \bigoplus_{\varphi'' \in \mathcal{S}_{\varphi}(\Gamma_1),~ \varphi'' \neq \varphi'} \C \cdot v_{\Gamma_1,\varphi''},
\end{eqnarray*}
where $c'$ and $c$ are non-zero scalars. This is a contradiction since $\{v_{\Gamma_1,\varphi'}~|~\varphi'\in \mathcal{S}(\Gamma_1)\}$ is a basis for $H_P(\Gamma_1)$. Thus, $c_{\varphi'}\neq 0$ for every $\varphi'\in \mathcal{S}_{\varphi}(\Gamma_1)$.
\end{proof}

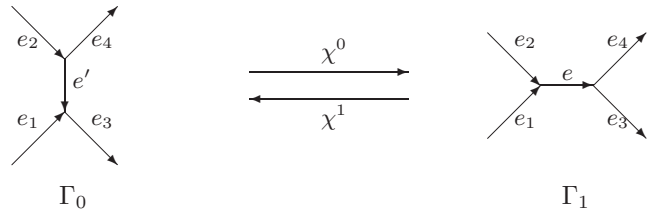
\begin{figure}[ht]

\setlength{\unitlength}{1pt}

\begin{picture}(360,75)(-180,-15)


\put(-120,0){\vector(1,1){20}}

\put(-100,20){\vector(1,-1){20}}

\put(-100,40){\vector(0,-1){20}}

\put(-100,40){\vector(1,1){20}}

\put(-120,60){\vector(1,-1){20}}

\put(-118,45){\small{$e_2$}}

\put(-118,15){\small{$e_1$}}

\put(-97,29){\small{$e'$}}

\put(-90,45){\small{$e_4$}}

\put(-90,15){\small{$e_3$}}

\put(-102,-15){$\Gamma_0$}


\put(-30,35){\vector(1,0){60}}

\put(30,25){\vector(-1,0){60}}

\put(-3,40){\small{$\chi^0$}}

\put(-3,15){\small{$\chi^1$}}


\put(60,10){\vector(1,1){20}}

\put(60,50){\vector(1,-1){20}}

\put(80,30){\vector(1,0){20}}

\put(100,30){\vector(1,1){20}}

\put(100,30){\vector(1,-1){20}}

\put(70,45){\small{$e_2$}}

\put(70,15){\small{$e_1$}}

\put(88,32){\small{$e$}}

\put(105,45){\small{$e_4$}}

\put(105,15){\small{$e_3$}}

\put(88,-15){$\Gamma_1$}

\end{picture}

\caption{$\chi$-morphisms}\label{figure-chi-maps-states}

\end{figure}

\subsection{$\chi$-morphisms} Let $\Gamma_0$ and $\Gamma_1$ be closed embedded MOY graphs that are identical except in the part shown in Figure \ref{figure-chi-maps-states}, where we assume that the color of $e_1$ is less than or equal to the color of $e_3$. By \cite[Proposition 4.12 and Lemma 4.13]{Wu-color-equi}, the local change $\Gamma_0 \leadsto \Gamma_1$ induces a homomorphism $\chi^0:H_P(\Gamma_0) \rightarrow H_P(\Gamma_1)$ and the local change $\Gamma_1 \leadsto \Gamma_0$ induces a homomorphism $\chi^1:H_P(\Gamma_1) \rightarrow H_P(\Gamma_0)$. See \cite[Section 4]{Wu-color-equi} for more details.

The goal of this subsection is to understand the actions of $\chi^0$ and $\chi^1$ on the basis of $H_P$. We need following technical results to do this.

\begin{lemma}\label{lemma-schur-vanish}
Suppose that $\mathbb{X}$ is an alphabet of $a$ indeterminates and $\mathbb{Y}$ is an alphabet of $b$ indeterminates. Assume $m>a$ and $n>b$. Then $S_{\lambda_{m,n}}(\mathbb{X}-\mathbb{Y}) = 0$.
\end{lemma}

\begin{proof}
Recall that
\begin{eqnarray*}
S_{\lambda_{m,n}}(\mathbb{X}-\mathbb{Y}) & = & \det (h_{n-i+j}(\mathbb{X}-\mathbb{Y}))_{1 \leq i,j \leq m} \\
& = & \left|%
\begin{array}{lcl}
  h_{n}(\mathbb{X}-\mathbb{Y}) & \cdots & h_{n-1+m}(\mathbb{X}-\mathbb{Y}) \\
  h_{n-1}(\mathbb{X}-\mathbb{Y}) & \cdots & h_{n-2+m}(\mathbb{X}-\mathbb{Y}) \\
  \cdots & \cdots & \cdots \\
  h_{n-m+1}(\mathbb{X}-\mathbb{Y}) & \cdots & h_{n}(\mathbb{X}-\mathbb{Y})
\end{array}%
\right|.
\end{eqnarray*}
Denote by $X_i$ and $Y_i$ the $i$-th elementary symmetric polynomials in $\mathbb{X}$ and $\mathbb{Y}$. For $k>b$, we have
\begin{eqnarray*}
h_k(\mathbb{X}-\mathbb{Y}) & = & \sum_{i=0}^b (-1)^iY_i h_{k-i}(\mathbb{X}) = \sum_{i=0}^b \sum_{j=1}^a(-1)^{i+j-1} X_jY_i h_{k-i-j}(\mathbb{X}) \\
& = & \sum_{j=1}^a (-1)^{j-1}X_j \sum_{i=0}^b (-1)^i Y_i h_{k-i-j}(\mathbb{X}) \\
& = & \sum_{j=1}^a (-1)^{j-1}X_j h_{k-j}(\mathbb{X}-\mathbb{Y}).
\end{eqnarray*}
Hence, the first row of the the above determinant is a combination of the next $a$ rows. This implies that $S_{\lambda_{m,n}}(\mathbb{X}-\mathbb{Y}) = 0$.
\end{proof}

\begin{corollary}\label{cor-schur-value-nonvanish}
Suppose that $\Omega_1, \Omega_2 \subset \C$, $|\Omega_1|=m$, $|\Omega_2|=n$ and $\Omega_1 \cap \Omega_2 =\emptyset$. Let $\mathbb{X}$ be an alphabet of $m$ indeterminates. Then 
\[
S_{\lambda_{m,n}}(\mathbb{X}-\Omega_2) = S_{\lambda_{m,n}}(\Omega_1-\Omega_2) \cdot \sfrac{\prod_{x\in \mathbb{X},~r\in\Omega_2}(x-r)}{\prod_{s\in\Omega_1,~r\in\Omega_2}(s-r)}.
\]
In particular, $S_{\lambda_{m,n}}(\Omega_1-\Omega_2) \neq 0$.
\end{corollary}

\begin{proof}
Note that $S_{\lambda_{m,n}}(\mathbb{X}-\Omega_2)$ is a non-vanishing symmetric polynomial in $\mathbb{X}$ of total degree $2mn$ and all partial degrees of $S_{\lambda_{m,n}}(\mathbb{X}-\Omega_2)$ are not greater than $2n$. (Recall that our degree is twice the usual polynomial degree.) So, by Theorem \ref{Chen-Louck-interpolation}, 
\begin{eqnarray*}
&& S_{\lambda_{m,n}}(\mathbb{X}-\Omega_2) =\\
& & \sum_{\Omega \subset \Omega_1 \cup \Omega_2,~|\Omega|=m} S_{\lambda_{m,n}}(\Omega-\Omega_2) \sfrac{\prod_{x \in \mathbb{X}, ~r \in (\Omega_1 \cup \Omega_2) \setminus \Omega} (x-r)}{\prod_{s \in \Omega, ~r \in (\Omega_1 \cup \Omega_2) \setminus \Omega} (s-r)}.
\end{eqnarray*}
But $S_{\lambda_{m,n}}(\Omega-\Omega_2) = S_{\lambda_{m,n}}((\Omega\setminus\Omega_2)-(\Omega_2\setminus\Omega))$. By Lemma \ref{lemma-schur-vanish}, this implies that $S_{\lambda_{m,n}}(\Omega-\Omega_2) = 0$ unless $\Omega=\Omega_1$. And the corollary follows.
\end{proof}

\begin{lemma}\label{chi-maps-action-states}
\begin{enumerate}[(i)]
	\item Let $\varphi$ be a state of $\Gamma_1$. 
\begin{itemize}
	\item If $\varphi(e_1)\cap\varphi(e_4) \neq \emptyset$, then there is no state of $\Gamma_0$ compatible with $\varphi$ and $\chi^1(v_{\Gamma_1,\varphi})=0$.
	\item If $\varphi(e_1)\cap\varphi(e_4) = \emptyset$, then there is a unique state $\varphi'$ of $\Gamma_0$ compatible with $\varphi$ and $\chi^1(v_{\Gamma_1,\varphi})=c\cdot v_{\Gamma_0,\varphi'}$, where $c \in \C\setminus \{0\}$.
\end{itemize}
  \item Let $\varphi'$ be a state of $\Gamma_0$. 
\begin{itemize}
	\item If $\varphi'(e_1)\cap\varphi'(e_4) \neq \emptyset$, then there is no state of $\Gamma_1$ compatible with $\varphi'$ and $\chi^0(v_{\Gamma_0,\varphi'})=0$.
	\item If $\varphi'(e_1)\cap\varphi'(e_4) = \emptyset$, then there is a unique state $\varphi$ of $\Gamma_1$ compatible with $\varphi'$ and $\chi^0(v_{\Gamma_0,\varphi'})=c'\cdot v_{\Gamma_1,\varphi}$, where $c' \in \C\setminus \{0\}$.
\end{itemize}
\end{enumerate}
\end{lemma}

\begin{proof}
Denote by $m$ and color of $e_1$ and by $n$ the color of $e_4$.

We prove part (i) first. Assume $\varphi(e_1)\cap\varphi(e_4) \neq \emptyset$ and there is a state $\varphi'$ of $\Gamma_0$ compatible with $\varphi$. Then $\varphi'(e_i)=\varphi(e_i)$ for $i=1,\dots,4$. So $\varphi'(e_1)\cap\varphi'(e_2)=\varphi(e_1)\cap\varphi(e_2) = \emptyset$. On the other hand, we have $\varphi(e_1)\cap\varphi(e_4)=\varphi'(e_1)\cap\varphi'(e_4)\subset \varphi'(e_4) \subset \varphi'(e_2)=\varphi(e_2)$. So $\emptyset \neq \varphi(e_1)\cap\varphi(e_4) \subset \varphi(e_1)\cap\varphi(e_2)$. This is a contradiction. So there is no state of $\Gamma_0$ compatible with $\varphi$. Thus, by Lemma \ref{lemma-morphism-compatible-states}, $\chi^1(v_{\Gamma_1,\varphi})=0$.

Assume $\varphi(e_1)\cap\varphi(e_4) = \emptyset$. Using $\varphi(e_1)\cup\varphi(e_2)=\varphi(e)=\varphi(e_3)\cup\varphi(e_4)$, we have that $\varphi(e_1)\subset\varphi(e_3)$, $\varphi(e_4)\subset\varphi(e_2)$ and $\varphi(e_3)\setminus\varphi(e_1)=\varphi(e_2)\setminus\varphi(e_4)$. Then it is easy to see that the only state $\varphi'$ of $\Gamma_0$ compatible with $\varphi$ is the unique one satisfying $\varphi'(e_i)=\varphi(e_i)$ for $i=1,\dots,4$, $\varphi'(e')=\varphi(e_3)\setminus\varphi(e_1)=\varphi(e_2)\setminus\varphi(e_4)$ and that $\varphi'$ agrees with $\varphi$ outside the part shown in Figure \ref{figure-chi-maps-states}. By Lemma \ref{lemma-morphism-compatible-states}, $\chi^1(v_{\Gamma_1,\varphi})=c\cdot v_{\Gamma_0,\varphi'}$. It remains to show that $c\neq 0$. By \cite[Proposition 4.12 and Lemma 4.13]{Wu-color-equi} and Lemma \ref{lemma-MOY-homology-multiplication}, we know that
\[
\chi^0 \circ \chi^1 (v_{\Gamma_1,\varphi}) = S_{\lambda_{m,n}}(\mathbb{X}_{e_1}-\mathbb{X}_{e_4}) \cdot v_{\Gamma_1,\varphi}= S_{\lambda_{m,n}}(\varphi(e_1)-\varphi(e_4)) \cdot v_{\Gamma_1,\varphi}.
\]
By Corollary \ref{cor-schur-value-nonvanish}, $S_{\lambda_{m,n}}(\varphi(e_1)-\varphi(e_4)) \neq 0$. So $\chi^0 \circ \chi^1 (v_{\Gamma_1,\varphi}) \neq 0$. This implies that $c \neq 0$.

Now we prove Part (ii). Assume $\varphi'(e_1)\cap\varphi'(e_4) \neq \emptyset$. Then $\varphi'(e_1)\cap\varphi'(e_2) \neq \emptyset$. This implies that no state of $\Gamma_1$ is compatible with $\varphi'$ and, by Lemma \ref{lemma-morphism-compatible-states}, $\chi^0(v_{\Gamma_0,\varphi'})=0$.

Assume $\varphi'(e_1)\cap\varphi'(e_4) = \emptyset$. Note that $\varphi'(e_1)\cap\varphi'(e') = \varphi'(e')\cap\varphi'(e_4) =\emptyset$. So 
\begin{eqnarray*}
\varphi'(e_1)\cap\varphi'(e_2) & = & \varphi'(e_1)\cap (\varphi'(e')\cup\varphi'(e_4)) =\emptyset, \\
\varphi'(e_4)\cap\varphi'(e_3) & = & \varphi'(e_4)\cap (\varphi'(e')\cup\varphi'(e_1)) =\emptyset, \\
\varphi'(e_1)\cup\varphi'(e_2) & = & \varphi'(e_3)\cup\varphi'(e_4) =\varphi'(e_1)\cup\varphi'(e')\cup\varphi'(e_4).
\end{eqnarray*}
Then it is easy to see that the only state $\varphi$ of $\Gamma_1$ compatible with $\varphi'$ is the unique one satisfying $\varphi(e_i)=\varphi'(e_i)$ for $i=1,\dots,4$, $\varphi(e)=\varphi'(e_1)\cup\varphi'(e')\cup\varphi'(e_4)$ and that $\varphi$ agrees with $\varphi'$ outside the part shown in Figure \ref{figure-chi-maps-states}. By Lemma \ref{lemma-morphism-compatible-states}, $\chi^0(v_{\Gamma_0,\varphi'})=c'\cdot v_{\Gamma_1,\varphi}$. It remains to show that $c'\neq 0$. By \cite[Proposition 4.12 and Lemma 4.13]{Wu-color-equi} and Lemma \ref{lemma-MOY-homology-multiplication}, we know that
\[
\chi^1 \circ \chi^0 (v_{\Gamma_0,\varphi'}) = S_{\lambda_{m,n}}(\mathbb{X}_{e_1}-\mathbb{X}_{e_4}) \cdot v_{\Gamma_0,\varphi'}= S_{\lambda_{m,n}}(\varphi(e_1)-\varphi(e_4)) \cdot v_{\Gamma_0,\varphi'}.
\]
By Corollary \ref{cor-schur-value-nonvanish}, $S_{\lambda_{m,n}}(\varphi(e_1)-\varphi(e_4)) \neq 0$. So $\chi^1 \circ \chi^0 (v_{\Gamma_0,\varphi'}) \neq 0$. This implies that $c' \neq 0$.
\end{proof}

\subsection{Circle creation and annihilation} Let $\Gamma$ be a closed embedded MOY graph and $\widetilde{\Gamma}= \Gamma \sqcup \bigcirc_m$, where $\bigcirc_m$ is a circle colored by $m$. We call the local change $\Gamma \leadsto \widetilde{\Gamma}$ a circle creation. It induces a homomorphism $\iota: H_P(\Gamma) \rightarrow H_P(\widetilde{\Gamma})$. We call the local change $\widetilde{\Gamma} \leadsto \Gamma$ a circle annihilation. It induces a homomorphism $\epsilon: H_P(\widetilde{\Gamma}) \rightarrow H_P(\Gamma)$. See \cite[Subsection 4.2]{Wu-color-equi} for more details.

\begin{lemma}\label{lemma-states-circle-creation}
\begin{enumerate}
	\item For every state $\varphi$ of $\Gamma$, there are $\bn{N}{m}$ states of $\widetilde{\Gamma}$ compatible with $\varphi$ under the circle creation. Moreover,
\[
\iota(v_{\Gamma,\varphi}) = \sum_{\widetilde{\varphi}\in \mathcal{S}_{\varphi}(\widetilde{\Gamma})} c_{\widetilde{\varphi}} \cdot v_{\widetilde{\Gamma},\widetilde{\varphi}},
\]
where $\mathcal{S}_{\varphi}(\widetilde{\Gamma})$ is the set of all states of $\widetilde{\Gamma}$ compatible with $\varphi$ under the circle creation and $c_{\widetilde{\varphi}} \in\C\setminus\{0\}$ for every $\widetilde{\varphi}\in \mathcal{S}_{\varphi}(\widetilde{\Gamma})$.
  \item For every state $\widetilde{\varphi}$ of $\widetilde{\Gamma}$, there is a unique state $\varphi$ of $\Gamma$ compatible with $\widetilde{\varphi}$ under the circle annihilation. Moreover, $\epsilon(v_{\widetilde{\Gamma},\widetilde{\varphi}}) = c \cdot v_{\Gamma,\varphi}$ for some $c \in\C\setminus\{0\}$.
\end{enumerate}
\end{lemma}

\begin{proof}
We prove Part (1) first. There are exactly $\bn{N}{m}$ states on $\bigcirc_m$. So $|\mathcal{S}_{\varphi}(\widetilde{\Gamma})|=\bn{N}{m}$. Mark $\bigcirc_m$ by a single alphabet $\mathbb{X}$. Recall that, by Proposition \ref{prop-circle-module},
\begin{eqnarray*}
&& H_P(\widetilde{\Gamma}) \\
& \cong & H_P(\Gamma) \otimes_\C H_P(\bigcirc_m) \\
& \cong & H_P(\Gamma) \otimes_\C \sfrac{\Sym(\mathbb{X})}{(h_N(\mathbb{X}-\Sigma),h_{N-1}(\mathbb{X}-\Sigma),\dots,h_{N+1-m}(\mathbb{X}-\Sigma))} \{q^{-m(N-m)}\}.
\end{eqnarray*}
Under this isomorphism, we have $\iota(v_{\Gamma,\varphi}) = v_{\Gamma,\varphi} \otimes 1$. By Theorem \ref{Chen-Louck-interpolation}, we have 
\[
1 = \sum_{\Omega \subset \Sigma, ~|\Omega|=m} \left( \sfrac{\prod_{x \in \mathbb{X},~ r \in \Sigma\setminus\Omega} (x-r)}{\prod_{s \in \Omega,~ r \in \Sigma\setminus\Omega}(s-r)} \right)
\]
in $\Sym(\mathbb{X})$ and therefore in $$H_P(\bigcirc_m) \cong \sfrac{\Sym(\mathbb{X})}{(h_N(\mathbb{X}-\Sigma),h_{N-1}(\mathbb{X}-\Sigma),\dots,h_{N+1-m}(\mathbb{X}-\Sigma))} \{q^{-m(N-m)}\}.$$ So 
\[
\iota(v_{\Gamma,\varphi}) = v_{\Gamma,\varphi} \otimes 1 = \sum_{\Omega \subset \Sigma, ~|\Omega|=m} v_{\Gamma,\varphi} \otimes\left( \sfrac{\prod_{x \in \mathbb{X},~ r \in \Sigma\setminus\Omega} (x-r)}{\prod_{s \in \Omega,~ r \in \Sigma\setminus\Omega}(s-r)} \right).
\]
It is easy to see that the right hand side of the above equation is of the form $\sum_{\widetilde{\varphi}\in \mathcal{S}_{\varphi}(\widetilde{\Gamma})} c_{\widetilde{\varphi}} \cdot v_{\widetilde{\Gamma},\widetilde{\varphi}}$ with $c_{\widetilde{\varphi}} \neq 0$ for each $\widetilde{\varphi}\in \mathcal{S}_{\varphi}(\widetilde{\Gamma})$. This completes the proof of Part (1).

For Part (2), note that given $\widetilde{\varphi}\in \mathcal{S}(\widetilde{\Gamma})$, there are a unique $\varphi\in \mathcal{S}(\Gamma)$ compatible with $\widetilde{\varphi}$ and a unique $\Omega \subset \Sigma$ with $|\Omega|=m$, such that 
\[
v_{\widetilde{\Gamma},\widetilde{\varphi}}=c' \cdot v_{\Gamma,\varphi} \otimes \left( \sfrac{\prod_{x \in \mathbb{X},~ r \in \Sigma\setminus\Omega} (x-r)}{\prod_{s \in \Omega,~ r \in \Sigma\setminus\Omega}(s-r)} \right)
\]
for some $c'\neq 0$. By Corollary \ref{cor-schur-value-nonvanish}, we have that
\[
\sfrac{\prod_{x \in \mathbb{X},~ r \in \Sigma\setminus\Omega} (x-r)}{\prod_{s \in \Omega,~ r \in \Sigma\setminus\Omega}(s-r)} = \sfrac{S_{\lambda_{m,N-m}}(\mathbb{X}-(\Sigma\setminus\Omega))}{S_{\lambda_{m,N-m}}(\Omega-(\Sigma\setminus\Omega))}.
\]
So, using \cite[Corollary 4.6]{Wu-color-equi}, we have
\begin{eqnarray*}
\epsilon(v_{\widetilde{\Gamma},\widetilde{\varphi}}) & =& c'\cdot \epsilon (\left(\sfrac{S_{\lambda_{m,N-m}}(\mathbb{X}-(\Sigma\setminus\Omega))}{S_{\lambda_{m,N-m}}(\Omega-(\Sigma\setminus\Omega))} \right) \cdot \iota (v_{\Gamma,\varphi})) \\
& = & c'\cdot \epsilon (\left(\sfrac{S_{\lambda_{m,N-m}}(\mathbb{X})}{S_{\lambda_{m,N-m}}(\Omega-(\Sigma\setminus\Omega))} \right) \cdot \iota (v_{\Gamma,\varphi})) \\
& =& c \cdot v_{\Gamma,\varphi},
\end{eqnarray*}
where $c\neq 0$.
\end{proof}

\begin{figure}[ht]
\[
\xymatrix{
 \input{saddle-1} \ar@<1ex>[rr]^{\eta}  && \input{saddle-2} \ar@<1ex>[ll]^{\widetilde{\eta}}
}
\]
\caption{}\label{saddle-states-fig-1}

\end{figure}

\subsection{Saddle moves} Let $\Gamma$ and $\widetilde{\Gamma}$ be embedded MOY graphs identical except in the part shown in Figure \ref{saddle-states-fig-1}. Let $\eta:H_P(\Gamma)\rightarrow H_P(\widetilde{\Gamma})$ and $\widetilde{\eta}:H_P(\widetilde{\Gamma}) \rightarrow H_P(\Gamma)$  be the homomorphisms induced by the saddle moves performed along the dotted lines. (See \cite[Section 4]{Wu-color-equi} for more details.) Denote by $\bigcirc_m$ the circle of color $m$ in $\widetilde{\Gamma}$ in Figure \ref{saddle-states-fig-1}.

\begin{lemma}\label{lemma-saddle-circle}
\begin{enumerate}
	\item For a state $\varphi$ of $\Gamma$, there is a unique state $\widetilde{\varphi}$ of $\widetilde{\Gamma}$ compatible with $\varphi$ under $\eta$. Moreover, $\eta(v_{\Gamma,\varphi}) = c \cdot v_{\widetilde{\Gamma},\widetilde{\varphi}}$, where $c\in\C\setminus\{0\}$.
	\item Let $\widetilde{\varphi}$ be a state of $\widetilde{\Gamma}$. If $\widetilde{\varphi}(\widetilde{e}) \neq \widetilde{\varphi}(\bigcirc_m)$, then there is no state of $\Gamma$ compatible with $\widetilde{\varphi}$ under $\widetilde{\eta}$. If $\widetilde{\varphi}(\widetilde{e}) = \widetilde{\varphi}(\bigcirc_m)$, then there is a unique state $\varphi$ of $\Gamma$ compatible with $\widetilde{\varphi}$ under $\widetilde{\eta}$. Moreover, 
	\[
	\widetilde{\eta}(v_{\widetilde{\Gamma},\widetilde{\varphi}}) =  \begin{cases}
	0 & \text{if } \widetilde{\varphi}(\widetilde{e}) \neq \widetilde{\varphi}(\bigcirc_m), \\
	\widetilde{c} \cdot v_{\Gamma,\varphi} & \text{for some } \widetilde{c}\in\C\setminus\{0\} \text{ if } \widetilde{\varphi}(\widetilde{e}) = \widetilde{\varphi}(\bigcirc_m).
	\end{cases}
	\]
\end{enumerate}
\end{lemma}

\begin{proof}
The statements about compatible states are easy to check, which implies that $\eta(v_{\Gamma,\varphi})$ and $\widetilde{\eta}(v_{\widetilde{\Gamma},\widetilde{\varphi}})$ must be of the forms given in the lemma. We only need to check that $c\neq 0$ and $\widetilde{c}\neq 0$ if $\widetilde{\varphi}(\widetilde{e}) = \widetilde{\varphi}(\bigcirc_m)$.

Denote by $\iota$ and $\epsilon$ the homomorphisms induced by the creation and annihilation of $\bigcirc_m$. By \cite[Propositions 4.17 and 4.19]{Wu-color-equi}, we have that $\epsilon \circ \eta =c_1 \cdot \id_{H_P(\Gamma)}$ and $\widetilde{\eta} \circ \iota = c_2 \cdot \id_{H_P(\widetilde{\Gamma})}$, where $c_1,c_2\neq 0$. So $\epsilon \circ \eta (v_{\Gamma,\varphi})= c_1 \cdot v_{\Gamma,\varphi} \neq 0$, which implies that $\eta(v_{\Gamma,\varphi}) \neq 0$ and $c\neq 0$. Next, assuming $\widetilde{\varphi}(\widetilde{e}) = \widetilde{\varphi}(\bigcirc_m)$, we prove that $\widetilde{\eta}(v_{\widetilde{\Gamma},\widetilde{\varphi}}) \neq 0$, which implies that $\widetilde{c}\neq 0$. Denote by $\varphi$ the state of $\Gamma$ compatible with $\widetilde{\varphi}$ under $\widetilde{\eta}$ and by $\mathcal{S}_\varphi(\widetilde{\Gamma})$ the set of states of $\widetilde{\Gamma}$ that are compatible to $\varphi$ under $\iota$ (that is, identical to $\widetilde{\varphi}$ outside $\bigcirc_m$.) Then $\widetilde{\eta} \circ \iota(v_{\Gamma,\varphi}) = c_2 \cdot v_{\Gamma,\varphi} \neq 0$. By Lemma \ref{lemma-states-circle-creation}, we have that $\iota(v_{\Gamma,\varphi}) = \sum_{\varphi' \in \mathcal{S}_\varphi(\widetilde{\Gamma})} c_{\varphi'} \cdot v_{\widetilde{\Gamma},\varphi'}$. But for any $\varphi' \in \mathcal{S}_\varphi(\widetilde{\Gamma})\setminus \{\widetilde{\varphi}\}$, there is no state of $\Gamma$ compatible with $\varphi'$ under $\widetilde{\eta}$. So $\widetilde{\eta}(v_{\widetilde{\Gamma},\varphi'})=0$. This implies that $c_{\widetilde{\varphi}} \cdot \widetilde{\eta}(v_{\widetilde{\Gamma},\widetilde{\varphi}}) = \widetilde{\eta} \circ \iota(v_{\Gamma,\varphi}) = c_2 \cdot v_{\Gamma,\varphi} \neq 0$. So $\widetilde{c} \neq 0$.
\end{proof}

\begin{figure}[ht]
\[
\xymatrix{
 \input{saddle-3} \ar[rr]^{\eta}  && \input{saddle-4}
}
\]
\caption{}\label{saddle-states-fig-2}

\end{figure}

\begin{lemma}\label{lemma-saddle-states}
Let $\Gamma_1$ and $\Gamma_2$ be embedded MOY graphs that are identical except in the part shown in Figure \ref{saddle-states-fig-2}. Denote by $\eta:H_P(\Gamma_1) \rightarrow H_P(\Gamma_1)$ the homomorphism induced by the saddle move along the dotted line. Suppose that $\varphi$ is a state of $\Gamma_1$.
\begin{enumerate}
	\item If $\varphi(e_1) \neq \varphi(e_2)$, then there is no state of $\Gamma_2$ compatible with $\varphi$ under $\eta$ and therefore $\eta(v_{\Gamma_1,\varphi})=0$.
	\item If $\varphi(e_1) = \varphi(e_2)$, then there is a unique state $\varphi'$ of $\Gamma_2$ compatible with $\varphi$ under $\eta$ and $\eta(v_{\Gamma_1,\varphi})= c \cdot v_{\Gamma_2,\varphi'}$, where $c\in\C\setminus\{0\}$.
\end{enumerate}
\end{lemma}

\begin{proof}
It is easy to verify the statements about compatible states. Most of the lemma follows from this. The only thing left to prove is that $c\neq 0$ in Part (2).

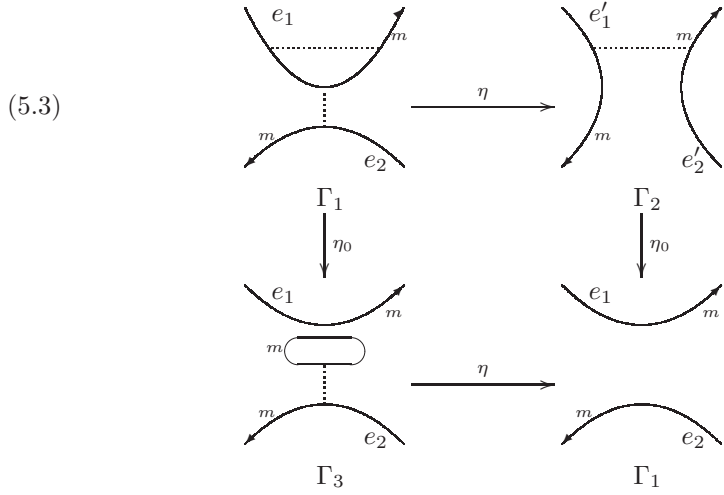
\begin{equation}\label{saddle-trick-diagram}
\xymatrix{
 \setlength{\unitlength}{1pt}
\begin{picture}(60,75)(-30,-15)

\put(30,60){\vector(1,1){0}}

\put(-30,0){\vector(-1,-1){0}}

\qbezier(-30,60)(0,0)(30,60)

\qbezier(-30,0)(0,30)(30,0)

\multiput(0,15)(0,2){7}{\line(0,1){.6}}

\multiput(-20,45)(2,0){21}{\line(1,0){.6}}

\put(-20,55){$e_1$}

\put(15,0){$e_2$}

\put(-25,10){\tiny{$m$}}

\put(25,48){\tiny{$m$}}

\put(-3,-15){$\Gamma_1$}
\end{picture} \ar[rr]^{\eta} \ar[d]^{\eta_0} && \setlength{\unitlength}{1pt}
\begin{picture}(60,75)(-30,-15)

\put(30,60){\vector(1,1){0}}

\put(-30,0){\vector(-1,-1){0}}

\qbezier(-30,60)(0,30)(-30,0)

\qbezier(30,0)(0,30)(30,60)

\multiput(-18,45)(2,0){19}{\line(1,0){.6}}

\put(-20,55){$e_1'$}

\put(15,0){$e_2'$}

\put(-18,10){\tiny{$m$}}

\put(13,48){\tiny{$m$}}

\put(-3,-15){$\Gamma_2$}
\end{picture} \ar[d]^{\eta_0} \\
 \input{saddle-6} \ar[rr]^{\eta}  && \setlength{\unitlength}{1pt}
\begin{picture}(60,75)(-30,-15)

\put(30,60){\vector(1,1){0}}

\put(-30,0){\vector(-1,-1){0}}

\qbezier(-30,60)(0,30)(30,60)

\qbezier(-30,0)(0,30)(30,0)

\put(-20,55){$e_1$}

\put(15,0){$e_2$}

\put(-25,10){\tiny{$m$}}

\put(23,48){\tiny{$m$}}

\put(-3,-15){$\Gamma_1$}
\end{picture}
}
\end{equation}

Assume $\varphi(e_1) = \varphi(e_2)$. Consider Diagram \eqref{saddle-trick-diagram}, where $\eta$, as in the lemma, is induced by the saddle move along the vertical dotted line and $\eta_0$ is induced by the saddle move along the horizontal dotted line. At matrix factorization level, $\eta$ and $\eta_0$ commute up to homotopy and scaling since they act on disjoint parts of the MOY graphs. (Note that these morphisms are only defined up to homotopy and scaling.) So, as homomorphisms of homology, $\eta$ and $\eta_0$ commute up to scaling by a non-zero scalar.

There is a unique state $\varphi''$ of $\Gamma_3$ that is compatible with $\varphi$ under $\eta_0$. Using $\varphi(e_1) = \varphi(e_2)$ and Lemma \ref{lemma-saddle-circle}, it is easy to see that $\eta_0(v_{\Gamma_1,\varphi}) = c_1 \cdot v_{\Gamma_3,\varphi''}$ and $\eta(v_{\Gamma_3,\varphi''}) = c_2 \cdot v_{\Gamma_1,\varphi}$, where $c_1,c_2 \neq 0$. So $\eta \circ \eta_0 (v_{\Gamma_1,\varphi}) = c_1c_2 \cdot v_{\Gamma_1,\varphi} \neq 0$ which implies that $\eta_0 \circ \eta (v_{\Gamma_1,\varphi}) \neq 0$ and therefore $\eta (v_{\Gamma_1,\varphi}) \neq 0$. This proves that $c \neq 0$ in Part (2).
\end{proof}

\section{A Basis for Generic Deformations of the Link Homology}\label{sec-basis-link-construction}

We prove Theorem \ref{thm-basis} in this section. As before, we fix a generic polynomial $P(X)$ of the form \eqref{def-P} and denote by $\Sigma=\Sigma(P)=\{r_1,\dots,r_N\}$ the set of roots of $P'(X)$. Throughout this section, $D$ is a closed knotted MOY graph. We will construct a basis for $H_P(D)$, which is a slight generalization of Theorem \ref{thm-basis}.

\subsection{States of a closed knotted MOY graph} Remove all vertices and crossings of $D$. This gives a set $A(D)$ of arcs starting and ending in vertices or crossings of $D$ and containing no vertices or crossings of $D$ in the interiors.

\begin{definition}\label{states-knotted-MOY}
A pre-state of $D$ is a function $\psi:A(D)\rightarrow \mathcal{P}(\Sigma)$. 

Let $v$ be a vertex of $D$ of the form in Figure \ref{def-state-MOY-vertex-fig}. As in Definition \ref{def-state-MOY}, we say that a pre-state $\psi$ of $D$ is admissible at $v$ if 
\begin{enumerate}
	\item $\psi(e_1),\dots,\psi(e_k)$ are pairwise disjoint,
	\item $\psi(e'_1),\dots,\psi(e'_l)$ are pairwise disjoint,
	\item $\psi(e_1)\cup\cdots\cup\psi(e_k)=\psi(e'_1)\cup\dots\cup\psi(e'_l)$.
\end{enumerate}

\begin{figure}[ht]
\[
\xymatrix{
\input{crossing-def-states-+} && \text{or} && \input{crossing-def-states--}
}
\]
\caption{}\label{def-state-knotted-MOY-vertex-fig-2}

\end{figure}

Let $c$ be a crossing of $D$ (depicted by either of the two diagrams in Figure \ref{def-state-knotted-MOY-vertex-fig-2}.) We say that 
\begin{itemize}
	\item a pre-state $\psi$ of $D$ is quasi-admissible at $c$ if $\psi(a_1)\cap\psi(a_2)=\psi(a_3)\cap\psi(a_4)$ and $\psi(a_1)\cup\psi(a_2)=\psi(a_3)\cup\psi(a_4)$,
	\item a pre-state $\psi$ of $D$ is admissible at $c$ if $\psi(a_1)=\psi(a_3)$ and $\psi(a_2)=\psi(a_4)$.
\end{itemize}

A pre-state of $D$ is called a quasi-state if it is admissible at all vertices of $D$ and quasi-admissible at all crossings of $D$. Denote by $\mathcal{QS}(D)$ the set of all quasi-states of $D$.

A pre-state $D$ is called a state if it is admissible at all vertices and crossings of $D$. Denote by $\mathcal{S}(D)$ the set of all states of $D$. Clearly, $\mathcal{S}(D) \subset \mathcal{QS}(D)$.
\end{definition}

Recall that a resolution of $D$ is an embedded MOY graph obtained by replacing (a small neighborhood of) each crossing by an embedded MOY graph of the form in Figure \ref{crossing-res-state-fig}.

\begin{figure}[ht]
\[
\input{crossing-res-state}
\]
\caption{}\label{crossing-res-state-fig}
\end{figure}

\begin{definition}\label{def-resolution-states}
Let $\psi$ be a pre-state of $D$. A resolution of $\psi$ is a pair $(\Gamma,\varphi)$, where $\Gamma$ is a resolution of $D$ and $\varphi$ is a state of $\Gamma$ that agrees with $\psi$ outside the replaced parts of $D$. Denote by $\mathcal{R}(\psi)$ the set of all resolutions of $\psi$.

Clearly, for every resolution $\Gamma$ of $D$ and every state $\varphi$ of $\Gamma$, the pair $(\Gamma,\varphi)$ is a resolution of a unique pre-state of $D$.
\end{definition}

\begin{lemma}\label{lemma-unique-state-res}
Suppose that $\psi$ is a pre-state of $D$.
\begin{enumerate}
  \item $\mathcal{R}(\psi) = \emptyset$ if $\psi$ is not a quasi-state.
	\item $|\mathcal{R}(\psi)| = 1$ if $\psi$ is a state.
\end{enumerate}
\end{lemma}

\begin{proof}
To prove Part (1), assume $\mathcal{R}(\psi) \neq \emptyset$ and let $(\Gamma,\varphi)$ be an element of $\mathcal{R}(\psi)$. Since $\varphi$ is admissible at all vertices of $\Gamma$, $\psi$ must admissible at all vertices of $D$. Now consider the crossings. Assume $c$ is any crossing of $D$ (depicted in Figure \ref{def-state-knotted-MOY-vertex-fig-2}) and is replaced in $\Gamma$ by the MOY graph in Figure \ref{crossing-res-state-fig}. Then by the admissibility of $\varphi$, we have that 
\begin{eqnarray*}
\psi(a_1)\cup \psi(a_2) & = & \varphi(a_1)\cup \varphi(e_1) \cup \varphi(e_2) \\
& = & \varphi(e_4) \cup \varphi(e_2) \\
& = & \varphi(a_4)\cup \varphi(e_3) \cup \varphi(e_2) \\
& = & \psi(a_4)\cup \psi(a_3)
\end{eqnarray*}
Write $U=\psi(a_1)\cap \psi(a_2)$. Then $U\cap \varphi(e_1) = \emptyset$ and therefore $U \subset \varphi(e_2) \subset \psi(a_3)$. Note that $U \subset \psi(a_1) \subset \varphi(e_4) = \psi(a_4)\cup \varphi (e_3)$ and that $\varphi (e_3) \cap U = \emptyset$ since $\varphi (e_3) \cap \varphi(e_2) =\emptyset$ and $U \subset \varphi(e_2)$. So $U \subset \psi(a_4)$. This implies that $\psi(a_1)\cap \psi(a_2) =U \subset \psi(a_3)\cap\psi(a_4)$. Similarly, one can show that $\psi(a_3)\cap\psi(a_4) \subset \psi(a_1)\cap \psi(a_2)$. So $\psi(a_1)\cap \psi(a_2) = \psi(a_3)\cap\psi(a_4)$. This shows that $\psi$ is a quasi-state.

For Part (2), assume $\psi$ is a state. Suppose that $c$ is any crossing of $D$ (depicted in Figure \ref{def-state-knotted-MOY-vertex-fig-2}) and is replaced by the MOY graph in Figure \ref{crossing-res-state-fig}. Then $\psi(a_1)=\psi(a_3)$ and $\psi(a_2)=\psi(a_4)$. In order for there to be a state $\varphi$ of the resolution agreeing with $\psi$ on $a_1,a_2,a_3,a_4$, we must have that 
\begin{itemize}
	\item $k=|\psi(a_2)\setminus \psi(a_1)|$,
	\item $\varphi(e_1)=\psi(a_2)\setminus \psi(a_1)$,
	\item $\varphi(e_2)=\psi(a_2)\cap \psi(a_1)$,
	\item $\varphi(e_3)=\psi(a_1)\setminus \psi(a_2)$,
	\item $\varphi(e_4)=\psi(a_1)\cup \psi(a_2)$. 
\end{itemize}
These conditions uniquely characterize the resolution $(\Gamma,\varphi)$ near $c$. Therefore, $|\mathcal{R}(\psi)| = 1$.
\end{proof}

\subsection{A decomposition of the chain complex $H(C_P(D),d_{mf})$} \label{subsec-decom-chain-quasi-state}
As an object of $\hch(\hmf_{\C,0})$, $C_P(D)$ has two commutative differentials, $d_{mf}$ and $d$, where $d_{mf}$ is the differential of the matrix factorizations associated to the resolutions of $D$ and $d$ is the differential defined in \cite[Definitions 6.3, 6.11 and 6.12]{Wu-color-equi}. Recall that $H_P(D)$ is defined to be $H(H(C_P(D),d_{mf}),d)$. As a graded and filtered $\C$-linear space, we have 
\begin{equation}\label{eq-decomp-chain-complex-MOY}
H(C_P(D),d_{mf}) = \bigoplus_{\Gamma} H_P(\Gamma)\|\mathsf{h}_D(\Gamma)\| \{q^{\rho_D(\Gamma)}\},
\end{equation}
where $\Gamma$ runs through all resolutions of $D$, $\mathsf{h}_D(\Gamma)$ and $\rho_D(\Gamma)$ are shifts of the homological grading and the quantum filtration determined by $D$ and $\Gamma$. $\mathsf{h}_D(\Gamma)$ and $\rho_D(\Gamma)$ are given implicitly in \cite[Definitions 6.3, 6.11 and 6.12]{Wu-color-equi}. Explicit formulas for them can be found in \cite[Subsection 2.3]{Wu-color-MFW}, where the notation for $\mathsf{h}_D(\Gamma)$ is ``$\mathsf{s}_{h,N}(D;\Gamma)$" and the notation for $\rho_D(\Gamma)$ is ``$\mathsf{s}_{q,N}(D;\Gamma)$".

By Theorem \ref{thm-decomp-MOY}, we have
\[
H_P(\Gamma) = \bigoplus_{\varphi \in \mathcal{S} (\Gamma)} Q_\varphi H_P(\Gamma) = \bigoplus_{\varphi\in \mathcal{S}(\Gamma)} \C \cdot v_{\Gamma,\varphi},
\]
where the decomposition does not preserve the quantum filtration. Thus, as a space with a homological grading,  
\[
H(C_P(D),d_{mf}) = \bigoplus_{\Gamma} \bigoplus_{\varphi\in \mathcal{S}(\Gamma)} \C \cdot v_{\Gamma,\varphi} \|\mathsf{h}_D(\Gamma)\| = \bigoplus_{\psi \in \mathcal{QS}(D)} \bigoplus_{(\Gamma,\varphi) \in \mathcal{R}(\psi)} \C \cdot v_{\Gamma,\varphi} \|\mathsf{h}_D(\Gamma)\|.
\] 

\begin{definition}\label{def-subcomplex-state}
For any quasi-state $\psi$ of $D$,
\[
\mathbf{C}(\psi) := \bigoplus_{(\Gamma,\varphi) \in \mathcal{R}(\psi)} \C \cdot v_{\Gamma,\varphi} \|\mathsf{h}_D(\Gamma)\|.
\]
\end{definition}

Recall that the differential $d$ acts only on the small neighborhoods replacing the crossings. (See \cite[Definitions 6.3 6.11 and 6.12]{Wu-color-equi} for more details.) So, by Lemma \ref{lemma-morphism-compatible-states}, for any $(\Gamma_0,\varphi_0) \in \mathcal{R}(\psi)$, we have that $d(v_{\Gamma_0,\varphi_0}) \in \mathbf{C}(\psi)$. Thus, $\mathbf{C}(\psi)$ is a subcomplex of $H(C_P(D),d_{mf})$. More precisely, we have the following lemma.

\begin{lemma}\label{lemma-chain-decomp-states}
There is a decomposition of chain complex 
\[
H(C_P(D),d_{mf}) = \bigoplus_{\psi \in \mathcal{QS}(D)} \mathbf{C}(\psi)
\]
that preserves the homological decomposition. Consequently,
\[
H_P(D) \cong \bigoplus_{\psi \in \mathcal{QS}(D)} H(\mathbf{C}(\psi)).
\]
\end{lemma}

\subsection{A basis for $H_P(D)$} Next, we construct a basis for $H_P(D)$ that is in one-to-one correspondence with the set $\mathcal{S}(D)$ of all states of $D$. Theorem \ref{thm-basis} is a special case of this construction.

\begin{definition}\label{def-homological-grading-unique-state-res}
For a state $\psi$ of $D$, let $(\Gamma,\varphi)$ be its unique resolution. Define $\mathsf{h}(\psi):=\mathsf{h}_D(\Gamma)$. (See Decomposition \eqref{eq-decomp-chain-complex-MOY} above.)
\end{definition}

\begin{lemma}\label{lemma-state-subcomplex-1-dim}
If $\psi$ is a state of $D$, then $H(\mathbf{C}(\psi))$ is $1$-dimensional and is spanned by the homology class $v_{D,\psi}:=[v_{\Gamma,\varphi}]$, where $(\Gamma,\varphi)$ is the unique resolution of $\psi$.  Moreover, $H(\mathbf{C}(\psi))$ has homological grading $\mathsf{h}(\psi)$.
\end{lemma}

\begin{proof}
By Lemma \ref{lemma-unique-state-res}, $\psi$ has a unique resolution $(\Gamma,\varphi)$. So $\mathbf{C}(\psi)= \C \cdot v_{\Gamma,\varphi} \|\mathsf{h}_D(\Gamma)\|$ is $1$-dimensional. But $d$ raises the homological grading by $1$. So $d|_{\mathbf{C}(\psi)}=0$. And the lemma follows.
\end{proof}

\begin{lemma}\cite[Theorem 2]{Gornik}\label{lemma-vanish-quasi-state-1}
Assume that, at all crossings of $D$, the two branches are always both colored by $1$. Then $H(\mathbf{C}(\psi))=0$ if $\psi$ is a quasi-state but not a state of $D$.
\end{lemma}

The following is the main theorem of this section. Note that Theorem \ref{thm-basis} is its special case.

\begin{theorem}\label{thm-basis-knotted-MOY}
Let $D$ be any knotted MOY graph. Then $\{v_{D,\psi}~|~\psi\in\mathcal{S}(D)\}$ is a $\C$-linear basis for $H_P(D)$. For each $\psi\in\mathcal{S}(D)$, $v_{D,\psi}$ has homological grading $\mathsf{h}(\psi)$, which is computed using the sum over all crossings of $D$ given in Definition \ref{def-states-homological-grading}.

Moreover, the decomposition 
\[
H_P(D) = \bigoplus_{\psi\in\mathcal{S}(D)} \C \cdot v_{D,\psi}
\]
is invariant under the Reidemeister moves and, up to shifts of the homological grading, invariant under fork sliding.
\end{theorem}

\begin{proof}
The statement about the formula for $\mathsf{h}(\psi)$ follows directly from the convention of the homological grading given in \cite[Definitions 6.3, 6.11 and 6.12]{Wu-color-equi}. We leave its proof to the reader.

Next, we prove that $\{v_{D,\psi}~|~\psi\in\mathcal{S}(D)\}$ is a $\C$-linear basis for $H_P(D)$. If all crossings of $D$ are between branches colored by $1$, then, by Lemmas \ref{lemma-chain-decomp-states}, \ref{lemma-state-subcomplex-1-dim} and \ref{lemma-vanish-quasi-state-1}, we know that $\{v_{D,\psi}~|~\psi\in\mathcal{S}(D)\}$ is a $\C$-linear basis for $H_P(D)$. (This special case is essentially \cite[Theorem 2]{Gornik}.) 

\begin{figure}[ht]
\[
\xymatrix{
\input{crossing-before-split-+} \ar@{~>}[d] && \text{or} && \input{crossing-before-split--} \ar@{~>}[d]  \\
\input{crossing-split-+} \ar@{~>}[d] && && \input{crossing-split--} \ar@{~>}[d] \\
\input{crossing-split-slide-+} && && \input{crossing-split-slide--} \\
}
\]
\caption{}\label{split-near crossing-fig}

\end{figure}

Now assume that $D$ is a general knotted MOY graph. Near every crossing of $D$, split each branch into edges colored by $1$ and slide these new edges cross each other. (See Figure \ref{split-near crossing-fig}.) This gives us a new knotted MOY graph $\widetilde{D}$ in which all crossings are between branches colored by $1$. For a crossing $c$ of $D$ between two branches colored by $m$ and $n$, set $\tau(c)=m!\cdot n!$ and define $\tau(D)=\prod_{c} \tau(c)$, where $c$ runs through all crossings of $D$. It is easy to check that $|\mathcal{S}(\widetilde{D})|= \tau(D) \cdot |\mathcal{S}(D)|$. From the above special case of the theorem, we know that $\dim_\C H_P(\widetilde{D})= |\mathcal{S}(\widetilde{D})|$. Moreover, by \cite[Theorems 3.12 and 7.1]{Wu-color-equi}, we know that $\dim_\C H_P(\widetilde{D}) = \tau(D) \cdot \dim_\C H_P(D)$. Putting these together, we get $\dim_\C H_P(D) = |\mathcal{S}(D)|$. But, by Lemmas \ref{lemma-chain-decomp-states} and \ref{lemma-state-subcomplex-1-dim}, $\{v_{D,\psi}~|~\psi\in\mathcal{S}(D)\}$ is linearly independent over $\C$. Therefore, it is a basis for $H_P(D)$.

Let $D$ and $D'$ be knotted MOY graphs differ by a single Reidemeister move or fork sliding. Denote by $f:H_P(D) \rightarrow H_P(D')$ the isomorphism induced by this change. For a state $\psi$ of $D$, let $\psi'$ be the state of $D'$ that is identical to $\psi$ outside the part changed by the Reidemeister move or fork sliding. The map $\psi \mapsto \psi'$ is a one-to-one correspondence between $\mathcal{S}(D)$ and $\mathcal{S}(D')$. From the definition of $v_{D,\psi}$ and $v_{D',\psi'}$, it is easy to check by Lemma \ref{lemma-morphism-compatible-states} that $f(v_{D,\psi}) = a\cdot v_{D',\psi'}$ for some $a \in \C$. The fact that $f$ is an isomorphism implies that $a\neq 0$. This shows that the decomposition 
\[
H_P(D) = \bigoplus_{\psi\in\mathcal{S}(D)} \C \cdot v_{D,\psi}
\]
is invariant under the Reidemeister moves and, up to shifts of the homological grading, invariant under fork sliding. (The possible homological grading shift in the case of a fork sliding comes from the grading convention given in \cite[Definition 6.11]{Wu-color-equi}.)
\end{proof}

\begin{remark}
It is straightforward to check that Lemmas \ref{lemma-bouquet-states}, \ref{lemma-edge-splitting-states}, \ref{chi-maps-action-states}, \ref{lemma-states-circle-creation} and \ref{lemma-saddle-states} are also true for knotted MOY graphs.
\end{remark}

\section{Non-degenerate Pairings and Co-pairings}\label{sec-pairings}

We prove in this section Propositions \ref{thm-pairing-inverse} and \ref{thm-pairing-op}. Again, we fix a generic polynomial $P(X)$ of the form \eqref{def-P} and denote by $\Sigma=\Sigma(P)=\{r_1,\dots,r_N\}$ the set of roots of $P'(X)$.

\subsection{Pairings, co-pairings and dualities} To shows that a graded vector space is the graded dual of another graded vector space, one only needs to establish a non-degenerate pairing that preserves the grading. For filtered spaces, one needs a bit more. In this paper, we provide this extra bit in the form of a non-degenerate co-pairing. The goal of this subsection is to prove Lemma \ref{lemma-pairing-copairing-dual}, which will be used in the proofs of Propositions \ref{thm-pairing-inverse} and \ref{thm-pairing-op}.

\begin{lemma}\label{lemma-filtered-dual-dim}
Let $V$ be a finite dimensional ($\zed$-)filtered linear space over $\C$. For a linear function $f:V\rightarrow \C$, $f \in \fil^i (V^\ast)$ if and only if $f|_{\fil^{-i-1}V}=0$. Therefore, $\dim_\C \fil^i (V^\ast) + \dim_\C \fil^{-i-1}V = \dim_\C V$.
\end{lemma}

\begin{proof}
If $f \in \fil^i (V^\ast)$, then $f(\fil^{-i-1}V) \subset \fil^{-1}\C=0$. So $f|_{\fil^{-i-1}V}=0$. 

If $f|_{\fil^{-i-1}V}=0$, then, for any $v\in V$,
\begin{itemize}
	\item $\deg v \leq -i-1$ $\Longrightarrow$ $f(v)=0$ $\Longrightarrow$ $\deg f(v) \leq i+ \deg v$,
	\item $\deg v \geq -i$ $\Longrightarrow$ $f(v)\in \C$ $\Longrightarrow$ $\deg f(v) = 0 \leq i+ \deg v$.
\end{itemize}
Thus, $f \in \fil^i (V^\ast)$.
\end{proof}

\begin{lemma}\label{lemma-pairing-copairing-dual}
Let $U$ and $V$ be two finite dimensional ($\zed$-)filtered linear spaces of the same dimension over $\C$. Assume there are a filtration preserving non-degenerate pairing $\nabla: U\otimes V \rightarrow \C$ and a filtration preserving non-degenerate co-pairing $\Delta:\C \rightarrow U\otimes V$. Then $U \cong V^\ast$ as filtered spaces.
\end{lemma}

\begin{proof}
Define a linear map $F:U \rightarrow V^\ast$ by $F(u)(v)= \nabla(u\otimes v)$. Since $\nabla$ is non-degenerate, $F$ is a $\C$-linear isomorphism. Since $\nabla$ preserves the filtration, we have $F(\fil^i U) \subset \fil^i V^\ast$. In particular, $\dim_\C \fil^i U \leq \dim_\C \fil^i V^\ast$. Similarly, using the filtration preserving non-degenerate pairing $\Delta^\ast: U^\ast \otimes V^\ast  \rightarrow \C$, we get that $\dim_\C \fil^j U^\ast \leq \dim_\C \fil^j V$. When $j=-i-1$, by Lemma \ref{lemma-filtered-dual-dim}, this is $\dim_\C U - \dim_\C \fil^i U \leq \dim_\C V - \dim_\C \fil^i V^\ast$. So $ \dim_\C \fil^i U \geq \dim_\C \fil^i V^\ast$. Thus, $F|_{\fil^i U}: \fil^i U \rightarrow \fil^i V^\ast$ is an isomorphism.
\end{proof}

\begin{figure}[ht]
\[
\xymatrix{
\input{crossing-blue-up-+} \ar@{~>}[r] & \input{crossing-red-down--} \ar@{~>}[rr] && \input{crossing-blue-up-+-red--} \\
\input{crossing-blue-up--} \ar@{~>}[r] & \input{crossing-red-down-+} \ar@{~>}[rr] && \input{crossing-blue-up---red-+}
}
\]
\caption{}\label{L-and-L-bar-fig}

\end{figure}

\subsection{Proposition \ref{thm-pairing-inverse}} We are now ready to prove Proposition \ref{thm-pairing-inverse}. 

\begin{proof}[Proof of Proposition \ref{thm-pairing-inverse}]
Let $L$ be a colored link. We fix a diagram for $L$. Denote by $\overline{L}$ the colored link obtained from $L$ by reversing the orientation of every component of $L$ and switching the upper- and lower-branches at each crossing. Using Reidemeister moves, we put $\overline{L}$ just behind $L$. (See Figure \ref{L-and-L-bar-fig}, where the \textcolor{SkyBlue}{blue} part represents $L$ and the \textcolor{BrickRed}{red} part represents $\overline{L}$.) Then each crossing in $L$ becomes a set of four crossings in this diagram of $L\sqcup\overline{L}$ and each arc in $L$ becomes a pair of arcs of the same color but opposite orientations.

\begin{figure}[ht]
\[
\xymatrix{
 \input{crossing-blue-up-+-red--} \ar@{<~>}[r] & \input{crossing-blue-up-+-red--saddle}
}
\]
\caption{}\label{L-and-L-bar-saddle-fig}

\end{figure}

As in Figure \ref{L-and-L-bar-saddle-fig}, perform a saddle move in the middle of each such pair of arcs. This gives a knotted MOY graph $D$, which is a collection of unlinked unknots. Denote by $\Psi:H_P(L\sqcup\overline{L}) \rightarrow H_P(D)$ the homomorphism induced by these saddle moves. Near the location of each of these saddle moves in $D$, one can perform another saddle move that reverses the effect of the first saddle move. Altogether, these new saddle moves change $D$ back into $L\sqcup\overline{L}$. Denote by $\Phi:H_P(D) \rightarrow H_P(L\sqcup\overline{L})$ the homomorphism induced by these new saddle moves. 

One can use Reidemeister (II) moves to change $D$ into a collection of disjoint circles in the plane and then use circle annihilations to remove all of these circles. This changes $D$ into $\emptyset$ and induces a homomorphism $\mathfrak{e}: H_P(D) \rightarrow \C (\cong H_P(\emptyset))$. Starting from $\emptyset$, one can use the circle creations that reverse the above circle annihilations and then the Reidemeister (II) moves that reverses aforementioned Reidemeister (II) moves to obtain $D$. This induces a homomorphism $\mathfrak{i}: \C \rightarrow H_P(D)$.

Let $c$ be a crossing of $L$. Assume the two branches of $c$ are colored by $m$ and $n$. Define $\rho(c)=m(N-m)+n(N-n)$. Then define $\rho(L)=\sum_c \rho(c)$, where $c$ runs through all crossings of $L$. By the definitions of the homomorphisms induced by saddle moves, circle creations and annihilations in \cite[Section 4]{Wu-color-equi}, we have $\deg_q \Psi \leq \rho(L)$, $\deg_q \Phi \leq \rho(L)$, $\deg_q \mathfrak{e} \leq -\rho(L)$ and $\deg_q \mathfrak{i} \leq -\rho(L)$. Thus, $\deg_q \mathfrak{e} \circ \Psi \leq 0$ and $\deg_q \Phi \circ \mathfrak{i} \leq 0$. Note that $H_P(L\sqcup\overline{L}) \cong H_P(L) \otimes_\C H_P(\overline{L})$. So $\overline{\nabla}:=\mathfrak{e} \circ \Psi$ and $\overline{\Delta}:=\Phi \circ \mathfrak{i}$ are linear pairing and co-pairing of $H_P(L)$ and $H_P(\overline{L})$ that preserve the quantum filtration.

It remains to show that $\overline{\nabla}$ and $\overline{\Delta}$ are non-degenerate. To do so, we use the bases $\{v_{L,\psi}~|~\psi \in \mathcal{S}(L)\}$ and $\{v_{L,\varphi}~|~\varphi \in \mathcal{S}(\overline{L})\}$ of $H_P(L)$ and $H_P(\overline{L})$ from Theorem \ref{thm-basis}.

The change from diagram $L$ to diagram $\overline{L}$ gives a one-to-one correspondence of components of $L$ and $\overline{L}$. For a state $\psi$ of $L$, denote by $\overline{\psi}$ the state of $\overline{L}$ such that $\psi(K)=\overline{\psi}(\overline{K})$ for any component $K$ of $L$ and its corresponding component $\overline{K}$ of $\overline{L}$. Clearly, $\psi \mapsto \overline{\psi}$ is a one-to-one correspondence of states of $L$ and $\overline{L}$. By Lemmas \ref{lemma-states-circle-creation}, \ref{lemma-saddle-states} and the invariance part of Theorem \ref{thm-basis}, we know that:
\begin{enumerate}[(i)]
	\item For any state $\psi$ of $L$ and any state $\varphi$ of $\overline{L}$,
	\[
	\overline{\nabla} (v_{L,\psi} \otimes v_{\overline{L},\varphi}) \begin{cases}
	\neq 0 & \text{if } \varphi = \overline{\psi}, \\
	=0 & \text{if } \varphi \neq \overline{\psi}.
	\end{cases}
	\]
	So $\overline{\nabla}$ is non-degenerate.
	\item \[\overline{\Delta} (1) = \sum_{\psi \in \mathcal{S}(L)} c_\psi v_{L,\psi} \otimes v_{\overline{L},\overline{\psi}},
	\]
	where $c_\psi \in \C\setminus\{0\}$ for all $\psi \in \mathcal{S}(L)$. So $\overline{\Delta}$ is non-degenerate.
\end{enumerate}
Therefore, $H_P(\overline{L}) \cong H_P(L)^\ast$ by Lemma \ref{lemma-pairing-copairing-dual}.
\end{proof}

\subsection{Proposition \ref{thm-pairing-op}} Next, we prove Proposition \ref{thm-pairing-op}. The proof is similar to that of Proposition \ref{thm-pairing-inverse}. We need the following lemma.

\begin{lemma}\label{lemma-l-N-crossings}
\begin{equation}\label{eq-l-N-+}
\hat{C}_P (\setlength{\unitlength}{1pt}
\begin{picture}(40,40)(-20,0)

\put(-20,-20){\vector(1,1){40}}

\put(20,-20){\line(-1,1){15}}

\put(-5,5){\vector(-1,1){15}}

\put(-11,15){\tiny{$_l$}}

\put(8,15){\tiny{$_N$}}

\end{picture}) \cong \hat{C}_P (\setlength{\unitlength}{.5pt}
\begin{picture}(85,45)(-40,45)

\put(-20,0){\vector(0,1){45}}

\put(-20,45){\vector(0,1){45}}

\put(20,0){\vector(0,1){45}}

\put(20,45){\vector(0,1){45}}

\put(-20,45){\vector(1,0){40}}

\put(-35,20){\tiny{$_N$}}

\put(25,20){\tiny{$_{l}$}}

\put(-32,65){\tiny{$_l$}}

\put(25,65){\tiny{$_N$}}

\put(-13,38){\tiny{$_{N-l}$}}

\end{picture})\|l\|\{q^{-l}\}, \vspace{.5cm}
\end{equation}
\begin{equation}\label{eq-l-N--}
\hat{C}_P (\setlength{\unitlength}{1pt}
\begin{picture}(40,40)(-20,0)

\put(20,-20){\vector(-1,1){40}}

\put(-20,-20){\line(1,1){15}}

\put(5,5){\vector(1,1){15}}

\put(-11,15){\tiny{$_l$}}

\put(8,15){\tiny{$_N$}}

\end{picture}) \cong \hat{C}_P (\setlength{\unitlength}{.5pt}
\begin{picture}(85,45)(-40,45)

\put(-20,0){\vector(0,1){45}}

\put(-20,45){\vector(0,1){45}}

\put(20,0){\vector(0,1){45}}

\put(20,45){\vector(0,1){45}}

\put(-20,45){\vector(1,0){40}}

\put(-35,20){\tiny{$_N$}}

\put(25,20){\tiny{$_{l}$}}

\put(-32,65){\tiny{$_l$}}

\put(25,65){\tiny{$_N$}}

\put(-13,38){\tiny{$_{N-l}$}}

\end{picture})\|-l\|\{q^{l}\}, \vspace{.7cm}
\end{equation}
where $\hat{C}_P$ is the unnormalized chain complex. (See \cite[Definition 6.11]{Wu-color-equi}.) Consequently,
\begin{equation}\label{eq-l-N--to+}
\hat{C}_P (\setlength{\unitlength}{1pt}
\begin{picture}(40,40)(-20,0)

\put(-20,-20){\vector(1,1){40}}

\put(20,-20){\line(-1,1){15}}

\put(-5,5){\vector(-1,1){15}}

\put(-11,15){\tiny{$_l$}}

\put(8,15){\tiny{$_N$}}

\end{picture}) \cong \hat{C}_P (\setlength{\unitlength}{1pt}
\begin{picture}(40,40)(-20,0)

\put(20,-20){\vector(-1,1){40}}

\put(-20,-20){\line(1,1){15}}

\put(5,5){\vector(1,1){15}}

\put(-11,15){\tiny{$_l$}}

\put(8,15){\tiny{$_N$}}

\end{picture})\|2l\|\{q^{-2l}\}, \vspace{.5cm}
\end{equation}
\begin{equation}\label{eq-l-N-+to-}
\hat{C}_P (\setlength{\unitlength}{1pt}
\begin{picture}(40,40)(-20,0)

\put(20,-20){\vector(-1,1){40}}

\put(-20,-20){\line(1,1){15}}

\put(5,5){\vector(1,1){15}}

\put(-11,15){\tiny{$_l$}}

\put(8,15){\tiny{$_N$}}

\end{picture}) \cong \hat{C}_P (\setlength{\unitlength}{1pt}
\begin{picture}(40,40)(-20,0)

\put(-20,-20){\vector(1,1){40}}

\put(20,-20){\line(-1,1){15}}

\put(-5,5){\vector(-1,1){15}}

\put(-11,15){\tiny{$_l$}}

\put(8,15){\tiny{$_N$}}

\end{picture})\|-2l\|\{q^{2l}\}. \vspace{.7cm}
\end{equation}
Moreover, at homology level, the above isomorphisms map the base element associated to each state to a multiple of the base element associated to the unique compatible state. 
\end{lemma}

\begin{proof}
By \cite[Lemma 3.8]{Wu-color-equi}, if a MOY graph has an edge with color greater than $N$, then its matrix factorization is homotopic to $0$. Thus, by \cite[Definitions 6.3 and 6.11]{Wu-color-equi}, 
\[
\hat{C}_P (\setlength{\unitlength}{1pt}
\begin{picture}(40,40)(-20,0)

\put(-20,-20){\vector(1,1){40}}

\put(20,-20){\line(-1,1){15}}

\put(-5,5){\vector(-1,1){15}}

\put(-11,15){\tiny{$_l$}}

\put(8,15){\tiny{$_N$}}

\end{picture}) \cong 0\rightarrow C_P (\setlength{\unitlength}{.5pt}
\begin{picture}(85,45)(-40,45)

\put(-20,0){\vector(0,1){45}}

\put(-20,45){\vector(0,1){45}}

\put(20,0){\vector(0,1){45}}

\put(20,45){\vector(0,1){45}}

\put(-20,45){\vector(1,0){40}}

\put(-35,20){\tiny{$_N$}}

\put(25,20){\tiny{$_{l}$}}

\put(-32,65){\tiny{$_l$}}

\put(25,65){\tiny{$_N$}}

\put(-13,38){\tiny{$_{N-l}$}}

\end{picture})\|l\|\{q^{-l}\} \rightarrow 0 \cong \hat{C}_P (\setlength{\unitlength}{.5pt}
\begin{picture}(85,45)(-40,45)

\put(-20,0){\vector(0,1){45}}

\put(-20,45){\vector(0,1){45}}

\put(20,0){\vector(0,1){45}}

\put(20,45){\vector(0,1){45}}

\put(-20,45){\vector(1,0){40}}

\put(-35,20){\tiny{$_N$}}

\put(25,20){\tiny{$_{l}$}}

\put(-32,65){\tiny{$_l$}}

\put(25,65){\tiny{$_N$}}

\put(-13,38){\tiny{$_{N-l}$}}

\end{picture})\|l\|\{q^{-l}\}. \vspace{.7cm}
\]
This proves Isomorphism \eqref{eq-l-N-+}. Isomorphism \eqref{eq-l-N--} follows similarly. Isomorphisms \eqref{eq-l-N--to+} and \eqref{eq-l-N-+to-} are obvious corollaries of \eqref{eq-l-N-+} and \eqref{eq-l-N--}. The conclusion about base elements of the homology follows easily from the definition of the basis and Lemma \ref{lemma-morphism-compatible-states}.
\end{proof}

\begin{figure}[ht]
\[
\xymatrix{
\input{crossing-blue-up-+} \ar@{~>}[r] & \input{crossing-red-up--} \ar@{~>}[rr] && \input{crossing-blue-up-+-red-up-} \\
\input{crossing-blue-up--} \ar@{~>}[r] & \input{crossing-red-up-+} \ar@{~>}[rr] && \input{crossing-blue-up---red-up+}
}
\]
\caption{}\label{L-and-L-op-fig}
\end{figure}

\begin{proof}[Proof of Proposition \ref{thm-pairing-op}]
Let $L$ be a colored link. We fix a diagram for $L$. Denote by $L^{op}$ the colored link obtained from $L$ by changing each color $k$ to $N-k$ and switching the upper- and lower-branches at each crossing. Using Reidemeister moves, we put $L^{op}$ just behind $L$. (See Figure \ref{L-and-L-op-fig}, where the \textcolor{SkyBlue}{blue} part represents $L$ and the \textcolor{BrickRed}{red} part represents $L^{op}$.) Then each crossing in $L$ becomes a set of four crossings in this diagram of $L\sqcup L^{op}$ and each arc of color $k$ in $L$ becomes a pair of parallel arcs in $L\sqcup L^{op}$ in the same direction with colors $k$ and $N-k$. 

\begin{figure}[ht]
\[
\xymatrix{
\input{v-vectors-red-blue} \ar@<1ex>[rr]^{\chi^0} && \input{v-vectors-red-blue-fussed} \ar@<1ex>[ll]^{\chi^1}
}
\]
\caption{}\label{L-and-L-op-arc-fusion-fig}
\end{figure}

In the middle of each such pair of arcs in $L\sqcup L^{op}$, perform the operation in Figure \ref{L-and-L-op-arc-fusion-fig}, which induces two homomorphisms $\chi^0$ and $\chi^1$ of the homology by \cite[Proposition 4.12 and Lemma 4.13]{Wu-color-equi}. All such operations together change $L\sqcup L^{op}$ into a knotted MOY graph $D$. Composing the $\chi$-homomorphisms corresponding to all such operations gives us homomorphisms $\nabla_1:C_P(L\sqcup L^{op})\rightarrow C_P(D)$ and $\Delta_1:C_P(D) \rightarrow C_P(L\sqcup L^{op})$. For a crossing $c$ of $L$ where the colors of the two branches are $m$ and $n$, let $\jmath(c)=m(N-m)+n(N-n)$. Define
\begin{equation}\label{eq-def-jmath}
\jmath(L)=\sum_c \jmath(c),
\end{equation}
where $c$ runs through all crossings of $L$. By \cite[Proposition 4.12 and Lemma 4.13]{Wu-color-equi}, one can easily check that $\nabla_1$ and $\Delta_1$ preserve the homological grading and their quantum degrees satisfy $\deg_q \nabla_1 \leq \jmath(L)$ and $\deg_q \Delta_1 \leq \jmath(L)$.

\begin{figure}[ht]
\[
\xymatrix{
\input{crossing-blue-up-+} \ar@{~>}[r] & \input{crossing-blue-up-+-red-up-} \ar@{~>}[rr] && \input{crossing-blue-up-+-red-up-fused} \\
\input{crossing-blue-up--} \ar@{~>}[r] & \input{crossing-blue-up---red-up+} \ar@{~>}[rr] && \input{crossing-blue-up---red-up+fused}
}\]
\caption{}\label{L-and-L-op-crossing-fusion-fig}
\end{figure}

Figure \ref{L-and-L-op-crossing-fusion-fig} shows what $D$ looks like near a crossing of $L$. If $n=m$, by \cite[Definition 6.11]{Wu-color-equi}, we have
\[
C_P(\input{crossing-blue-up---red-up+fused-m-m}) \cong \hat{C}_P(\input{crossing-blue-up---red-up+fused-m-m})\|2m-N\|\{q^{N-2m}\}. \vspace{1cm}
\]
By \cite[Theorem 7.1]{Wu-color-equi}, we have that
\[
\hat{C}_P(\input{crossing-blue-up---red-up+fused-m-m}) \simeq \hat{C}_P(\input{crossing-blue-up---red-up+fused-m-m-1}). \vspace{1cm}
\]
By Isomorphism \eqref{eq-l-N-+to-} in Lemma \ref{lemma-l-N-crossings}, we have
\[
\hat{C}_P(\input{crossing-blue-up---red-up+fused-m-m-1}) \cong \hat{C}_P(\input{crossing-blue-up---red-up+fused-m-m-2})\|-2m\|\{q^{2m}\}. \vspace{1cm}
\]
Applying \cite[Theorem 7.1]{Wu-color-equi} again, we get
\[
\hat{C}_P(\input{crossing-blue-up---red-up+fused-m-m-2}) \simeq\hat{C}_P(\input{crossing-blue-up---red-up+fused-m-m-3}).        \vspace{1cm}
\]
Next, applying Isomorphism \eqref{eq-l-N-+} in Lemma \ref{lemma-l-N-crossings}, we get
\[
\hat{C}_P(\input{crossing-blue-up---red-up+fused-m-m-3}) \cong  C_P(\input{crossing-blue-up---red-up+fused-m-m-4})\|N\|\{q^{-N}\}   \vspace{1cm}
\]
Putting all these together, we get
\begin{equation}\label{eq-neg-res-case1}
C_P(\input{crossing-blue-up---red-up+fused-m-m}) \simeq C_P(\input{crossing-blue-up---red-up+fused-m-m-4}). \vspace{1cm}
\end{equation}
Similarly, one can check that, if $n=N-m$, then
\begin{equation}\label{eq-neg-res-case2}
C_P(\input{crossing-blue-up---red-up+fused-m-N-m}) \simeq C_P(\input{crossing-blue-up---red-up+fused-m-N-m-4}), \vspace{1cm}
\end{equation}
and if $n\neq m$ or $N-m$, then
\begin{equation}\label{eq-neg-res-case3}
C_P(\input{crossing-blue-up---red-up+fused-m-n}) \simeq C_P(\input{crossing-blue-up---red-up+fused-m-n-4})\|N-2m\|\{q^{2m-N}\} . \vspace{1cm}
\end{equation}
Using the shifting $\mathsf{s}$ defined in Lemma \ref{lemma-grading-shift-pairing-op}, we get that, in all three cases,
\begin{equation}\label{eq-neg-res-general}
C_P(\input{crossing-blue-up---red-up+fused-m-n}) \simeq C_P(\input{crossing-blue-up---red-up+fused-m-n-4})\|-i\|\{q^{i}\} . \vspace{1cm}
\end{equation}
where $i=\mathsf{s}(\setlength{\unitlength}{1pt}
\begin{picture}(40,20)(-20,0)

\put(20,-20){\textcolor{SkyBlue}{\vector(-1,1){40}}}

\put(-20,-20){\textcolor{SkyBlue}{\line(1,1){15}}}

\put(5,5){\textcolor{SkyBlue}{\vector(1,1){15}}}

\put(-11,15){\textcolor{SkyBlue}{\tiny{$_m$}}}

\put(9,15){\textcolor{SkyBlue}{\tiny{$_n$}}}

\end{picture})$. \vspace{.7cm} Similarly, 
\begin{equation}\label{eq-pos-res-general}
C_P(\input{crossing-blue-up-+-red-up-fused-m-n}) \simeq C_P(\input{crossing-blue-up---red-up+fused-m-n-4})\|-j\|\{q^{j}\} , \vspace{1cm}
\end{equation}
where $j=\mathsf{s}(\setlength{\unitlength}{1pt}
\begin{picture}(40,20)(-20,0)

\put(-20,-20){\textcolor{SkyBlue}{\vector(1,1){40}}}

\put(20,-20){\textcolor{SkyBlue}{\line(-1,1){15}}}

\put(-5,5){\textcolor{SkyBlue}{\vector(-1,1){15}}}

\put(-11,15){\textcolor{SkyBlue}{\tiny{$_m$}}}

\put(9,15){\textcolor{SkyBlue}{\tiny{$_n$}}}

\end{picture})$. \vspace{.7cm}

Recall that each crossing in $L$ leads to a local configuration in $D$ of the form on the left hand side of \eqref{eq-neg-res-general} or \eqref{eq-pos-res-general}. (See Figure \ref{L-and-L-op-crossing-fusion-fig}.) One can change $D$ into an embedded MOY graph $\Gamma$ (that is, no crossings) by replacing each of these with the corresponding local configuration on the right hand side of \eqref{eq-neg-res-general} or \eqref{eq-pos-res-general}. The homotopy equivalence \eqref{eq-neg-res-general} and \eqref{eq-pos-res-general} imply that 
\begin{equation}\label{eq-D-to-Gamma}
C_P(D) \simeq C_P(\Gamma)\|-\mathsf{s}(L)\|\{q^{\mathsf{s}(L)}\}.
\end{equation}
The homotopy equivalence \eqref{eq-D-to-Gamma} induces isomorphisms $\nabla_2:H_P(D)\rightarrow H_P(\Gamma)$ and $\Delta_2:H_P(\Gamma) \rightarrow H_P(D)$, where $\nabla_2$ has homological grading $\mathsf{s}(L)$ and quantum degree $-\mathsf{s}(L)$, and $\Delta_2$ has homological grading $-\mathsf{s}(L)$ and quantum degree $\mathsf{s}(L)$. Moreover, the change from $D$ to $\Gamma$ comes with an obvious one-to-one correspondence between compatible states of $D$ and $\Gamma$. By Lemmas \ref{lemma-l-N-crossings} and \ref{lemma-morphism-compatible-states}, $\nabla_2$ and $\Delta_2$ map the base element associated to each state to a multiple of the base element associated to the unique compatible state.

Now apply edge merging to each \textcolor{BrickRed}{red}-\textcolor{SkyBlue}{blue} ``bubble" in $\Gamma$. This changes $\Gamma$ into a collection of disjoint circles of color $N$ in the plane. Then apply circle annihilation to each of these circles. This gives the empty MOY graph $\emptyset$. Together, these changes induce a homomorphism $\nabla_3:H_P(\Gamma) \rightarrow \C (= H_P(\emptyset)$.) The circle creations and edge splittings that reverse the above changes induce a homomorphism $\Delta_3:\C \rightarrow H_P(\Gamma)$. From the definitions of the homomorphisms induced by these local changes, one can see that $\nabla_3$ and $\Delta_3$ preserve the homological grading and have quantum degree $\leq -\jmath(L)$, where $\jmath(L)$ is defined by Equation \eqref{eq-def-jmath}.

Finally, define
\begin{eqnarray*}
\nabla^{op} & = & \nabla_3 \circ \nabla_2 \circ \nabla_1: H_P(L)\otimes H_P(L^{op}) (\cong H_P(L\sqcup L^{op})) \rightarrow \C, \\
\Delta^{op} & = & \Delta_1 \circ \Delta_2 \circ \Delta_3: \C \rightarrow  H_P(L)\otimes H_P(L^{op}).
\end{eqnarray*}
Then $\nabla^{op}$ is a pairing of homological grading $\mathsf{s}(L)$ and quantum degree $\leq -\mathsf{s}(L)$ and $\Delta^{op}$ is a co-pairing of homological grading $-\mathsf{s}(L)$ and quantum degree $\leq \mathsf{s}(L)$ For every state $\psi$ of $L$, define the state $\psi^{op}$ of $L^{op}$ by $\psi^{op}(K^{op})=\Sigma \setminus \psi(K)$ for every component $K$ of $L$, where $K^{op}$ is the component of $L^{op}$ that corresponds to $K$. The map $\psi\mapsto\psi^{op}$ is a one-to-one correspondence between the states of $L$ and $L^{op}$. By Lemmas \ref{lemma-edge-splitting-states}, \ref{chi-maps-action-states}, \ref{lemma-states-circle-creation} and the fact that $\nabla_2$ and $\Delta_2$ preserve states, one can easily check that
\begin{enumerate}[(i)]
	\item For any states $\psi$ of $L$ and $\varphi$ of $L^{op}$,
	\[
	\nabla^{op} (v_{L,\psi}\otimes v_{L^{op},\varphi}) \begin{cases}
	\neq 0 & \text{if } \varphi = \psi^{op}, \\
	=0 & \text{if } \varphi \neq \psi^{op}.
	\end{cases}
	\] 
	So $\nabla^{op}$ is non-degenerate.
	\item $\Delta^{op}(1) = \sum_{\psi \in \mathcal{S}(L)} c_\psi \cdot v_{L,\psi} \otimes v_{L^{op},\psi^{op}}$, where $c_\psi \in \C\setminus\{0\}$ for every $\psi \in \mathcal{S}(L)$. So $\Delta^{op}$ is non-degenerate.
\end{enumerate}
Thus, by Lemma \ref{lemma-pairing-copairing-dual}, we have $H_P(L^{op}) \cong \Hom_\C(H_P(L),\C) \|-\mathsf{s}(L)\| \{q^{\mathsf{s}(L)}\}$.
\end{proof}

\section{Bounds for the Slice Genus and the Self Linking Number}\label{sec-bounds}

We prove Theorems \ref{thm-ras-genus} and \ref{thm-ras-bennequin} in this section. The arguments we use are more or less straightforward generalizations of those in \cite{Ras1,Shumakovitch}. Again, we fix a generic polynomial $P(X)$ of the form \eqref{def-P} and denote by $\Sigma=\Sigma(P)=\{r_1,\dots,r_N\}$ the set of roots of $P'(X)$.

\subsection{Colored $\mathfrak{sl}(N)$-Rasmussen invariants for links} For convenience, we introduce the slightly more general ``colored $\mathfrak{sl}(N)$-Rasmussen invariants for links", which generalize the Rasmussen invariant for links defined in \cite[Subsection 7.1]{Beliakova-Wehrli}.

\begin{definition}\label{color-ras-links}
Let $L$ be an uncolored link. Denote by $L^{(m)}$ be the colored link obtained by coloring every component of $L$ by $m$. We call a state $\psi$ of $L^{(m)}$ a constant state if $\psi(K_1)=\psi(K_2)$ for any two components $K_1$ and $K_2$ of $L$. The subspace $\mathsf{K}(L^{(m)})$ of $H_P(L^{(m)})$ spanned by base elements associated to constant states inherits the quantum filtration $\fil$ of $H_P(L^{(m)})$.

Define
\[
s_P^{(m)}(L) = \frac{1}{2} (\max\deg_q \mathsf{K}(L^{(m)}) + \min\deg_q\mathsf{K}(L^{(m)})),
\]
where 
\begin{eqnarray*}
\max\deg_q \mathsf{K}(L^{(m)}) & = & \min\{j|\fil^j \mathsf{K}(L^{(m)}) = \mathsf{K}(L^{(m)})\}, \\
\min\deg_q \mathsf{K}(L^{(m)}) & = & \max\{j|\fil^{j-1} \mathsf{K}(L^{(m)}) = 0\}.
\end{eqnarray*}
\end{definition}

Clearly, if $L$ is a knot, then this definition coincides with Definition \ref{def-ras}.

\subsection{Smooth link cobordism and the slice genus} We prove in this subsection Theorem \ref{thm-ras-genus}. To do that, let us first recall some basic facts about smooth link cobordisms.

A smooth cobordism from link $L_0$ to link $L_1$ is a properly smoothly embedded compact oriented surface $S$ in $S^3\times[0,1]$ such that $\partial S = (L_0\times\{0\}) \cup (-L_1\times\{1\})$, where $-L_1$ is $L_1$ with the opposite orientation. A smooth link cobordism is said to be elementary if it is either a Reidemeister move or a Morse move (that is, a saddle move, a circle creation or annihilation.) Any smooth link cobordism admits a movie presentation, that is, a decomposition into elementary cobordisms. (See \cite{CS1,CS2} for more details.)

\begin{definition}\label{def-compatibility-states-cobordism}
Let $S$ be a smooth cobordism from link $L_0$ to link $L_1$. States $\psi_0$ of $L_0^{(m)}$ and $\psi_1$ of $L_1^{(m)}$ are said to be compatible via $S$ if, for every component $C$ of $S$, all boundary components of $C$ are assigned the same subset of $\Sigma$ by $\psi_0$ and $\psi_1$.
\end{definition}

For a smooth cobordism $S$ from link $L_0$ to link $L_1$, fix a movie presentation $(S_1,S_2,\cdots,S_l)$ of $S$. Let $D_{j-1}$ and $D_j$ be the initial and terminal ends of $S_j$. In particular, $D_0=L_0$ and $D_l=L_1$. If $S_j$ is a Reidemeister move, then define $F_{S_j}:H_P(D_{j-1}^{(m)})\rightarrow H_P(D_j^{(m)})$ to be the isomorphism given by \cite[Theorem 8.1]{Wu-color-equi}. If $S_j$ is a Morse move, then define $F_{S_j}:H_P(D_{j-1}^{(m)})\rightarrow H_P(D_j^{(m)})$ to be the homomorphism induced by the matrix factorization morphism associated to it. (See \cite[Section 4]{Wu-color-equi}.) Define $F=F_{S_l}\circ \cdots \circ F_{S_1}:H_P(L_0^{(m)})\rightarrow H_P(L_1^{(m)})$. The following is a straightforward generalization of \cite[Proposition 4.1]{Ras1}.

\begin{lemma}\label{lemma-cobordism-map-states}
Suppose that $S$ has no closed components. Then, for any state $\psi$ of $L_0^{(m)}$,
\[
F(v_{L_0^{(m)},\psi}) = \sum_{\varphi} c_\varphi \cdot v_{L_1^{(m)},\varphi},
\]
where $\varphi$ runs through all states of $L_1^{(m)}$ compatible with $\psi$ via $S$ and $c_\varphi \in \C\setminus\{0\}$ for each such $\varphi$.
\end{lemma}

\begin{proof}
Define $\widehat{S}_j=S_1\cup\cdots\cup S_j$ and $\widehat{F}_j=F_{S_j}\circ \cdots \circ F_{S_1}:H_P(L_0^{(m)})\rightarrow H_P(D_j^{(m)})$. Note that $S=\widehat{S}_l$ and $F= \widehat{F}_l$. We prove by induction that 
\begin{equation}\label{eq-cobordism-map-states-induction}
\widehat{F}_j(v_{L_0^{(m)},\psi}) = \sum_{\varphi} c_\varphi \cdot v_{D_j^{(m)},\varphi},
\end{equation}
where $\varphi$ runs through all states of $D_j^{(m)}$ compatible with $\psi$ via $\widehat{S}_j$ and $c_\varphi \in \C\setminus\{0\}$ for each such $\varphi$.

Define $\widehat{F}_0=\id: H_P(L_0^{(m)})\rightarrow H_P(L_0^{(m)})$. When $j=0$, Equation \eqref{eq-cobordism-map-states-induction} is trivially true. Assume Equation \eqref{eq-cobordism-map-states-induction} is true for $j-1$. Consider $\widehat{F}_j = F_{S_j}\circ \widehat{F}_{j-1}$. If $S_j$ is a Reidemeister move, then by Theorem \ref{thm-basis}, Equation \eqref{eq-cobordism-map-states-induction} is true for $j$. If $S_j$ is a Morse move (a saddle move, a circle creation or annihilation,) then by Lemmas \ref{lemma-states-circle-creation} and \ref{lemma-saddle-states}, Equation \eqref{eq-cobordism-map-states-induction} is true for $j$. This completes the induction. So Equation \eqref{eq-cobordism-map-states-induction} is true for $j=0,1,\dots,l$. And the lemma follows.
\end{proof}

We call a component $C$ of $S$ semi-closed if $\partial C \subset L_0\times\{0\}$ or $\partial C \subset -L_1\times\{1\}$. The following corollary generalizes \cite[inequality (7.1)]{Beliakova-Wehrli}.

\begin{corollary}\label{coro-cobordism-map-constant-states}
Suppose that $S$ has neither closed components nor semi-closed components. Then $F|_{\mathsf{K}(L_0^{(m)})}:\mathsf{K}(L_0^{(m)})\rightarrow \mathsf{K}(L_1^{(m)})$ is an isomorphism, where $\mathsf{K}(L_i^{(m)})$ is the subspace of $H_P(L_i^{(m)})$ spanned by base elements associated to constant states. Consequently, 
\begin{equation}\label{eq-ras-deg-K-cobordism}
|s_P^{(m)}(L_1)-s_P^{(m)}(L_0)|\leq -m(N-m)\chi(S),
\end{equation}
where $\chi(S)$ is the Euler characteristic of $S$.
\end{corollary}

\begin{proof}
Let $\psi$ be a constant state of $L_0^{(m)}$. Since $S$ has no semi-closed components, a state $\varphi$ of $L_1^{(m)}$ is compatible via $S$ to $\psi$ if and only if it assigns to every component of $L_1$ the same constant $m$-element subset of $\Sigma$ that $\psi$ assigns to components of $L_0$. Since $S$ has no closed components, by Lemma \ref{lemma-cobordism-map-states}, $F(v_{L_0^{(m)},\psi}) = c_\psi \cdot v_{L_1^{(m)},\varphi}$ where $\varphi$ is the unique constant state of $L_1^{(m)}$ compatible to $\psi$ via $S$, and $c_\psi \in \C\setminus \{0\}$. This shows that $F(\mathsf{K}(L_0^{(m)})) \subset \mathsf{K}(L_1^{(m)})$ and $F|_{\mathsf{K}(L_0^{(m)})}:\mathsf{K}(L_0^{(m)})\rightarrow \mathsf{K}(L_1^{(m)})$ is an isomorphism.

From the definition of $F_{S_j}$ (see also \cite[Section 4 and Theorem 8.1]{Wu-color-equi},) we know that $\deg_q F_{S_j} \leq -m(N-m)\chi(S_j)$. So 
\[
\deg F \leq \sum_{j=1}^l \deg_q F_{S_j} \leq -\sum_{j=1}^l m(N-m)\chi(S_j) = -m(N-m)\chi(S).
\]

Assume the element $u\in \mathsf{K}(L_1^{(m)})$ attains the maximal quantum degree of $\mathsf{K}(L_1^{(m)})$. Since $F|_{\mathsf{K}(L_0^{(m)})}:\mathsf{K}(L_0^{(m)})\rightarrow \mathsf{K}(L_1^{(m)})$ is surjective, there is a $v\in \mathsf{K}(L_0^{(m)})$ such that $F(v)=u$. So
\[
\max\deg_q \mathsf{K}(L_1^{(m)}) = \deg_q u \leq \deg_q F + \deg_q v \leq -m(N-m)\chi(S) + \max\deg_q \mathsf{K}(L_0^{(m)}).
\]
Therefore,
\begin{equation}\label{eq-max-deg-K-cobordism}
\max\deg_q \mathsf{K}(L_1^{(m)}) - \max\deg_q \mathsf{K}(L_0^{(m)}) \leq -m(N-m)\chi(S).
\end{equation}
Assume $v'\in \mathsf{K}(L_0^{(m)})$ attains the minimal quantum degree of $\mathsf{K}(L_0^{(m)})$. Since $F|_{\mathsf{K}(L_0^{(m)})}:\mathsf{K}(L_0^{(m)})\rightarrow \mathsf{K}(L_1^{(m)})$ is injective, $F(v')\neq 0$ and hence
\[
\min\deg_q \mathsf{K}(L_1^{(m)}) \leq \deg_q F(v') \leq \deg_q F + \deg_q v' \leq -m(N-m)\chi(S) + \min\deg_q \mathsf{K}(L_0^{(m)}).
\]
So
\begin{equation}\label{eq-min-deg-K-cobordism}
\min\deg_q \mathsf{K}(L_1^{(m)}) - \min\deg_q \mathsf{K}(L_0^{(m)}) \leq -m(N-m)\chi(S).
\end{equation}
Inequalities \eqref{eq-max-deg-K-cobordism} and \eqref{eq-min-deg-K-cobordism} give us that
\begin{equation}\label{eq-ras-deg-K-cobordism-1}
s_P^{(m)}(L_1)-s_P^{(m)}(L_0) \leq -m(N-m)\chi(S).
\end{equation}
Looking at $S$ ``up-side-down", we can view $S$ as a smooth cobordism from $L_1$ to $L_0$. Applying the same argument, we get
\begin{equation}\label{eq-ras-deg-K-cobordism-2}
s_P^{(m)}(L_0)-s_P^{(m)}(L_1) \leq -m(N-m)\chi(S).
\end{equation}
Inequalities \eqref{eq-ras-deg-K-cobordism-1} and \eqref{eq-ras-deg-K-cobordism-2} imply Inequality \eqref{eq-ras-deg-K-cobordism}.
\end{proof}

\begin{lemma}\label{lemma-res-unknot}
Denote by $\bigcirc$ the unknot. Then $s_P^{(m)}(\bigcirc)=0$.
\end{lemma}

\begin{proof}
This follows easily from \cite[Corollary 6.1]{Wu-color} and \cite[Proposition 9.8]{Wu-color-equi}.
\end{proof}

Theorem \ref{thm-ras-genus} is now just a special case of Corollary \ref{coro-cobordism-map-constant-states}.

\begin{proof}[Proof of Theorem \ref{thm-ras-genus}]
Let $\overline{S}$ be any connected compact smooth surface in $D^4$ bounded by the knot $K \subset S^3=\partial D^4$. Removing a small disc from the interior of $\overline{S}$, we get a connected smooth cobordism $S$ from $K$ to the unknot $\bigcirc$. Then by Corollary \ref{coro-cobordism-map-constant-states} and Lemma \ref{lemma-res-unknot}, 
\begin{eqnarray*}
|s_P^{(m)}(K)| & = & |s_P^{(m)}(K)-s_P^{(m)}(\bigcirc)| \\
& \leq & -m(N-m)\chi(S) = -m(N-m)(\chi(\overline{S})-1) = 2m(N-m)g(\overline{S}),
\end{eqnarray*}
where $g(\overline{S})$ is the genus of $\overline{S}$. And Theorem \ref{thm-ras-genus} follows.
\end{proof}

\subsection{The self linking number} We prove in this subsection Theorem \ref{thm-ras-bennequin}. In fact, with colored $\mathfrak{sl}(N)$-Rasmussen invariants for links, we give a version of Theorem \ref{thm-ras-bennequin} for all links instead of just for knots.

\begin{lemma}\label{lemma-ras-unlink-estimation}
Denote by $\bigcirc^{\sqcup b}$ the unlink with $b$ components. Then 
\[
|s_P^{(m)}(\bigcirc^{\sqcup b})| \leq m(N-m)(b-1).
\] 
\end{lemma}

\begin{proof}
There is an obvious smooth cobordism $S$ from the unknot $\bigcirc$ to $\bigcirc^{\sqcup b}$ that is homeomorphic to a disc with $b$ punctures. Apply Corollary \ref{coro-cobordism-map-constant-states} to this cobordism, we get $|s_P^{(m)}(\bigcirc^{\sqcup b})| \leq -m(N-m)\chi(S) = m(N-m)(b-1)$.
\end{proof}

\begin{lemma}\label{lemma-ras-neg-braids}
Let $B$ be a closed negative braid with $b$ strands and $l$ crossings. Then 
\[
s_P^{(m)}(B)=m(N-m)(b-l-1).
\]
\end{lemma}

\begin{proof}
Recall that $H_P(B^{(m)})$ is defined to be the homology of a chain complex of the form
\[
\bigoplus_{\Gamma} H_P(\Gamma)\|\mathsf{h}_{B^{(m)}}(\Gamma)\|\{q^{\rho_{B^{(m)}}(\Gamma)}\},
\]
where $\Gamma$ runs through all MOY resolutions of $B^{(m)}$, $\mathsf{h}_{B^{(m)}}(\Gamma)$ and $\rho_{B^{(m)}}(\Gamma)$ are the shifts of the homological grading and the quantum filtration. (See Subsection \ref{subsec-decom-chain-quasi-state} for more about $\mathsf{h}_{B^{(m)}}(\Gamma)$ and $\rho_{B^{(m)}}(\Gamma)$.) Since $B$ has only negative crossings, it is easy to check that the minimum of $\mathsf{h}_{B^{(m)}}(\Gamma)$ is $0$ and is only attained by the MOY resolution $(\bigcirc^{\sqcup b})^{(m)}$. One can check that $\rho_{B^{(m)}}((\bigcirc^{\sqcup b})^{(m)})=-lm(N-m)$. So the chain complex used to define $H_P(B^{(m)})$ is of the form
\[
0\rightarrow H_P((\bigcirc^{\sqcup b})^{(m)})\{q^{-lm(N-m)}\} \rightarrow \cdots
\]
Recall that the base vector associated to each state of $B^{(m)}$ is realized by unique a state of a unique MOY resolution of $B^{(m)}$. (See Lemmas \ref{lemma-unique-state-res} and \ref{lemma-state-subcomplex-1-dim}.) It is easy to check that, for any constant state $\psi$ of $B^{(m)}$, its resolution is $((\bigcirc^{\sqcup b})^{(m)}, \varphi)$, where $\varphi$ is the constant state of $(\bigcirc^{\sqcup b})^{(m)}$ assigning the same subset of $\Sigma$ as $\psi$. According to Lemma \ref{lemma-state-subcomplex-1-dim}, $v_{B^{(m)},\psi}=[v_{(\bigcirc^{\sqcup b})^{(m)}, \varphi}]$. Also note that $H_P((\bigcirc^{\sqcup b})^{(m)})\{q^{-lm(N-m)}\}$ contains no non-zero boundary elements. So this implies that
\[
\mathsf{K}(B^{(m)}) \cong \mathsf{K}((\bigcirc^{\sqcup b})^{(m)})\{q^{-lm(N-m)}\},
\]
where the isomorphism preserve the quantum filtration. Thus,
\begin{equation}\label{eq-ras-neg-braid-to-unlink} 
s_P^{(m)}(B) = s_P^{(m)}(\bigcirc^{\sqcup b}) - lm(N-m).
\end{equation}

We first prove the lemma in the special case when $B$ is a knot.

Assume $B$ is a knot. By Theorem \ref{thm-ras-genus}, we have 
\begin{equation}\label{eq-ras-neg-braid-knot}
s_P^{(m)}(B) \geq -2m(N-m)g_\ast(B).
\end{equation}
Since $B$ has only negative crossings, we know that
\begin{equation}\label{eq-genus-neg-braid=knot}
2g_\ast(B) =l+1-b.
\end{equation}
(See for example \cite[Subsection 5.2]{Ras1} for a proof of Equation \eqref{eq-genus-neg-braid=knot}.) Now using \eqref{eq-ras-neg-braid-to-unlink}, \eqref{eq-ras-neg-braid-knot}, \eqref{eq-genus-neg-braid=knot} and Lemma \ref{lemma-ras-unlink-estimation}, we get
\begin{eqnarray*}
-2m(N-m)g_\ast(B) &\leq & s_P^{(m)}(B) = s_P^{(m)}(\bigcirc^{\sqcup b}) - lm(N-m) \\
& \leq & m(N-m)(b-l-1) = -2m(N-m)g_\ast(B).
\end{eqnarray*}
So the above inequality must be an equation. Therefore, we get 
\[
s_P^{(m)}(B) =  m(N-m)(b-l-1),
\]
which proves the lemma in the special case when $B$ is a knot.

As a by-product of the above argument, we get that 
\begin{equation}\label{eq-ras-unlink}
s_P^{(m)}(\bigcirc^{\sqcup b}) = m(N-m)(b-1).
\end{equation}
It is clear that, for a general closed negative braid $B$, the lemma follows from Equations \eqref{eq-ras-neg-braid-to-unlink} and \eqref{eq-ras-unlink}.
\end{proof}

\begin{theorem}\label{thm-ras-bennequin-links}
For a closed braid $B$ with writhe $w$ and $b$ strands, 
\[
s_P^{(m)}(B) \leq  m(N-m) (w+b-1).
\]
\end{theorem}

\begin{figure}[ht]
\[
\xymatrix{
\input{crossing-m-m-+braid} \ar@{~>}[rr] && \input{crossing-m-m-+braid-saddle} \ar@{~>}[rr] && \input{crossing-m-m-+braid-saddle-final}
}\]
\caption{}\label{braid-remove-positive-crossings-fig}
\end{figure}

\begin{proof}
Suppose that $B$ has $l_+$ positive crossings and $l_-$ negative crossings. Then $w=l_+-l_-$. At each positive crossing of $B$, perform the saddle move and the Reidemeister I move in Figure \ref{braid-remove-positive-crossings-fig}. This give a braid $B_-$ of $b$ strands with $l_-$ negative crossings and no positive crossings. Moreover, this also gives a smooth link cobordism $S$ from $B$ to $B_-$ with Euler characteristic $\chi(S)=-l_+$ that has no closed or semi-closed components.

By Lemma \ref{lemma-ras-neg-braids} and Corollary \ref{coro-cobordism-map-constant-states}, we have
\begin{eqnarray*}
s_P^{(m)}(B_-) & = & m(N-m)(b-l_- -1), \\
|s_P^{(m)}(B) - s_P^{(m)}(B_-)| & \leq & m(N-m)l_+,
\end{eqnarray*}
which imply that
\[
s_P^{(m)}(B) \leq m(N-m)(b-l_- -1) + m(N-m)l_+ = m(N-m) (w+b-1).
\]
\end{proof}

Theorem \ref{thm-ras-bennequin} is a special case of Theorem \ref{thm-ras-bennequin-links}.

\begin{proof}[Proof of Theorem \ref{thm-ras-bennequin}]
Let $B$ be a closed braid representation of $K$ such that $SL(B)=\overline{SL}(K)$. Assume $B$ has $b$ strands and writhe $w$. Denote by $B_{mir}$ the mirror image of $B$, which has $b$ strands and writhe $-w$. By Theorem \ref{thm-ras-bennequin-links} and Corollary \ref{coro-Ras-equations}, 
\begin{eqnarray*}
s_P^{(m)}(B_{mir}) & \leq & m(N-m)(b-w-1), \\
s_P^{(m)}(B_{mir}) & = & -s_P^{(m)}(B) = -s_P^{(m)}(K).
\end{eqnarray*}
Thus, 
\begin{eqnarray*}
s_P^{(m)}(K) & \geq & m(N-m)(w-b+1)  = m(N-m)(SL(B)+1) \\
& = & m(N-m)(\overline{SL}(K)+1).
\end{eqnarray*}
\end{proof}

\end{document}